\documentclass[11pt,reqno]{amsart}
\usepackage[margin=0.92in]{geometry}
\usepackage{amsmath,amssymb,amsthm,graphicx,amsxtra, setspace}
\usepackage{mathabx}
\usepackage[colorlinks,backref]{hyperref}
\usepackage{esint}
\usepackage{hyperref}
\usepackage[dvipsnames]{xcolor}
\usepackage{dsfont}
\usepackage[utf8]{inputenc}
\usepackage{mathrsfs}
\usepackage{upgreek}
\usepackage{mathtools}
\usepackage[makeroom]{cancel}
\allowdisplaybreaks

\usepackage[cyr]{aeguill}

\colorlet{darkblue}{blue!50!black}

\hypersetup{
	colorlinks,%
	citecolor=blue,%
	filecolor=red,%
	linkcolor=darkblue,%
	urlcolor=blue,%
	pdfnewwindow=true,%
	pdfstartview={FitH}
}

\colorlet{darkblue}{red!100!black}

\newtheorem{theorem}{Theorem}[section]

\newtheorem{lemma}[theorem]{Lemma}

\newtheorem{definition}[theorem]{Definition}

\newtheorem{remark}[theorem]{Remark}

\let\originalleft\left
\let\originalright\right
\renewcommand{\left}{\mathopen{}\mathclose\bgroup\originalleft}
\renewcommand{\right}{\aftergroup\egroup\originalright}

\def\e{\boldsymbol{e}}

\def\T{T\wedge\tau_N}

\def\T3{\mathbb{T}^3}

\def\diver{\mathrm{div}}

\def\curl{\mathrm{curl}}

\def\J{\mathrm{J}}

\def\y{\boldsymbol{y}}

\def\Y{\mathrm{Y}}

\def\X{\mathbb{X}}
\def\cR{\mathcal R}

\def\g{\mathbf{g}}

\def\h{\boldsymbol{h}}
\def\z{\boldsymbol{z} }
\def\v{\boldsymbol{v}}
\def\w{\boldsymbol{w}}
\def\n{\boldsymbol{n}}

\def\pfrak{\mathfrak{p}}

\def\Wb{\mathbb{W}}

\def\O{\Omega}

\def\Vb{\mathbb{V}}

\def\u{\boldsymbol{u}}
\def\un{\boldsymbol{u}_n}
\def\gn{\mathbf{g}_n}

\def\Uad{\mathcal{U}_{ad}}

\def\Ucal{\mathcal{U}}
\def\Gbf{\mathbf{G}}
\def\Hbf{\mathbf{H}}
\def\Kbf{\mathbf{K}}
\def\Wb{\mathbb{W}}
\def\Db{\mathbb{D}}
\def\wiWb{\widetilde{\mathbb{W}}}

\newcommand{\eps}{\varepsilon}

\newcommand{\R}{\mathbb{R}}

\begin{document}
\title[Boundary control for optimal mixing]{Boundary control for optimal mixing by two-dimensional second-grade fluids}
 

\thanks{$^*$Corresponding author\\$^1$Chair for Dynamics, Control, Machine Learning and Numerics, Department of Mathematics, Friedrich-Alexander-Universit\"at Erlangen-N\"urnberg, 91058 Erlangen, Germany.}


 \author[K. Kinra]{Kush Kinra$^{1,*}$}

	\begin{abstract}
			We study optimal boundary mixing of a nondiffusive scalar transported by a
		two-dimensional incompressible second-grade fluid.  The control is the
		tangential traction in a Navier-slip boundary condition, and the principal
		objective is the terminal $(H^1(\Omega))'$ mix-norm, supplemented by a
		quadratic control cost and an optional enstrophy reward.  The second-grade
		constitutive law introduces a spatially filtered acceleration and a
		generalized vorticity; after lifting the nonhomogeneous boundary data, both
		$\g$ and $\partial_t\g$ enter the state equation, where $\g$ is the boundary control variable. We prove global state well-posedness for a passive scalar and well-posedness on a common,
		control-independent local interval for an active scalar.  We then establish
		existence of an optimal control, directional differentiability of the
		control-to-state map in the topology required by the terminal objective, and
		a weak backward adjoint formulation.  A duality identity yields the first-order variational inequality.  In the absence of the friction coefficient $\beta$ and enstrophy reward weight $\zeta$ (that is, $\beta=\zeta=0$), the additional adjoint regularity implies uniqueness of the optimal control for
		all sufficiently large control penalties $\gamma$.
	\end{abstract}

\maketitle


\section{Introduction}

Mixing by incompressible flows is a fundamental mechanism in many natural and
industrial processes.  When molecular diffusion is negligible on the time
scale under consideration, the evolution of a scalar concentration
$\theta$ is governed by the transport equation
\begin{equation}\label{eq:intro:transport}
	\partial_t\theta+\u\cdot\nabla\theta=0,
	\qquad \diver\u=0.
\end{equation}
If, in addition, the velocity satisfies the impermeability condition
$\u\cdot\n=0$ on the boundary, then all the $L^p$-norms of $\theta$ are
conserved.  These norms therefore cannot distinguish a poorly mixed scalar
from one that has been rearranged into increasingly fine scales.  The
underlying mechanism is the repeated stretching and folding of material sets,
which lies at the heart of chaotic advection
\cite{Aref1984,Ottino1989}.  Negative Sobolev norms overcome the limitation of
the conserved $L^p$-quantities by assigning greater weight to the large spatial
scales and have consequently become standard quantitative measures of mixing;
see, for example,
\cite{MathewMezicPetzold2005,LinThiffeaultDoering2011,Thiffeault2012}.  In a
bounded domain $\O$, we use the dual norm of $H^1(\O)$,
\begin{equation}\label{eq:intro:mix-norm}
	\|f\|_{(H^1(\O))'}^2
	:=\|\Lambda^{-1}f\|_{L^2(\O)}^2
	=(f,\Lambda^{-2}f),
\end{equation}
where $\Lambda^2=I-\Delta_N$ is the Neumann realization of $I-\Delta$.  Thus,
$\eta=\Lambda^{-2}f$ is the solution of
\begin{equation*}
	-\Delta\eta+\eta=f \quad\text{in }\O,
	\qquad \partial_{\n}\eta=0 \quad\text{on }\Gamma.
\end{equation*}
The decay of \eqref{eq:intro:mix-norm}, while the $L^2$-norm is conserved,
quantifies the transfer of scalar fluctuations from large to small spatial
scales.

The optimization of this transfer has been studied under a variety of
constraints on the stirring field.  Mathew, Mezi\'c, Grivopoulos, Vaidya, and
Petzold formulated optimal control problems for advective mixing in Stokes
flows using a negative Sobolev mix-norm \cite{MathewMezicEtAl2007}.  Lin,
Thiffeault, and Doering investigated instantaneous optimal stirring under
fixed energetic constraints \cite{LinThiffeaultDoering2011}, while Lunasin,
Lin, Novikov, Mazzucato, and Doering compared optimal stirring subject to fixed
energy, power, and palenstrophy \cite{LunasinEtAl2012}.  These works established
the effectiveness of mix-norm-based objectives, but they prescribe or optimize
the stirring field in the interior.  In many applications, by contrast, only
the boundary of the vessel is accessible to actuation.

A natural way to enhance this transfer is to actuate the fluid through the
boundary of the container.  Such a mechanism is motivated by wall-driven
stirring and has the advantage that the control does not act directly in the
interior of the fluid.  Boundary feedback for mixing in a two-dimensional
channel was considered in \cite{AamoKrsticBewley2003}.  For nondiffusive
scalars, Hu initiated a systematic Navier-slip boundary-control theory for
unsteady Stokes flows in \cite{HuStokes2018}.  Hu and Wu subsequently treated
two-dimensional Navier-Stokes flows, including both passive and active
scalars, and established existence of optimal controls and first-order
necessary conditions \cite{HuWuSICON2018}.  Because the zero-diffusivity
transport equation imposes strong regularity requirements on the
control-to-state map, vanishing-diffusivity approximations were developed for
Stokes flows in \cite{HuApproxStokes2020} and for Navier-Stokes flows in
\cite{HuWuJDE2019}.  More recently, Hu, Rautenberg, and Zheng constructed a
nonlinear feedback strategy for advective mixing under an energy constraint
\cite{HuRautenbergZheng2023}, while Zheng, Hu, and Wu developed and tested
numerical algorithms for dynamic boundary control in the unsteady Stokes
setting \cite{ZhengHuWu2023}.


Many fluids relevant to mixing, however, exhibit non-Newtonian effects that
are not represented by the Navier-Stokes equations.  Second-grade fluids form
a classical family of incompressible differential-type fluids and provide a
model for materials with a weak elastic response.  Their constitutive origins
go back to \cite{RivlinEricksen1955,DunnFosdick1974}.  Analytical results for
second-grade fluids with Navier-slip boundary conditions can be found in
\cite{BusuiocRatiu2003,BusuiocIftimie2006}; the spatial estimates associated
with the generalized vorticity are also central in the stochastic
well-posedness analysis of \cite{ChemetovCipriano2017}.

Boundary optimal mixing is by now well developed for Newtonian Stokes and
Navier--Stokes flows; see, in particular,
\cite{HuStokes2018,HuWuSICON2018,HuWuJDE2019}.  A separate literature studies
optimal control of second-grade fluids, primarily for distributed controls
and quadratic velocity-tracking objectives
\cite{AradaCipriano2015,AlmeidaChemetovCipriano2022}.  We are not aware of an
earlier analysis that combines tangential boundary-traction control,
second-grade-fluid dynamics, a nondiffusive transported scalar, and a
terminal negative-Sobolev mixing objective.

Some state estimates can be adapted from the existing second-grade-fluid
theory, and the choice of the mix-norm is motivated by the Newtonian mixing
literature.  The analytical framework needed here is nevertheless different.
The nonhomogeneous traction requires a high-order Stokes lifting; the filtered
acceleration makes both the control and its time derivative appear after the
lifting; the scalar equation supplies no parabolic smoothing; and the natural
adjoint regularity does not justify a formal strong tangential boundary
condition.  Resolving these points, proving terminal-state differentiability,
and obtaining a lifting-independent optimality condition constitute the main
new analytical contribution.

The purpose of this paper is to develop a boundary optimal-control theory for
mixing by a two-dimensional second-grade fluid. Let $T>0$, and let $\O\subset\R^2$ be a bounded, simply connected domain
	whose boundary $\Gamma$ is of class $C^{4,1}$.  Set
\begin{equation*}
	\Upsilon_\alpha(\u):=\u-\alpha\Delta\u,
	\qquad
	\Db(\u):=\frac12\bigl(\nabla\u+(\nabla\u)^{\mathsf T}\bigr).
\end{equation*}
and consider the coupled system
\begin{equation}\label{state:eq:state}
	\left\{
	\begin{aligned}
		\partial_t\theta+\u\cdot\nabla\theta&=0,
		&&\text{in }(0,T)\times\O,\\
		\partial_t\Upsilon_\alpha(\u)-\nu\Delta\u
		+\curl\Upsilon_\alpha(\u)\times\u+\nabla p
		&=\xi\theta e_2,
		&&\text{in }(0,T)\times\O,\\
		\diver\u&=0,
		&&\text{in }(0,T)\times\O,\\
		\theta(0)=\theta_0, \;\; \u(0)& =\u_0,
		&&\text{in }\O.
	\end{aligned}
	\right.
\end{equation}
 Here $\nu>0$ is the viscosity and $\alpha>0$ is the second-grade, or
viscoelastic, material parameter.  Through
$\Upsilon_\alpha(\u)=\u-\alpha\Delta\u$, the parameter $\alpha$ weights the
small-scale spatial contribution to the generalized momentum and vorticity.
In the present notation, $\alpha$ has the physical dimension of length
squared because it occurs in the Helmholtz-type operator
$I-\alpha\Delta$. Thus, setting $\ell_\alpha:=\sqrt{\alpha}$, the
quantity $\ell_\alpha$ represents the intrinsic length scale associated
with the second-grade correction; cf.~\cite{OliverShkoller2001}, where
the corresponding operator is written as $I-\ell_\alpha^2\Delta$.
Under a nondimensionalization based on a characteristic length $L_0$,
the dimensionless second-grade parameter is
$\alpha/L_0^2=(\ell_\alpha/L_0)^2$.  Formally,
$\alpha\to0$ recovers the Navier--Stokes momentum equation.  No monotone
dependence of the attainable mixing on $\alpha$ is asserted here; identifying
that dependence is a natural analytical and computational question.
The parameter $\xi\in\{0,1\}$ distinguishes the passive and active cases:
$\xi=0$ gives a passive scalar, whereas $\xi=1$ gives a Boussinesq-type active
scalar.  The boundary control
$\g$ enters through the nonhomogeneous Navier-slip conditions
\begin{equation}\label{state:eq:boundary}
	\u\cdot\n=0,
	\qquad
	\bigl[2\nu\n\cdot\Db(\u)+\beta\u\bigr]\cdot\tau
	=\g\cdot\tau
	\quad\text{on }(0,T)\times\Gamma,
\end{equation}
where $\beta\geq0$ is the friction coefficient, $\n$ is the outward unit normal and $\tau$ denotes the unit
tangent to $\Gamma$.  The control is therefore the tangential traction exerted
at the boundary, rather than a force distributed throughout the domain.

\begin{remark}
	In equation \eqref{state:eq:state}, the vector product $\times$ for 2D vectors $\y=(y_1,y_2)$ and $\z=(z_1,z_2)$ is calculated as $\y\times\z= (y_1,y_2,0)\times (z_1,z_2,0)$. The curl of vector $\y$ is equal to $\curl \y = \frac{\partial y_2}{\partial x_1}-\frac{\partial y_1}{\partial x_2}$ and the vector product of $\curl \y$ with the vector $\z$ is understood as 
	\begin{align*}
		\curl\y\times\z = (0,0,\curl\y)\times(z_1,z_2,0).
	\end{align*}
\end{remark}

For the time $T_\xi$ on which the state system is well-posed, we minimize
\begin{equation}\label{eq:intro:cost}
	\J(\g)
	:=\frac12\|\theta(T_\xi;\g)\|_{(H^1(\O))'}^2
	+\frac\gamma2\|\g\|_{\Ucal}^2
	-\frac\zeta2\int_0^{T_\xi}
	\|\curl\u(t;\g)\|_{L^2(\O)}^2\,dt,
\end{equation}
over the closed, bounded, and convex set $\Uad$.  Here $\gamma>0$ is the
control-penalty parameter, $\zeta\geq0$ weights the enstrophy reward, and
\begin{equation}\label{eq:intro:control-space}
	\Ucal
	:=\left\{
	\g\in L^2\bigl(0,T;\Vb^{5/2}(\Gamma)\bigr):
	\partial_t\g\in L^2\bigl(0,T;\Vb^{3/2}(\Gamma)\bigr)
	\right\}.
\end{equation}
The admissible controls additionally satisfy $\|\g\|_{\Ucal}\leq M$ and the
compatibility condition at $t=0$ determined by \eqref{state:eq:boundary} and
$\u_0$.  The terminal negative Sobolev norm is the principal mixing objective, whereas
	the enstrophy term is an optional reward for time-integrated rotational
	activity.  Enstrophy is only a proxy for stirring and does not by itself imply
	a smaller terminal mix-norm.  For every fixed $\zeta\geq0$, boundedness of
	$\Uad$ and the uniform state estimate keep $\J$ bounded from below, so no
	upper restriction on $\zeta$ is needed for existence.  
	

The passage from a Newtonian to a second-grade fluid is not a formal
replacement in the existing boundary-control theory.  Indeed, the term
$\partial_t\Upsilon_\alpha(\u)$ couples the time derivative to a second-order
spatial operator, while the nonlinearity involves the generalized vorticity
$\curl\Upsilon_\alpha(\u)$.  Moreover, after the nonhomogeneous boundary condition
is incorporated into the weak formulation, both $\g$ and
$\partial_t\g$ occur in the boundary forcing.  This explains the temporal
regularity in \eqref{eq:intro:control-space} and requires a sufficiently
regular Stokes lifting of the boundary data.  For the active scalar, there is
an additional two-way coupling: the scalar drives the velocity through the
buoyancy term, while the velocity transports the scalar.  At the level needed
for the control problem, the resulting high-order estimates close on a
control-independent local interval, whereas the passive problem is global in
time.

The nondiffusive transport equation creates a further difficulty in the
sensitivity analysis.  Although the velocity depends regularly on the
boundary traction, the scalar equation has no parabolic smoothing.  Strong
stability of the transported scalar must therefore be proved directly from
the flow stability.  This issue is particularly delicate in establishing the
directional derivative of the control-to-state map at the terminal time in the
topology used by the mix-norm.  Finally, the second-grade nonlinearities and
the nonhomogeneous boundary input make a strong adjoint boundary condition
unavailable at the natural energy level.  We consequently formulate the
velocity adjoint variationally and express the reduced derivative through a
lifting of the boundary variation.  A key point is to prove that the resulting
optimality condition does not depend on the particular lifting.

  Within the literature described above, we are not aware of an earlier analytical treatment of boundary optimal mixing by a second-grade fluid for
either passive or Boussinesq-type active nondiffusive scalars.  The main
contributions are as follows.
\begin{itemize}
	\item We establish existence and uniqueness for the state system
	\eqref{state:eq:state}-\eqref{state:eq:boundary}.  For $\xi=0$, the solution
	exists globally on $[0,T]$; for $\xi=1$, it exists on a time interval
	$[0,T_\xi]$ (see \eqref{eqn:time} below) that is uniform with respect to $\g\in\Uad$.  In particular,
	$\u\in L^\infty(0,T_\xi;H^3(\O))$, which supplies the spatial Lipschitz
	regularity required by the nondiffusive transport equation.
	
	\item Using uniform state estimates, compactness, and lower
	semicontinuity, we prove that the minimization problem associated with
	\eqref{eq:intro:cost} admits at least one optimal control.  The boundedness
	of $\Uad$ also permits the negative enstrophy term in the objective.
	
	\item We prove quantitative Lipschitz stability with respect to the boundary
	control, establish well-posedness of the linearized state system, and show
	that the control-to-state map is directionally differentiable.  In particular,
	the linearized scalar is identified as the derivative of the terminal state
	in $(H^1(\O))'$, which is the topology required to differentiate the mixing
	term.
	
	\item We construct a unique weak solution of the backward adjoint system and
	derive a duality identity between the linearized and adjoint variables.  This
	yields a variational inequality characterizing every optimal control.  The
	reduced boundary functional is expressed through a Stokes lifting, and we
	prove that its value, and hence the first-order necessary condition, is
	independent of the chosen lifting.
	
	\item In the case $\beta=\zeta=0$, we obtain the additional adjoint
	regularity needed to compare two solutions of the optimality system.  We then
	prove the uniqueness of the optimal control whenever
	$\gamma>0$ is sufficiently large.
\end{itemize}

Thus, the paper connects two previously separate directions: boundary control
of mixing for Newtonian flows and optimal control of second-grade fluids.  In
contrast with the former, the controlled velocity is governed by a
higher-order nonlinear constitutive law; in contrast with the latter, the
quantity to be optimized is not a velocity-tracking functional but the
large-scale content of a transported, nondiffusive scalar.  The analysis also
keeps the passive and active cases in a common framework through the parameter
$\xi$, while making explicit the different time horizons imposed by the
coupling.


The remainder of the paper is organized as follows.  In
Section~\ref{sec1}, we formulate the state and control problems and introduce
the functional framework.  More specifically, the control problem and the
auxiliary estimates for the nonhomogeneous Navier-slip lifting are presented
in Subsections~\ref{subsec:control-problem} and
\ref{subsec:auxiliary-results}, respectively, while the Faedo-Galerkin
construction of the state solution is developed in
Subsection~\ref{subsec:galerkin}.  The well-posedness of the state system is
established in Section~\ref{Sec4}, and the existence of an optimal control is
proved in Section~\ref{sec:existence-optimal-control}.  The linearized system
and the corresponding stability estimates are studied in
Sections~\ref{sec:linearized-state} and~\ref{sec:stability}, respectively,
whereas the directional differentiability of the control-to-state mapping is
proved in Section~\ref{sec:gateaux}.  In Section~\ref{sec:adjoint}, we
construct the adjoint system, and in Subsection~\ref{subsec:duality} we derive
the associated duality relation.  The first-order necessary optimality
condition is obtained in Section~\ref{sec:first-order}.  Finally,
Section~\ref{sec:uniqueness} is devoted to the additional adjoint regularity
and the uniqueness of the optimal control for $\beta=\zeta=0$ and sufficiently
large $\gamma$.

\section{Mathematical formulation}\setcounter{equation}{0}\label{sec1}
In this section, we introduce the function spaces required for our analysis. We also formulate the control problem and present several auxiliary results that will be used throughout the paper. Throughout the paper, $\O\subset\R^2$ is bounded and simply connected and $\Gamma=\partial\O$ is of class $C^{4,1}$. This regularity is used in the $H^4$ estimate for the boundary lifting and in the high-order elliptic estimates for the state and adjoint systems. The assumption of simple connectedness avoids the harmonic ambiguity that may occur in a stream-function description on multiply connected domains.


\subsection{Function spaces} 
Set
\begin{align*}
	\Vb^s (\O) & := \{ \u\in H^s(\O) \; : \; \diver\u=0,\; \u\cdot\n|_{\Gamma}=0 \}, \;  && s\geq0,\\
	\Vb^s (\Gamma) & := \{ \g\in H^s(\Gamma) \; : \;  \g\cdot\n|_{\Gamma}=0 \}, \;  && s\geq0,\\
	\Wb & := \{ \u \in \Vb^2(\O) \; : \;[2\nu \n \cdot \Db(\u)  + \beta \u]\cdot\tau|_{\Gamma} =0 \text{ with } \beta\geq0 \},\\
	\wiWb &  := \Wb \cap H^3(\O). 
\end{align*}
We denote $(\cdot,\cdot)$ as the inner product in $L^2(\O)$ and we write
$\mathbb P$ for the Leray projector onto $\Vb^0(\O)$. For any Banach space $\X$, the associate norm will be denoted by $\|\cdot\|_{\X}$.  

On the space $\Vb^1(\O)$, we define the following inner product:
\begin{align}\label{Vb1}
	(\u,\v)_{\Vb^1(\O)}   = (\u,\v) + 2\alpha(\Db(\u),\Db(\v)) + \frac{\alpha \beta}{\nu} \int_{\Gamma}(\u\cdot\tau)(\v\cdot\tau) d S, \text{ for } \u,\v\in \Vb^1(\O),
\end{align}
and the corresponding norm $\|\cdot\|_{\Vb^1(\O)}$.  We denote by $\mathcal{M}_{\alpha\beta}:\Vb^1(\O)\to(\Vb^1(\O))'$ the operator induced by
\eqref{Vb1}:
\begin{align}
	\langle\mathcal{M}_{\alpha\beta} \u_1 , \u_2 \rangle = (\u_1,\u_2)_{\Vb^1(\O)}, \text{ for } \u_1,\u_2\in \Vb^1(\O).
\end{align}
Moreover, for $\u\in\Wb$ and $\v\in \Vb^1(\O)$, we have
\begin{align}\label{eqn-IbP}
	(\u,\v)_{\Vb^1(\O)}  =  (\Upsilon_\alpha(\u) , \v).
\end{align}

Next, on the space $\wiWb$, we define the following inner product:
\begin{align}\label{wiWb:inner:product}
	(\u,\v)_{\wiWb}  = (\curl\Upsilon_\alpha(\u), \curl\Upsilon_\alpha(\v) ) + (\u,\v)_{\Vb^1(\O)},  \text{ for } \u,\v\in \wiWb,
\end{align}
and the corresponding norm $\|\cdot\|_{\wiWb}$.

In addition, we have the following estimates from \cite{ChemetovCipriano2017} which will be used in the sequel.

\begin{lemma}[{\cite[Lemma 3.3]{ChemetovCipriano2017}}]
	For each $\u\in \Wb$, we have 
	\begin{align*}
		\|\Upsilon_\alpha(\u) - \mathbb{P} \Upsilon_\alpha(\u)\|_{L^2(\O)} \leq \|\u\|_{H^1(\O)},\\
		\|\Upsilon_\alpha(\u) - \mathbb{P} \Upsilon_\alpha(\u)\|_{H^1(\O)} \leq \|\u\|_{H^2(\O)}.
	\end{align*}
\end{lemma}

\begin{lemma}[{\cite[Lemma 3.4]{ChemetovCipriano2017}}]\label{lem:1.6}
	We have 
	\begin{align}
		\|\u\|_{H^2(\O)} & \leq  C ( \|\mathbb{P} \Upsilon_\alpha(\u)\|_{L^2(\O)} +  \|\u\|_{H^1(\O)}), && \text{ for all } \; \u\in  \Wb; \label{lem:1.6:1}\\
		\|\u\|_{H^3(\O)} & \leq  C ( \|\curl \Upsilon_\alpha(\u)\|_{L^2(\O)} +  \|\u\|_{H^1(\O)}), && \text{ for all } \; \u\in  \wiWb,\label{lem:1.6:2}
	\end{align}
	for some constant $C>0$.
\end{lemma}

\begin{lemma}[{\cite[Lemma 3.5]{ChemetovCipriano2017}}]\label{lem:curl:estimates}
	We have 
	\begin{equation}\label{curl:estimates}
		\left\{       
		\begin{aligned}
			|(\curl \Upsilon_\alpha(\u) \times \v, \w)| & \leq  C \|\u\|_{H^3(\O)} \|\v\|_{L^2(\O)} \|\w\|_{H^2(\O)}, 
			    \text{ for all } \; \u\in  H^3(\O), \v\in L^2(\O), \w \in H^2(\O);\\
			|(\curl \Upsilon_\alpha(\u) \times \v, \w)| & \leq  C \|\u\|_{H^1(\O)} \|\v\|_{H^3(\O)} \|\w\|_{H^3(\O)},   \text{ for all } \; \u\in  \wiWb, \v, \w \in H^3(\O);  \\
			|(\curl \Upsilon_\alpha(\u) \times \v, \u)| & \leq  C \|\u\|_{H^1(\O)}^2 \|\v\|_{H^3(\O)},    \qquad\qquad \text{ for all } \; \u\in  \wiWb, \v \in H^3(\O);
		\end{aligned}
		\right.
	\end{equation}
	for some constant $C>0$.
\end{lemma}

\begin{remark}
	Note that,
	\begin{itemize}
		\item     due to Korn inequality, the norms $\|\cdot\|_{H^1(\O)}$ and $\|\cdot\|_{\Vb^1(\O)}$ are equivalent, that is,
		\begin{align}\label{Korn:equivalent}
			\|\u\|_{H^1(\O)} \leq C (\|\Db(\u)\|_{L^2(\O)}^2 + \|\u\|_{L^2(\O)}^2 ), \;\;\; \text{ for all } \; \u \in \Vb^1(\O).
		\end{align}
		\item     the norms $\|\cdot\|_{H^3(\O)}$ and $\|\cdot\|_{\wiWb}$ are equivalent, see Lemma \ref{lem:1.6}.
	\end{itemize}
\end{remark}

Let us now define a trilinear functional $b:\Vb^1(\O) \times\Vb^1(\O) \times\Vb^1(\O) \to \R$ by 
\begin{align*}
	b(\u,\v,\w) = \int_{\O}[(\u(x)\cdot\nabla)\v(x)] \cdot \w(x) dx.
\end{align*}
With this notation, we have the following known equalities (see \cite{ChemetovCipriano2017}):
\begin{equation}\label{trilinear1}
	\left\{ \begin{aligned}
		b(\u,\v,\w) & = - b(\u,\w,\v), \qquad\qquad\qquad\qquad  \text{ for all } \; \u,\v,\w \in\Vb^1(\O), \\
		(\curl \Upsilon_\alpha(\u) \times \v, \w) & = b(\w, \v, \Upsilon_\alpha(\u)) - b(\v, \w, \Upsilon_\alpha(\u)),  
		\\ & \qquad\qquad \qquad
		\text{ for all } \; \u\in\Vb^1(\O)\cap H^3(\O) \text{ and } \v,\w \in\Vb^1(\O).
	\end{aligned}
	\right.
\end{equation}

\subsection{Control problem}\label{subsec:control-problem}
Let $T_\xi>0$ denote the time (see \eqref{eqn:time} below) such that the unique solution $\u$ to the state equation \eqref{state:eq:state} in the sense of Definition \ref{def:WS:state} below belongs to $L^{\infty}(0,{T}_{\xi};H^3(\O))$. 

Let  $(\theta,\u)$ be the unique solution to coupled system \eqref{state:eq:state}.  Our main goal is to control the norm $\|\theta(T_{\xi})\|_{(H^{1}(\O))^\prime}$ by a boundary control to the velocity field $\u$. 
Define the Hilbert control space
\begin{equation}\label{eq:control-space}
	\Ucal
	:=\left\{
	\g\in L^2(0,T;\Vb^{5/2}(\Gamma)):
	\partial_t\g\in L^2(0,T;\Vb^{3/2}(\Gamma))
	\right\},
\end{equation}
with norm
\begin{equation}\label{eq:control-norm}
	\|\g\|_{\Ucal}^2
	:=\|\g\|_{L^2(0,T;\Vb^{5/2}(\Gamma))}^2
	+\|\partial_t\g\|_{L^2(0,T;\Vb^{3/2}(\Gamma))}^2.
\end{equation}
The Lions-Magenes trace theorem gives a bounded linear map
\begin{equation}\label{eq:control-trace}
	\operatorname{Tr}_0:\Ucal\to
	(\Vb^{3/2}(\Gamma),\Vb^{5/2}(\Gamma))_{1/2,2}
	=\Vb^2(\Gamma),
	\qquad
	\|\g(0)\|_{H^2(\Gamma)}\leq C_{\rm tr}\|\g\|_{\Ucal}.
\end{equation}

Assume that $\u_0\in\Vb^3(\O)$ and that its tangential Navier traction
defines the element
\begin{equation}\label{eq:initial-control-trace}
	\g_{\rm in}
	:=\left(
	\bigl[2\nu\n\cdot\Db(\u_0)+\beta\u_0\bigr]\cdot\tau
	\right)\tau
	\in\Vb^2(\Gamma).
\end{equation}
Choose $M>0$ sufficiently large that there exists
$\g_{\rm ref}\in\Ucal$ with
$\g_{\rm ref}(0)=\g_{\rm in}$ and
$\|\g_{\rm ref}\|_{\Ucal}\leq M$, and define
\begin{equation}\label{eqn:admissible:set}
	\Uad
	:=\left\{
	\g\in\Ucal:
	\|\g\|_{\Ucal}\leq M,
	\quad \g(0)=\g_{\rm in}
	\right\}.
\end{equation}
Thus $\Uad$ is nonempty, closed, bounded, and convex in $\Ucal$.

\begin{remark}[Compatibility with the initial data]
	The condition in \eqref{eqn:admissible:set} is equivalent to
	\[
	\g(0)\cdot\tau
	=\bigl[2\nu\n\cdot\Db(\u_0)+\beta\u_0\bigr]\cdot\tau
	\quad\text{on }\Gamma,
	\]
	because every element of $\Ucal$ has zero normal component. Let
	$\Gbf=\mathcal R\g$ be the lifting introduced below and set
	$\v=\u-\Gbf$. Then
	\[
	\v(0)=\u_0-\Gbf(0),\qquad \v(0)\cdot\n=0,
	\]
	and
	\begin{align*}
		&\bigl[2\nu\n\cdot\Db(\v(0))+\beta\v(0)\bigr]\cdot\tau 
		=
		\bigl[2\nu\n\cdot\Db(\u_0)+\beta\u_0\bigr]\cdot\tau
		-\g(0)\cdot\tau=0.
	\end{align*}
	Hence $\v(0)\in\wiWb$. Moreover, if $\g_1,\g_2\in\Uad$, then
	$(\g_1-\g_2)(0)=0$. Therefore the lifting of every admissible variation
	has zero initial value. If $\u_0$ satisfies the homogeneous Navier-slip
	condition, then $\g_{\rm in}=0$ and the compatibility condition reduces to
	$\g(0)=\boldsymbol{0}$.
\end{remark}

It is useful to introduce the affine compatibility space and its tangent
space,
\begin{equation}\label{eq:affine-control-space}
	\Ucal_{\g_{\rm in}}
	:=\{\g\in\Ucal:\g(0)=\g_{\rm in}\}
	=\g_{\rm ref}+\Ucal_0,
	\qquad
	\Ucal_0:=\{\h\in\Ucal:\h(0)=0\}.
\end{equation}

For $\g\in\Uad$, let $(\theta(\g),\u(\g))$ denote the solution of \eqref{state:eq:state} on $[0,T_\xi]$.  Define $\J:\Uad\to\R^+$ by
\begin{equation}\label{control:eq:cost}
 \J(\g)
 :=\frac12 \|\theta(T_{\xi};\g)\|_{(H^1(\O))'}^2
 +\frac\gamma2\|\g\|_{\Ucal}^2
 -\frac\zeta2\int_0^{T_{\xi}}\|\curl\u(t;\g)\|_{L^2(\O)}^2 d t,
\end{equation}
where $\gamma>0$ and $\zeta\geq0$. The optimal-control problem reads as
\begin{align}\label{eqn:control:problem}
    \min_{\g\in\Uad}\bigg\{\J(\g) \; : \; (\theta(\g),\u(\g)) \text{ is the unique solution of \eqref{state:eq:state} with given boundary data } \g.\bigg\}.
\end{align}

\subsection{Auxiliary results}\label{subsec:auxiliary-results}

We use the coercive Stokes resolvent to lift the nonhomogeneous tangential
traction. For a tangential boundary field $\g$, consider
\begin{equation}\label{Stokes}
	\left\{
	\begin{aligned}
		-\Delta\Gbf+\Gbf+\nabla\pi_{\Gbf}&=0,
		&\diver\Gbf&=0 &&\text{in }\O,\\
		\Gbf\cdot\n&=0,
		&\bigl[2\nu\n\cdot\Db(\Gbf)+\beta\Gbf\bigr]\cdot\tau
		&=\g\cdot\tau &&\text{on }\Gamma,
	\end{aligned}
	\right.
\end{equation}
where the pressure is normalized by
$\int_\O\pi_{\Gbf}\,dx=0$.

The following lemmas are used to transform the system with nonhomogeneous
boundary data into a system with homogeneous boundary data.
\begin{lemma}[{\cite[Lemma 2.2]{HuWuJDE2019}}]\label{vorticit-boundary}
	Let  $\O\subset\R^2$ be an open bounded and connected domain with boundary $\Gamma\in C^2$. Assume that $\v\in C^1(\overline{\Omega})$ satisfying the Navier-slip boundary conditions
	\begin{align*}
		\v \cdot\n =0 \;\;\; & \text{ and } \;\;\; [2\nu \n \cdot \Db(\v )  + \beta \v ] \cdot\tau  = \g\cdot\tau, \;\; \text{ on } \;\; \Gamma. 
	\end{align*}
	Then, we have
	\begin{align*}
		\curl \v = \left(2\kappa -\frac{\beta}{\nu}\right)(\v\cdot\tau) + \frac{1}{\nu}(\g\cdot\tau) , \;\; \text{ on } \;\;  \Gamma,
	\end{align*}
	where    $\kappa$ denotes the signed curvature of $\Gamma$.
\end{lemma}

\begin{lemma}[Stokes lifting]\label{Lifting:estimates}
	Let $\O\subset\R^2$ be a bounded domain whose boundary $\Gamma$ is of
	class $C^{4,1}$, and let $\beta\geq0$. For
	$s\in\{3/2,5/2\}$ and $\g\in\Vb^s(\Gamma)$, problem
	\eqref{Stokes} has a unique solution
	$(\Gbf,\pi_{\Gbf})$, with the normalization
	$\int_\O\pi_{\Gbf}\,dx=0$. The solution operator
	\[
	\mathcal R:\g\mapsto\Gbf
	\]
	is linear and satisfies
	\begin{align}
		\|\mathcal R\g\|_{H^3(\O)}
		&\leq C\|\g\|_{H^{3/2}(\Gamma)},
		\label{H3-estimate}\\
		\|\mathcal R\g\|_{H^4(\O)}
		&\leq C\|\g\|_{H^{5/2}(\Gamma)}.
		\label{H4-estimate}
	\end{align}
	Consequently, for every $\g\in\Ucal$,
	\begin{equation}\label{eq:time-lifting-estimate}
		\mathcal R\g
		\in L^2(0,T;H^4(\O))
		\cap H^1(0,T;H^3(\O)),
		\qquad
		\partial_t(\mathcal R\g)
		=\mathcal R(\partial_t\g),
	\end{equation}
	and
	\begin{equation}\label{eq:time-lifting-bound}
		\|\mathcal R\g\|_{L^2(0,T;H^4(\O))}
		+\|\partial_t\mathcal R\g\|_{L^2(0,T;H^3(\O))}
		\leq C_R\|\g\|_{\Ucal}.
	\end{equation}
	The constants $C$ and $C_R$ depend only on $\O$, $\nu$, and $\beta$,
	and are independent of $\g$.
\end{lemma}

\begin{proof}
	We divide the proof into several steps.
	
	\medskip
	\noindent\textbf{Step 1: Variational solvability and the basic $H^2$ estimate.}
	The weak formulation of \eqref{Stokes} is
	\begin{align}\label{eq:weak-stokes-lifting}
		&2(\Db(\Gbf),\Db(\boldsymbol\phi))
		+(\Gbf,\boldsymbol\phi)
		+\frac{\beta}{\nu}
		\int_\Gamma
		(\Gbf\cdot\tau)(\boldsymbol\phi\cdot\tau)\,dS
		=
		\frac1\nu
		\left\langle
		\g\cdot\tau,\boldsymbol\phi\cdot\tau
		\right\rangle_{H^{-1/2}(\Gamma)\times H^{1/2}(\Gamma)},
	\end{align}
for every $\boldsymbol\phi\in\Vb^1(\O)$. The bilinear form on the left-hand side is continuous and coercive on
	$\Vb^1(\O)$, by Korn's inequality and the presence of the
	$L^2(\O)$ term. Therefore, the Lax-Milgram theorem gives a unique weak
	velocity $\Gbf\in\Vb^1(\O)$. The associated pressure is recovered by
	the de Rham theorem and is uniquely determined by the normalization
	$\int_\O\pi_{\Gbf}\,dx=0$.
	
	Moreover,
	\begin{equation}\label{eq:weak-lifting-estimate}
		\|\Gbf\|_{H^1(\O)}
		\leq C\|\g\|_{H^{-1/2}(\Gamma)}.
	\end{equation}
	Writing the tangential boundary condition as
	\[
	[2\Db(\Gbf)\n]\cdot\tau
	+\frac{\beta}{\nu}\Gbf\cdot\tau
	=
	\frac1\nu\g\cdot\tau,
	\]
	the $H^2$ regularity theorem for the stationary Stokes system with
	Navier boundary conditions (see for example \cite[Theorem 4.5]{AcevedoAmroucheConcaGhosh2021}), applied with the interior forcing
	$-\Gbf\in L^2(\O)$, gives
	\begin{align*}
		\|\Gbf\|_{H^2(\O)}
		+\|\pi_{\Gbf}\|_{H^1(\O)/\R}
		&\leq
		C\left(
		\|\Gbf\|_{L^2(\O)}
		+\|\g\|_{H^{1/2}(\Gamma)}
		\right)
		\leq C\|\g\|_{H^{1/2}(\Gamma)}.
	\end{align*}
	 Thus,
	\begin{equation}\label{H2-estimate}
		\|\Gbf\|_{H^2(\O)}
		+\|\pi_{\Gbf}\|_{H^1(\O)/\R}
		\leq C\|\g\|_{H^{1/2}(\Gamma)}.
	\end{equation}
	
	If $\g=0$, taking $\boldsymbol\phi=\Gbf$ in
	\eqref{eq:weak-stokes-lifting} gives
	\[
	2\|\Db(\Gbf)\|_{L^2(\O)}^2
	+\|\Gbf\|_{L^2(\O)}^2
	+\frac{\beta}{\nu}
	\|\Gbf\cdot\tau\|_{L^2(\Gamma)}^2
	=0.
	\]
	Hence $\Gbf=0$. The equation then gives
	$\nabla\pi_{\Gbf}=0$, and the pressure normalization implies
	$\pi_{\Gbf}=0$. This also proves uniqueness. Linearity of
	$\mathcal R$ follows directly from the linearity of \eqref{Stokes}.
	
	\medskip
	\noindent\textbf{Step 2: Regularity of the curvature and the boundary
		multipliers.}
	Let $d$ be the signed distance to $\Gamma$ in a tubular neighborhood
	of $\Gamma$. Since $\Gamma$ is of class $C^{4,1}$,
	\cite[Theorem 1]{LiNirenberg2005} implies that
	\[
	d\in C^{4,1}.
	\]
	On $\Gamma$, up to the sign determined by the orientation,
	\[
	\n=\nabla d,
	\qquad
	\kappa=\diver\n=\Delta d.
	\]
	Consequently,
	\begin{equation}\label{eq:normal-curvature-regularity}
		\n\in C^{3,1}(\Gamma),
		\qquad
		\tau\in C^{3,1}(\Gamma),
		\qquad
		\kappa\in C^{2,1}(\Gamma).
	\end{equation}
	Here $\tau$ is obtained from $\n$ by a fixed rotation through
	$\pi/2$. In local boundary coordinates, Rademacher's theorem gives
	\[
	C^{2,1}(\Gamma)
	\hookrightarrow W^{3,\infty}(\Gamma)
	\hookrightarrow H^3(\Gamma);
	\]
	see, for example, \cite[Section~3.1]{EvansGariepy2015} for the
	Lipschitz differentiability result.
	
	More importantly, the Sobolev multiplier theorem
	\cite[Theorem 7.28]{BehzadanHolst2022}, applied in a finite family of
	one-dimensional boundary charts and combined with a partition of
	unity, shows that multiplication by a $C^{2,1}(\Gamma)$ function is
	bounded on $H^q(\Gamma)$ for every $0\leq q\leq5/2$. Therefore,
	\begin{equation}\label{eq:kappa-multiplier}
		\|\kappa\psi\|_{H^q(\Gamma)}
		\leq C_\kappa\|\psi\|_{H^q(\Gamma)},
		\qquad
		0\leq q\leq\frac52.
	\end{equation}
	Likewise, since $\tau\in C^{3,1}(\Gamma)$,
	\begin{equation}\label{eq:tau-multiplier}
		\|\v\cdot\tau\|_{H^q(\Gamma)}
		\leq C_\tau\|\v\|_{H^q(\Gamma)},
		\qquad
		0\leq q\leq\frac52.
	\end{equation}
	In particular, the coefficient
	$2\kappa-\beta/\nu$ is a bounded multiplier on both
	$H^{3/2}(\Gamma)$ and $H^{5/2}(\Gamma)$.
	
	\medskip
	\noindent\textbf{Step 3: The vorticity problem.}
	Set
	\[
	\mathcal G:=\curl\Gbf.
	\]
	Taking the curl of the first equation in \eqref{Stokes} gives
	\[
	-\Delta\mathcal G+\mathcal G=0
	\qquad\text{in }\O.
	\]
	The Navier-slip vorticity identity in
	Lemma~\ref{vorticit-boundary}, first for smooth functions and then by
	density in the appropriate trace spaces, gives
	\begin{equation}\label{eq:lift-vorticity}
		\left\{
		\begin{aligned}
			-\Delta\mathcal G+\mathcal G&=0
			&&\text{in }\O,\\
			\mathcal G&=a
			&&\text{on }\Gamma,
		\end{aligned}
		\right.
		\qquad
		a:=
		\left(2\kappa-\frac{\beta}{\nu}\right)
		(\Gbf\cdot\tau)
		+\frac1\nu(\g\cdot\tau).
	\end{equation}
	
	For later use, the trace theorem
	\cite[Theorem 1.5.1.2]{Grisvard1985}, together with
	\eqref{eq:kappa-multiplier}-\eqref{eq:tau-multiplier}, gives
	\begin{equation}\label{eq:a-general-estimate}
		\|a\|_{H^{m-1/2}(\Gamma)}
		\leq
		C\left(
		\|\Gbf\|_{H^m(\O)}
		+\|\g\|_{H^{m-1/2}(\Gamma)}
		\right),
		\qquad m\in\{2,3\}.
	\end{equation}
	
	\medskip
	\noindent\textbf{Step 4: The $H^3$ estimate.}
	Assume first that $\g\in\Vb^{3/2}(\Gamma)$. Taking $m=2$ in
	\eqref{eq:a-general-estimate} and using \eqref{H2-estimate}, we obtain
	\begin{align}
		\|a\|_{H^{3/2}(\Gamma)}
		&\leq
		C\left(
		\|\Gbf\|_{H^2(\O)}
		+\|\g\|_{H^{3/2}(\Gamma)}
		\right) 
		\leq C\|\g\|_{H^{3/2}(\Gamma)}.
		\label{eq:a-H32}
	\end{align}
	The regularity theory for the scalar Dirichlet problem
	\eqref{eq:lift-vorticity}; see
	\cite[Theorems 2.4.2.5 and 2.5.1.1]{Grisvard1985}, yields
	\begin{equation}\label{eq:vorticity-H2}
		\|\mathcal G\|_{H^2(\O)}
		\leq C\|a\|_{H^{3/2}(\Gamma)}
		\leq C\|\g\|_{H^{3/2}(\Gamma)}.
	\end{equation}
	
	The Sobolev div-curl regularity estimate for fields with prescribed
	normal trace, cf. \cite{ABDG1998}, gives
	\begin{align*}
		\|\Gbf\|_{H^3(\O)}
		\leq C\bigl(
		&\|\curl\Gbf\|_{H^2(\O)}
		+\|\diver\Gbf\|_{H^2(\O)}
		+\|\Gbf\cdot\n\|_{H^{5/2}(\Gamma)}
		+\|\Gbf\|_{L^2(\O)}
		\bigr).
	\end{align*}
	Since $\diver\Gbf=0$ and $\Gbf\cdot\n=0$, estimates
	\eqref{H2-estimate} and \eqref{eq:vorticity-H2} imply
	\[
	\|\Gbf\|_{H^3(\O)}
	\leq C\|\g\|_{H^{3/2}(\Gamma)}.
	\]
	This proves \eqref{H3-estimate}.
	
	\medskip
	\noindent\textbf{Step 5: The $H^4$ estimate.}
	Let now $\g\in\Vb^{5/2}(\Gamma)$. The estimate just proved implies
	\[
	\|\Gbf\|_{H^3(\O)}
	\leq
	C\|\g\|_{H^{3/2}(\Gamma)}
	\leq
	C\|\g\|_{H^{5/2}(\Gamma)}.
	\]
	Taking $m=3$ in \eqref{eq:a-general-estimate}, we obtain
	\begin{align}
		\|a\|_{H^{5/2}(\Gamma)}
		&\leq
		C\left(
		\|\Gbf\|_{H^3(\O)}
		+\|\g\|_{H^{5/2}(\Gamma)}
		\right)
		\leq C\|\g\|_{H^{5/2}(\Gamma)}.
		\label{eq:a-H52}
	\end{align}
	Another application of the Dirichlet regularity theorem gives
	\begin{equation}\label{eq:vorticity-H3}
		\|\mathcal G\|_{H^3(\O)}
		\leq C\|a\|_{H^{5/2}(\Gamma)}
		\leq C\|\g\|_{H^{5/2}(\Gamma)}.
	\end{equation}
	Applying the div-curl estimate at the next order, we find
	\begin{align*}
		\|\Gbf\|_{H^4(\O)}
		\leq C\bigl(
		&\|\curl\Gbf\|_{H^3(\O)}
		+\|\diver\Gbf\|_{H^3(\O)}
		+\|\Gbf\cdot\n\|_{H^{7/2}(\Gamma)}
		+\|\Gbf\|_{L^2(\O)}
		\bigr).
	\end{align*}
	Using $\diver\Gbf=0$, $\Gbf\cdot\n=0$,
	\eqref{H2-estimate}, and \eqref{eq:vorticity-H3}, we conclude that
	\[
	\|\Gbf\|_{H^4(\O)}
	\leq C\|\g\|_{H^{5/2}(\Gamma)}.
	\]
	This proves \eqref{H4-estimate}. 
	
	\medskip
	\noindent\textbf{Step 6: Time-dependent lifting.}
	By \eqref{H4-estimate},
	\[
	\|\mathcal R\g\|_{L^2(0,T;H^4(\O))}
	\leq
	C\|\g\|_{L^2(0,T;H^{5/2}(\Gamma))}.
	\]
	Since $\mathcal R:H^{3/2}(\Gamma)\to H^3(\O)$ is bounded and linear,
	a standard property of Bochner-Sobolev spaces gives
	\[
	\partial_t(\mathcal R\g)
	=\mathcal R(\partial_t\g)
	\quad\text{in }\mathcal D'(0,T;H^3(\O)).
	\]
	Hence, by \eqref{H3-estimate},
	\[
	\|\partial_t\mathcal R\g\|_{L^2(0,T;H^3(\O))}
	\leq
	C\|\partial_t\g\|_{L^2(0,T;H^{3/2}(\Gamma))}.
	\]
	Combining the last two estimates with the definition of
	$\|\g\|_{\Ucal}$ proves
	\eqref{eq:time-lifting-estimate} and
	\eqref{eq:time-lifting-bound}. This completes the proof.
\end{proof}


\subsection{Faedo-Galerkin approximation}\label{subsec:galerkin}
In this section, we introduce the finite-dimensional Faedo-Galerkin approximation scheme for the state system. We first construct suitable bases for the scalar and velocity components.

\subsubsection{ {Scalar basis.}}
Let $\{\varphi_j\}_{j\geq 1}\subset H^2(\Omega)$ be the
eigenfunctions of the Neumann Laplacian:
\begin{align*}
	-\Delta\varphi_j+\varphi_j
	&=\mu_j\varphi_j
	&&\text{in }\Omega, \\
	\partial_n\varphi_j
	&=0
	&&\text{on }\Gamma,
\end{align*}
such that $\{\varphi_j\}_{j\geq1}$ is an orthonormal basis of
$L^2(\Omega)$ and an orthogonal basis of $H^1(\Omega)$. In
particular,
\[
1\leq \mu_1\leq\mu_2\leq\cdots,
\qquad
\mu_j\to+\infty.
\]

\subsubsection{{Velocity basis.}}
Since the embedding $\wiWb \hookrightarrow \mathbb V^1(\Omega)$ is dense and compact, there exist eigenfunctions $\{\boldsymbol e_j\}_{j\geq1}\subset \wiWb$ such that $\{\boldsymbol e_j\}_{j\geq1}$ is an orthonormal basis of $\mathbb V^1(\Omega)$ and an orthogonal basis of $\wiWb$, and eigenvalues
$\{\lambda_j\}_{j\geq1}$ satisfying
\begin{align}
	(\boldsymbol e_j,\boldsymbol\phi)_
	{\wiWb}
	&=
	\lambda_j
	(\boldsymbol e_j,\boldsymbol\phi)_{\mathbb V^1(\Omega)},
	\qquad
	\boldsymbol\phi\in\wiWb,
	\label{eq:velocity-eigenbasis}
\end{align}
where
\[
0<\lambda_1\leq\lambda_2\leq\cdots,
\qquad
\lambda_j\to+\infty.
\]
 For $k\in\mathbb N$, define
\[
X_k:=\operatorname{span}\{\varphi_1,\ldots,\varphi_k\},
\qquad
Y_k:=\operatorname{span}
\{\boldsymbol e_1,\ldots,\boldsymbol e_k\}.
\]

\subsubsection{ {Galerkin variables.}}\label{Galerkin}
Let $\Gbf=\cR\g$ (see Lemma \ref{Lifting:estimates} for the operator $\cR$).  We seek approximate solutions in the form
\begin{align}
	\theta_k(t,x)
	&=
	\sum_{j=1}^{k}c_j^{k}(t)\varphi_j(x),
	\label{eq:theta-galerkin-expansion}\\
	\v_k(t,x)
	&=
	\sum_{j=1}^{k}a_j^{k}(t)\boldsymbol e_j(x),
	\label{eq:v-galerkin-expansion}\\
	\u_k(t,x)
	&:=
	\v_k(t,x)+\Gbf(t,x).
	\label{eq:u-galerkin-expansion}
\end{align}
Thus, $\v_k$ satisfies the homogeneous Navier-slip
condition, whereas $\u_k$ satisfies
\[
\u_k\cdot\boldsymbol n=0,
\qquad
\bigl(2\nu \Db(\u_k)\boldsymbol n
+\beta\u_k\bigr)\cdot\tau
=
\g\cdot \tau.
\]

Let $P_k^\theta:L^2(\Omega)\to X_k$ and
$P_k^{\v}:\mathbb V^1(\Omega)\to Y_k$ denote the corresponding
orthogonal projections, that is,
\begin{align*}
	P_k^\theta f
	&:=
	\sum_{j=1}^{k}
	(f,\varphi_j)_{L^2(\Omega)}\varphi_j, \;\;\;
	P_k^{\v}\boldsymbol w
	:=
	\sum_{j=1}^{k}
	(\boldsymbol w,\boldsymbol e_j)_{\mathbb V^1(\Omega)}
	\boldsymbol e_j.
\end{align*}
The approximate initial data are chosen as
\begin{align*}
	\theta_k(0)
	&=P_k^\theta\theta_0\;\; \text{ and } \;\;
	\v_k(0)
	=
	P_k^{\v}\bigl(\u_0-\Gbf(0)\bigr),
\end{align*}
so that
\[
\u_k(0)
=
\Gbf(0)
+
P_k^{\v}\bigl(\u_0-\Gbf(0)\bigr).
\]

\subsubsection{Finite-dimensional system.}
The coefficients
$\{c_j^k,a_j^k\}_{j=1}^{k}$ are determined by requiring that,
for every $i=1,\ldots,k$,
\begin{align}
	(\partial_t\theta_k,\varphi_i)_{L^2(\Omega)}
	+
	\bigl(
	(\v_k+\Gbf)\cdot\nabla\theta_k,
	\varphi_i
	\bigr)_{L^2(\Omega)}
	&=0,
	\label{eq:scalar-galerkin}
\end{align}
and
\begin{align}
	&(\partial_t\v_k,\boldsymbol e_i)_
	{\mathbb V^1(\Omega)}
	+2\nu
	(\Db(\v_k),\Db(\boldsymbol e_i))_{L^2(\Omega)}
	+\beta\int_\Gamma
	(\v_k\cdot\tau)
	(\boldsymbol e_i\cdot\tau)\,dS
	+b\bigl(
	\boldsymbol e_i,\v_k,
	\Upsilon_\alpha(\v_k)
	\bigr)
	\nonumber\\
	& -b\bigl(
	\v_k,\boldsymbol e_i,
	\Upsilon_\alpha(\v_k)
	\bigr)
	+b\bigl(
	\boldsymbol e_i,\Gbf,
	\Upsilon_\alpha(\v_k)
	\bigr)
	-b\bigl(
	\Gbf,\boldsymbol e_i,
	\Upsilon_\alpha(\v_k)
	\bigr)
	+b\bigl(
	\boldsymbol e_i,\v_k,
	\Upsilon_\alpha(\Gbf)
	\bigr)
	-b\bigl(
	\v_k,\boldsymbol e_i,
	\Upsilon_\alpha(\Gbf)
	\bigr)
	\nonumber\\
	&=
	\Bigl(
	\xi\theta_k\boldsymbol e_2
	-\partial_t\Upsilon_\alpha(\Gbf)
	+\nu\Delta\Gbf
	-\operatorname{curl}\Upsilon_\alpha(\Gbf)
	\times\Gbf,
	\boldsymbol e_i
	\Bigr).
	\label{eq:velocity-galerkin}
\end{align}

Since the bases are orthonormal in $L^2(\Omega)$ and
$\mathbb V^1(\Omega)$, respectively, equations
\eqref{eq:scalar-galerkin}-\eqref{eq:velocity-galerkin}
constitute a finite-dimensional system of ordinary differential
equations for
\[
\bigl(
c_1^k,\ldots,c_k^k,
a_1^k,\ldots,a_k^k
\bigr).
\]
Standard finite-dimensional existence theory gives a local
Galerkin solution. The uniform energy estimates subsequently
allow the solution to be extended to the required time interval
and provide the compactness needed to pass to the limit
$k\to\infty$.

\begin{remark}[High-order calculations and passage to the limit]
	We record how the higher-order operations used below are justified. First,
	all differentiations, curl operations, and tests by generalized-vorticity
	variables are performed on the finite-dimensional Galerkin system. The
	velocity basis lies in $\wiWb$; hence the approximate velocities satisfy the
	homogeneous Navier-slip condition and all integrations by parts are valid.
	
	For the generalized-vorticity estimate, set
	\[
	\omega_k=\curl\Upsilon_\alpha(\v_k),
	\qquad
	\omega_{\Gbf}=\curl\Upsilon_\alpha(\Gbf).
	\]
	After taking curl, the terms
	$(\v_k\cdot\nabla\omega_k,\omega_k)$ and
	$(\Gbf\cdot\nabla\omega_k,\omega_k)$ vanish because the advecting fields
	are divergence free and have zero normal trace. The remaining terms are
	well defined and uniformly integrable because
	\[
	\Gbf\in L^2(0,T;H^4(\O)),
	\qquad
	\partial_t\Gbf\in L^2(0,T;H^3(\O)),
	\]
	and the Galerkin estimates give
	$\v_k$ bounded in $L^\infty(0,T_\xi;H^3(\O))$.
	
	When the momentum equation is tested by
	$\mathbb P\Upsilon_\alpha(\v_k)$, Lemma~\ref{lem:1.6} is applied only to
	the homogeneous field $\v_k$. The nonhomogeneous part is always kept in the
	fixed lifting $\Gbf$ and is estimated by
	Lemma~\ref{Lifting:estimates}. No boundary regularity is inferred directly
	from the nonhomogeneous physical velocity.
	
	After extracting a subsequence, the uniform bounds and the time-derivative
	estimates give
	\[
	\v_k\to\v\quad\text{strongly in }L^2(0,T_\xi;H^2(\O)),
	\qquad
	\curl\Upsilon_\alpha(\v_k)\rightharpoonup
	\curl\Upsilon_\alpha(\v)\quad\text{in }L^2.
	\]
	Thus every nonlinear term is a product of one strongly convergent factor
	and one weakly convergent factor. For example,
	\[
	\curl\Upsilon_\alpha(\v_k)\times(\v_k+\Gbf)
	\rightharpoonup
	\curl\Upsilon_\alpha(\v)\times(\v+\Gbf)
	\]
	in the distributional formulation. The normal and Navier-slip conditions
	pass to the limit by continuity of the trace maps from $H^2(\O)$.
	
	Whenever the scalar equation is differentiated twice, the calculation is
	first carried out for smooth initial data. The corresponding estimates are
	uniform under approximation of the initial datum; density and weak lower
	semicontinuity then yield the asserted estimate for the stated data class.
\end{remark}

\section{Well-posedness of system \eqref{state:eq:state}}\setcounter{equation}{0}\label{Sec4}
In this section, we establish the existence and uniqueness of strong solutions to the state system \eqref{state:eq:state}. To this end, we first establish the well-posedness of the following system subject to homogeneous boundary conditions:
\begin{equation}\label{state:eq:state:lift}
\left\{
\begin{aligned}
\partial_t\theta+(\v+\Gbf) \cdot\nabla\theta&=0,
 && \hspace{-15mm}\text{in }(0,T)\times\O,\\
   \partial_t\Upsilon_\alpha(\v)  - \nu\Delta\v
  & +\curl\Upsilon_\alpha(\v)\times\v + \curl\Upsilon_\alpha(\v)\times\Gbf  + \curl\Upsilon_\alpha(\Gbf) \times\v 
 +\nabla  q  \\
 &= \xi \theta e_2 - \partial_t\Upsilon_\alpha(\Gbf)  + \nu\Delta\Gbf - \curl\Upsilon_\alpha(\Gbf) \times\Gbf,
 && \hspace{-15mm} \text{in }(0,T)\times\O,\\
 \diver\v &=0,
 && \hspace{-15mm}\text{in }(0,T)\times\O,\\
\theta(0) = \theta_0, \;\; \v(0) & = \u_0 - \Gbf(0), && \hspace{-15mm}\text{in }\O,\\
     \v\cdot\n & =0, \;\;\; \text{ and } \;\;\; [2\nu \n \cdot \Db(\v)  + \beta \v]\cdot\tau  = 0, && \hspace{-15mm}\text{on }(0,T)\times\Gamma,
\end{aligned}
\right.
\end{equation}
where $\Gbf=\cR\g$.  
By \eqref{eq:time-lifting-bound},
\begin{equation}\label{eq:uniform-lifting}
	\|\Gbf\|_{L^2(0,T;H^4(\O))}
	+\|\partial_t\Gbf\|_{L^2(0,T;H^3(\O))}
	\leq C_RM,
\end{equation}
uniformly over $\Uad$. Moreover, $\Gbf(0)=\mathcal R\g_{\rm in}$ is independent of the particular admissible control. Hence $\v_0:=\u_0-\mathcal R\g_{\rm in}\in\wiWb$ is fixed throughout the optimization problem.

\begin{definition}\label{def:WS:state}
Let $T>0$ be given. Let $(\theta_0, \u_0)\in H^{\xi}(\Omega)\times H^3(\O)$ such that $\u_0$ satisfies $\eqref{state:eq:state}_{3}$ and \eqref{state:eq:boundary} (at $t=0$), $\Gbf\in L^2(0,T;H^4(\O))$ and $\partial_t\Gbf\in L^2(0,T;H^3(\O))$ satisfying system \eqref{Stokes}, where $\xi=0$ gives the passive scalar and $\xi=1$ gives the two-dimensional Boussinesq-type active scalar.  Define
\begin{equation}\label{eqn:time}
	T_\xi:=
	\begin{cases}
		T, & \xi=0,\\
		T^*\bigl(\|\theta_0\|_{H^1(\O)},
		\|\u_0-\mathcal R\g_{\rm in}\|_{\wiWb},M\bigr)\wedge T,
		& \xi=1,
	\end{cases}
\end{equation}
where $M>0$ is the constant appearing in \eqref{eqn:admissible:set} and $T^{\ast}>0$ is the time appearing in \eqref{v:theta:estimate} below. 
  A pair  $(\theta, \v) \in L^{\infty}(0,T_{\xi}; H^{\xi}) \times L^{\infty}(0,T_{\xi}; \wiWb)$  with $(\partial_t\theta, \partial_t \v) \in L^{\infty}(0,T_{\xi};(H^{1-\xi}(\O))^{\prime}) \times L^{2}(0,T_{\xi};\Vb^1(\O)),$ is called a \emph{solution} to system \eqref{state:eq:state:lift}, if $(\theta(0),\v(0))=(\theta_0,\u_0 - \Gbf(0))$ in $H^{\xi}(\Omega)\times \wiWb$, 
  \begin{itemize}
      \item for any $\boldsymbol{\phi}\in \Vb^1(\O),$ 
		\begin{align}\label{Weak:state:v}
		& \left(\partial_t\v, \boldsymbol{\phi}\right)_{\Vb^1(\O)} + 2\nu(\Db(\v), \Db(\boldsymbol{\phi}))  + \beta \int_{\Gamma}\v\cdot\boldsymbol{\phi} dS + b(\boldsymbol{\phi}, \v, \Upsilon_\alpha (\v)) - b(\v, \boldsymbol{\phi}, \Upsilon_\alpha(\v) ) \nonumber\\
        & + b(\boldsymbol{\phi}, \Gbf, \Upsilon_\alpha(\v)) - b(\Gbf, \boldsymbol{\phi}, \Upsilon_\alpha(\v)) + b(\boldsymbol{\phi}, \v, \Upsilon_\alpha(\Gbf) ) - b(\v, \boldsymbol{\phi}, \Upsilon_\alpha(\Gbf)) \nonumber \\
        & = (\xi \theta e_2 - \partial_t \Upsilon_\alpha(\Gbf) + \nu\Delta\Gbf - \curl\Upsilon_\alpha(\Gbf)\times\Gbf, \boldsymbol{\phi}),
		\end{align}
		for a.e. $t\in[0,T_{\xi}]$;
        \item  for any $\psi\in H^{1-\xi}(\O)$
        \begin{align}\label{Weak:state:theta}
		& \left\langle\partial_t\theta, \psi\right\rangle_{(H^{1-\xi}(\O))^{\prime}\times H^{1-\xi}(\O)} + \langle\diver[(\v+\Gbf)\theta],\psi\rangle_{(H^{1-\xi}(\O))^{\prime}\times H^{1-\xi}(\O)} =0,
		\end{align}
        for a.e. $t\in[0,T_{\xi}].$
  \end{itemize}
\end{definition}

	\begin{remark}
	All constants entering the active-state estimate are uniform for
	$\g\in\Uad$ by \eqref{eq:uniform-lifting}; consequently, $T_\xi$ is
	independent of the particular admissible control.
\end{remark}
\begin{remark}
	Here and below, $H^\xi(\O)$ means $L^2(\O)$ when $\xi=0$ and $H^1(\O)$ when $\xi=1$.
\end{remark}

\begin{theorem}\label{Wellposednes:state:lift}
  Assume that $T_{\xi}>0$ as in \eqref{eqn:time}. Let $(\theta_0, \u_0)\in H^{\xi}(\Omega)\times \Vb^3(\O)$ such that $\u_0$ satisfies $\eqref{state:eq:state}_{3}$, $\Gbf\in L^2(0,T;H^4(\O))$ and $\partial_t\Gbf\in L^2(0,T;H^3(\O))$ satisfying system \eqref{state:eq:state:lift}, where $\xi=0$ gives the passive scalar and $\xi=1$ gives the two-dimensional Boussinesq-type active scalar.  Then, there exists a unique solution pair $(\theta, \v)$ to system \eqref{state:eq:state:lift} in the sense of Definition \ref{def:WS:state}.   
  
  In addition, assuming \(\theta_0 \in H^1(\O)\) for \(\xi = 0\), the solution pair retains the same regularity as in the case \(\xi = 1\) but globally in time.
\end{theorem}
\begin{proof}[Proof of Theorem \ref{Wellposednes:state:lift}]
	We use the Galerkin system \eqref{eq:scalar-galerkin}-\eqref{eq:velocity-galerkin} constructed above. The following estimates are first obtained for $(\theta_k,\v_k)$; their constants are independent of $k$. Standard ODE theory provides a local approximate solution, and the uniform bounds extend it to $[0,T_\xi]$.
	
    \vskip 2mm
    \noindent
    \textbf{Transport equation $\eqref{state:eq:state:lift}_3$:} 
    \smallskip
    \noindent
    \textit{For $\theta_0\in L^2(\O)$.} Taking inner product with $\theta$ in $\eqref{state:eq:state:lift}_3$, integrating by parts and using the conditions $\diver(\v+\Gbf)=0$ and $(\v+\Gbf)\cdot\n=0$, we get
    \begin{align}
        \frac{1}{2}\frac{d}{dt}\|\theta(t)\|_{L^2(\O)}^2 =0 
    \end{align}
    which implies
    \begin{align}
        \|\theta\|_{L^{\infty}(0,T;L^2(\O))} = \|\theta_0\|_{L^2(\O)}<\infty.
    \end{align}
    
    \smallskip
    \noindent
    \textit{For $\theta_0\in H^1(\O)$.} Applying $\nabla$ to the equation $\eqref{state:eq:state:lift}_3$, taking the inner product with $\nabla\theta$ and using the conditions $\diver(\v+\Gbf)=0$ and $(\v+\Gbf)\cdot\n=0$ gives
    \begin{align}\label{theta:H1:estimate}
    \frac{1}{2}\frac{d}{dt}\|\nabla\theta(t)\|_{L^2(\O)}^2 \leq \|\nabla(\v+\Gbf)\|_{L^{\infty}(\O)}  \|\nabla\theta(t)\|_{L^2(\O)}^2.
    \end{align}
    \vskip 2mm
    \noindent
    \textbf{Second grade fluid equation $\eqref{state:eq:state:lift}_1$:} Taking the inner product of $\eqref{state:eq:state:lift}_1$ with $\v$, integrating by parts, and utilizing $\eqref{trilinear1}_2$, we obtain 
    \begin{align}\label{EEv:1}
     \frac{1}{2}\frac{d}{dt}\|\v\|^2_{\Vb^1(\O)} 
     & 
    = -2\nu \|\Db(\v)\|^2_{L^2(\O)} - \nu\beta\|\v\|^2_{L^2(\Gamma)}  + \underbrace{(\curl\Upsilon_\alpha(\v) \times \v, \v)}_{=0, \text{ by } \eqref{trilinear1}_2} + (\curl\Upsilon_\alpha(\v) \times\Gbf, \v)  
    \nonumber\\ 
    & \quad + \underbrace{(\curl\Upsilon_\alpha(\Gbf) \times\v, \v)}_{=0, \text{ by } \eqref{trilinear1}_2} - \underbrace{(\curl\Upsilon_\alpha(\Gbf) \times\Gbf, \v)}_{=b(\v, \Gbf, \Upsilon_\alpha(\Gbf) ) - b(\Gbf, \v, \Upsilon_\alpha(\Gbf)).}
    + (\xi \theta e_2 - \partial_t \Upsilon_\alpha(\Gbf) + \nu\Delta\Gbf, \v).
    \end{align}
    Let us now estimate each term on the right hand side of \eqref{EEv:1} using $\eqref{curl:estimates}_3$, \eqref{Korn:equivalent}, H\"older's and Young's inequalities as follows:
    \begin{align}
    |(\curl\Upsilon_\alpha(\v) \times\Gbf, \v)| & \leq C \|\v\|_{H^1(\O)}^2\|\Gbf\|_{H^3(\O)}\leq C \|\v\|_{\Vb^1(\O)}^2[1+\|\Gbf\|_{H^3(\O)}^2], \label{EEv:2}\\
        |b(\v, \Gbf, \Upsilon_\alpha(\Gbf) ) - b(\Gbf, \v, \Upsilon_\alpha(\Gbf))| & \leq [\|\v\|_{L^4(\O)}\|\nabla\Gbf\|_{L^4(\O)} + \|\Gbf\|_{L^{\infty}(\O)} \|\nabla\v\|_{L^2(\O)} ] \|\Upsilon_\alpha(\Gbf)\|_{L^2(\O)}
        \nonumber\\ & \leq C \|\v\|_{\Vb^1(\O)}\|\Gbf\|_{H^2(\O)}^2 
        \nonumber\\ & \leq C \|\v\|^2_{\Vb^1(\O)}\|\Gbf\|_{H^2(\O)}^2 + C \|\Gbf\|_{H^2(\O)}^2, \label{EEv:3}\\
        |(\xi \theta e_2 - \partial_t \Upsilon_\alpha(\Gbf) + \nu\Delta\Gbf, \v)|& \leq  \xi \|\theta\|_{L^2(\O)}^2 +  \|\partial_t \Gbf\|_{H^2(\O)}^2 + \|\Gbf\|_{H^2(\O)}^2 + C \|\v\|^2_{\Vb^1(\O)}.\label{EEv:4}
    \end{align}
   Combining \eqref{EEv:1}-\eqref{EEv:4}, we obtain
\begin{align}\label{H1:estimate}
    & \frac{1}{2}\frac{d}{dt}\|\v(t)\|^2_{\Vb^1(\O)} + 2\nu \|\Db(\v(t))\|^2_{L^2(\O)} + \nu\beta\|\v(t)\|^2_{L^2(\Gamma)}
    \nonumber\\ & \leq \xi \|\theta(t)\|_{L^2(\O)}^2 +  \|\partial_t \Gbf(t)\|_{H^2(\O)}^2 + \|\Gbf(t)\|_{H^2(\O)}^2 + C [1+\|\Gbf(t)\|_{H^3(\O)}^2] \|\v(t)\|^2_{\Vb^1(\O)}.
\end{align}
 for a.e. $t\in[0,T]$.
 
 Next, applying    `$\curl$' to the equation $\eqref{state:eq:state:lift}_1$, using the vector identity 
 \begin{align}\label{identity:curl:curl}
     \curl(\curl(\z_1)\times\z_2)= \z_2\cdot \nabla(\curl\z_1)
 \end{align}
 and taking the inner product with $\curl(\v-\alpha\Delta\v)$, we achieve
\begin{align}\label{H3:estimate:0}
    & \frac{1}{2}\frac{d}{dt}\|\curl\Upsilon_\alpha(\v)\|^2_{L^2(\O)}  + \frac{\nu}{\alpha}\|\curl\Upsilon_\alpha(\v)\|^2_{L^2(\O)}  
    \nonumber\\
    & = \frac{\nu}{\alpha}\left(\curl\v, \curl\Upsilon_\alpha(\v)\right) - \underbrace{(\v\cdot\nabla(\curl\Upsilon_\alpha(\v)), \curl\Upsilon_\alpha(\v)) }_{=0, \text{ using $\v\cdot\n|_{\Gamma}=0$ and $\diver\v=0$}.}
     - \underbrace{(\Gbf\cdot\nabla(\curl\Upsilon_\alpha(\v)), \curl\Upsilon_\alpha(\v))}_{=0, \text{ using $\Gbf\cdot\n|_{\Gamma}=0$ and $\diver\Gbf=0$}.}
     \nonumber\\ 
     & \quad  - (\v\cdot\nabla(\curl\Upsilon_\alpha(\Gbf)), \curl\Upsilon_\alpha(\v)) 
    - (\Gbf\cdot\nabla(\curl\Upsilon_\alpha(\Gbf)), \curl\Upsilon_\alpha(\v)) 
    \nonumber\\ 
    & \quad  + (\curl[\xi \theta e_2 - \partial_t \Upsilon_\alpha(\Gbf) + \nu\Delta\Gbf] , \curl\Upsilon_\alpha(\v))
    \nonumber\\
    & \leq C [\|\v\|_{\Vb^1(\O)}  +  \|\v\|_{L^{\infty}(\O)}\|\Gbf\|_{H^4(\O)} + \|\Gbf\|^2_{H^4(\O)}  + \|\partial_t\Gbf\|_{H^3(\O)} + \|\Gbf\|_{H^3(\O)} ]  \|\curl\Upsilon_\alpha(\v)\|_{L^2(\O)}
    \nonumber\\ & \quad + \xi \|\nabla\theta\|_{L^2(\O)}\|\curl\Upsilon_\alpha(\v)\|_{L^2(\O)}
    \nonumber\\ & \leq \xi \|\nabla\theta\|_{L^2(\O)}^2 + C [\|\v\|^2_{\Vb^1(\O)}  +  \|\v\|_{H^{3}(\O)}^2 + \|\Gbf\|^2_{H^4(\O)}  + \|\partial_t\Gbf\|^2_{H^3(\O)}  ]
    \nonumber\\ & \quad + C [1+ \|\Gbf\|^2_{H^4(\O)}] \|\curl\Upsilon_\alpha(\v)\|_{L^2(\O)}^2,
\end{align}
where we have used Sobolev, H\"older's and Young's inequalities in the last two steps. Now, in view of \eqref{wiWb:inner:product} and the equivalence between the norms $\|\cdot\|_{H^3(\O)}$ and $\|\cdot\|_{\wiWb}$, combining \eqref{H1:estimate} and \eqref{H3:estimate:0} yields
\begin{align}\label{H3:estimate}
     \frac{1}{2}\frac{d}{dt}\|\v\|^2_{\wiWb}  
     & \leq   C [ \|\Gbf\|^2_{H^4(\O)}   + \|\partial_t\Gbf\|^2_{H^3(\O)} ] 
     + [1+ \|\Gbf\|^2_{H^4(\O)}] \|\v\|_{\wiWb}^2
       + \xi[\|\nabla\theta\|_{L^2(\O)}^2 + \|\theta_0\|_{L^2(\O)}^2].
\end{align}
\vskip 2mm
\noindent
\textit{\underline{Case I: $\xi=0$}}. We apply the Gronwall Lemma to \eqref{H3:estimate} (for $\xi=0$) to arrive at
\begin{align}\label{v:H3:estimate}
     \|\v(t)\|^2_{\wiWb} 
    & \leq \left[\|\v_0\|^2_{\wiWb} +  C \int_{0}^T[\|\Gbf(s)\|^2_{H^4(\O)}   + \|\partial_t\Gbf(s)\|^2_{H^3(\O)}]ds \right]  e^{C \int_{0}^T[1 + \|\Gbf(s)\|^2_{H^4(\O)}]ds}
    \nonumber\\
    & \leq \left[\|\u_0\|^2_{H^3(\O)} +\|\g(0)\|^2_{H^{\frac32}(\Gamma)} +  C \int_{0}^T[\|\g(s)\|_{H^{\frac52}(\Gamma)}^2   + \|\partial_t\g(s)\|_{H^{\frac32}(\Gamma)}^2]ds \right] 
    \nonumber\\
    & \quad \times \exp\left\{C \int_{0}^T[1 + \|\g(s)\|_{H^{\frac52}(\Gamma)}^2]ds\right\}
    \nonumber\\ & < +\infty.
\end{align}
It is clear that for $\xi = 0$, no regularity assumptions on $\theta$ are required. However, for the case $\xi = 1$, the situation is entirely different, as it inherently requires $\theta_0 \in H^1(\Omega)$ (see \textbf{Case II} below). 

Now suppose, for $\xi=0$ and $\theta_0\in H^1(\O)$, then applying Gronwall's inequality, Sobolev inequality and \eqref{H3:estimate} to \eqref{theta:H1:estimate} obtain
\begin{align}\label{theeta:H1:estimate}
  \sup_{t\in[0,T]}  \|\nabla\theta(t)\|^2_{L^2(\O)}&  \leq \|\nabla\theta_0\|_{L^2(\O)}e^{\int_{0}^T\|\nabla(\v(t)+\Gbf(t))\|_{\infty} dt }
  \nonumber\\
  & \leq \|\nabla\theta_0\|_{L^2(\O)}e^{\int_{0}^T\|\v(t)+\Gbf(t))\|_{H^3(\O)} dt }
  \nonumber\\
  & \leq \|\nabla\theta_0\|_{L^2(\O)}\exp\left\{\int_{0}^T\|\v(t)\|_{H^3(\O)}  dt + \int_{0}^T\|\g(t)\|_{H^{\frac32}(\Gamma)} dt \right\} < + \infty.
\end{align}

\vskip 2mm
\noindent
\textit{\underline{Case II: $\xi=1$}}. Let us now prove that when $\xi = 1$ and $(\theta_0, \v_0) \in H^{1}(\Omega) \times \wiWb$, there exists a time $T^*>0$ such that the solution pair satisfies the regularity $(\theta, \v) \in L^{\infty}(0,T^*; H^{1}(\Omega)) \times L^{\infty}(0,T^*; \wiWb)$. 

A combination of \eqref{theta:H1:estimate} and \eqref{H3:estimate} gives
\begin{align}\label{H1+H3:estimate}
     \frac{1}{2}\frac{d}{dt}\left[\|\v\|^2_{\wiWb}+\|\nabla\theta\|_{L^2(\O)}^2\right]  
     & \leq   C [ \|\Gbf\|^2_{H^4(\O)}   + \|\partial_t\Gbf\|^2_{H^3(\O)} ] 
     + [1+ \|\Gbf\|^2_{H^4(\O)}] \|\v\|_{H^{3}(\O)}^2
     \nonumber\\ & \quad + [\|\nabla\theta\|_{L^2(\O)}^2 + \|\theta_0\|_{L^2(\O)}^2] + \|\nabla(\v+\Gbf)\|_{L^{\infty}(\O)}  \|\nabla\theta\|_{L^2(\O)}^2
     \nonumber\\ 
     &  \leq C[\|\theta_0\|_{L^2(\O)}^2 + \|\Gbf\|^2_{H^4(\O)}   + \|\partial_t\Gbf\|^2_{H^3(\O)} ] 
     \nonumber\\ & \quad + C[1+ \|\Gbf\|^2_{H^4(\O)}] [\|\nabla\theta\|_{L^2(\O)}^2 + \|\v\|_{H^{3}(\O)}^2]
      +  C \|\v\|_{H^{3}(\O)}  \|\nabla\theta\|_{L^2(\O)}^2
      \nonumber\\ 
     &  \leq C[\|\theta_0\|_{L^2(\O)}^2 + \|\Gbf\|^2_{H^4(\O)}   + \|\partial_t\Gbf\|^2_{H^3(\O)} ] 
     \nonumber\\ & \quad + C[1+ \|\Gbf\|^2_{H^4(\O)}] [ \|\v\|_{\wiWb}^2 + \|\nabla\theta\|_{L^2(\O)}^2 ]^{\frac32}. 
\end{align}
This implies that there exists a time $T^*>0$ depending on $\|\theta_0\|_{H^1(\O)}$, $\|\v_0\|_{\wiWb}$ and $M$ such that 
\begin{align}\label{v:theta:estimate}
\sup_{t\in[0,T^*]}\left[\|\v(t)\|^2_{\wiWb}+\|\nabla\theta(t)\|_{L^2(\O)}^2\right] & \leq 4 (1+ \|\v_0\|^2_{\wiWb}+\|\nabla\theta_0\|_{L^2(\O)}^2)
    \nonumber\\
    & \leq 4 (1+ \|\u_0\|^2_{\wiWb} +\|\g(0)\|^2_{H^{\frac32}(\Gamma)}+\|\nabla\theta_0\|_{L^2(\O)}^2).
\end{align}
 This concludes the proof of the regularity of the pair $(\theta,\v)$. 
\vskip 2mm
\noindent
\textbf{Weak time derivative estimates and initial data:} For $T_{\xi}>0$ as in \eqref{eqn:time}, we obtained that $(\theta, \v) \in L^{\infty}(0,T_{\xi}; H^{\xi}) \times L^{\infty}(0,T_{\xi}; \wiWb)$. Taking the inner product of equation \eqref{state:eq:state:lift} with $\partial_t\v$ and estimating the resulting terms in the same manner as in the $H^3$-estimate, we obtain $\partial_t\v \in L^{2}(0,T_{\xi};\Vb^1(\O)).$

For the transport equation, we consider
\begin{align}\label{theta:derivative}
   \sup_{\psi\in H^{1-\xi}(\O)} |\left(\partial_t\theta, \psi\right)_{(H^{1-\xi}(\O))^{\prime}\times H^{1-\xi}(\O)}| & = \sup_{\psi\in H^{1-\xi}(\O)} |(\diver[(\v+\Gbf)\theta], \psi)|
   \nonumber\\ 
   & \leq \|\diver [(\v+\Gbf)\theta]\|_{(H^{1-\xi}(\O))^{\prime}}
   \nonumber\\
   & \leq \begin{cases}
       \|(\v+\Gbf)\theta\|_{L^2(\O)}, & \text{ for } \xi=0\\
       \|(\v+\Gbf)\nabla\theta\|_{L^{2}(\O)}, & \text{ for } \xi=1.
   \end{cases}
   \nonumber\\
   & \leq \begin{cases}
       \|\v+\Gbf\|_{L^\infty(\O)}\|\theta\|_{L^2(\O)}, & \text{ for } \xi=0\\
       \|\v+\Gbf\|_{L^\infty(\O)}\|\nabla\theta\|_{L^{2}(\O)}, & \text{ for } \xi=1,
   \end{cases}
\end{align}
which gives that $\partial_t\theta \in L^{\infty}(0,T_{\xi};(H^{1-\xi}(\O))^{\prime})$. Recall that for $\xi=0$ and $\theta_0\in H^1(\O)$, it holds that $(\theta, \v) \in L^{\infty}(0,T; H^{1}) \times L^{\infty}(0,T; \wiWb)$. Consequently from \eqref{theta:derivative}, we obtain $\partial_t\theta \in L^{\infty}(0,T;L^{2}(\O))$. Using the weak time derivative estimates, one can verify that $(\theta(0),\v(0))=(\theta_0,\u_0 - \Gbf(0))$ in $H^{\xi}(\Omega)\times \wiWb$. 

The above estimates for $(\theta,\v)$ and $(\partial_t\theta,\partial_t\v)$ allow us to pass to the limit in the finite-dimensional Galerkin approximation scheme; for example see the proof of Theorem \ref{existence:control} below. This completes the proof of the existence of a solution pair $(\theta,\v)$ to the system \eqref{state:eq:state:lift}.

\vskip 2mm
\noindent
\textbf{Uniqueness:} Let $(\theta_1, \v_1), (\theta_2, \v_2) \in L^{\infty}(0,T_{\xi}; H^{\xi}(\O)) \times L^{\infty}(0,T_{\xi}; \wiWb)$ be two solutions of system \eqref{state:eq:state:lift}. Therefore, the pair  $(\tilde{\theta}:=\theta_1-\theta_2, \tilde{\v}:=\v_1- \v_2)$ satisfies:
\begin{equation}\label{state:eq:state:lift:difference}
\left\{
\begin{aligned}
	 \partial_t\tilde{\theta}+ \tilde{\v} \cdot\nabla\theta_1 + (\v_2+\Gbf) \cdot\nabla\tilde{\theta}&=0,
	&& \hspace{-15mm}\text{in }(0,T_\xi)\times\O,\\
   \partial_t\Upsilon_\alpha(\tilde{\v})  - \nu\Delta\tilde{\v}
  +\curl\Upsilon_\alpha(\tilde{\v}) \times\v_1 & + \curl\Upsilon_\alpha(\v_2) \times\tilde{\v}  + \curl\Upsilon_\alpha(\tilde{\v}) \times\Gbf \\ + \curl\Upsilon_\alpha(\Gbf) \times\tilde{\v} 
 & = -\nabla {(q_1-q_2)}
 + \xi \tilde{\theta} e_2 ,
 && \hspace{-15mm} \text{in }(0,T_\xi)\times\O,\\
 \diver\tilde{\v} &=0,
 && \hspace{-15mm}\text{in }(0,T_\xi)\times\O,\\
 \tilde{\theta}(0)  = 0, \;\; \tilde{\v}(0) & = \boldsymbol{0}, && \hspace{-15mm}\text{in }\O,\\
     \v\cdot\n   =0 \;\;\; \text{ and } \;\;\;  [2\nu \n \cdot \Db(\v)  & + \beta \v]\cdot\tau   = 0, && \hspace{-15mm}\text{on } (0,T_\xi)\times \Gamma,
\end{aligned}
\right.
\end{equation}
where $q_1$ and $q_2$ are the pressure scalar corresponding to  $(\theta_1, \v_1) $ and $ (\theta_2, \v_2)$, respectively. 
\vskip 2mm
\noindent
\textit{\underline{Case I: $\xi=0$}}.
Taking the inner product with $\tilde{\v}$ to the equation $\eqref{state:eq:state:lift:difference}_1$, we write
\begin{align}\label{Uni:1}
   & \frac{1}{2}\frac{d}{dt}\|\tilde{\v}\|^2_{\Vb^1(\O)}  + 2\nu \|\Db(\tilde{\v})\|^2_{L^2(\O)} + \nu\beta\|\tilde{\v}\|^2_{L^2(\Gamma)}
   \nonumber\\ 
   & =  -  (\curl\Upsilon_\alpha(\tilde{\v}) \times(\v_1+\Gbf), \tilde{\v}) 
      - \underbrace{(\curl\Upsilon_\alpha(\v_2+\Gbf) \times\tilde{\v}, \tilde{\v})}_{=0, \text{ by } \eqref{trilinear1}_2.}  + (\xi \tilde{\theta} e_2, \tilde{\v})
    \nonumber\\
    & \underbrace{\leq}_{\eqref{curl:estimates}_3} C \|\v_2+\Gbf\|_{H^3(\O)}\|\tilde{\v}\|^2_{\Vb^1(\O)}.
\end{align}
An application of the Gronwall lemma to \eqref{Uni:1} and the fact that $\v_2+\Gbf\in L^{\infty}(0,T_{\xi};\wiWb)$ provide that $\v_1=\v_2$. Taking the inner product with $\tilde{\theta}$ to the equation $\eqref{state:eq:state:lift:difference}_3$, and using $\v_1=\v_2$, $\diver(\v_2+\Gbf)=0$ and $(\v_2+\Gbf)\cdot\n=0$, we have
\begin{align}
        \frac{1}{2}\frac{d}{dt}\|\tilde{\theta}\|_{L^2(\O)}^2 &  = - \int_{\O} [\tilde{\v} \cdot\nabla\theta_1 + (\v_2+\Gbf) \cdot\nabla\tilde{\theta}]\tilde{\theta} dx =0,
    \end{align}
    which immediately gives that $\theta_1=\theta_2$.
    
    \vskip 2mm
\noindent
\textit{\underline{Case I: $\xi=1$}}.  A calculation similar to \eqref{Uni:1} gives
\begin{align}\label{Uni:2}
    \frac{1}{2}\frac{d}{dt}\|\tilde{\v}\|^2_{\Vb^1(\O)} &
    \leq C [1+\|\v_2+\Gbf\|_{H^3(\O)}] \|\tilde{\v}\|^2_{\Vb^1(\O)} +  \|\tilde{\theta}\|^2_{L^2(\O)}.
\end{align}

Let us recall from \cite[Lemma 4]{BusuiocIftimie2006} that the following inequality holds:
\begin{align}\label{infty:inequality}
    \|f\|_{L^{\infty}(\O)} \leq \frac{C}{\sqrt{\eps}} \|f\|_{H^1(\O)}^{1-\eps}\|f\|^{\eps}_{H^2(\O)}, \;\;\; \text{ for all } \eps\in(0,1],
\end{align}
for any function $f\in H^2(\O)$. Taking the inner product with $\tilde{\theta}$ to the equation $\eqref{state:eq:state:lift:difference}_3$, and using $\diver(\v_2+\Gbf)=0$ and $(\v_2+\Gbf)\cdot\n=0$, we get
\begin{align}\label{Uni:3}
        \frac{1}{2}\frac{d}{dt}\|\tilde{\theta}\|_{L^2(\O)}^2 &  = - \int_{\O} (\tilde{\v} \cdot\nabla\theta_1 )\tilde{\theta} dx  
        \nonumber\\
        & \leq \|\tilde{\v}\|_{L^{\infty}(\O)}  \|\nabla\theta_1\|_{L^2(\O)}\|\tilde{\theta}\|_{L^2(\O)}
        \nonumber\\
        & \leq \frac{C}{\sqrt{\eps}} \|\tilde{\v}\|^{1-\eps}_{\Vb^{1}(\O)}  \|\tilde{\v}\|^{\eps}_{H^{2}(\O)}\|\nabla\theta_1\|_{L^2(\O)}\|\tilde{\theta}\|_{L^2(\O)},
\end{align}
where we have used \eqref{infty:inequality} in the last inequality.

Let us denote $\mathbf{U}(t):=\|\tilde{\v}(t)\|^2_{\Vb^1(\O)}+\|\tilde{\theta}(t)\|_{L^2(\O)}^2$. Now, since $(\theta_1, \v_1), (\theta_2, \v_2) \in L^{\infty}(0,T_{\xi}; H^{\xi}(\O)) \times L^{\infty}(0,T_{\xi}; \wiWb)$, there exists a constant $\widebar{K}>0$ (independent of $\eps$) such that by adding \eqref{Uni:2} and \eqref{Uni:3}, we obtain
\begin{align}\label{}
    \frac{d}{dt}\mathbf{U}(t) 
    \leq  \frac{\widebar{K}}{\sqrt{\eps}} [\mathbf{U}(t) +  \mathbf{U}(t)^{1-\frac{\eps}{2}}].
\end{align}
This implies
\begin{align}\label{}
    \frac{d}{dt}[\mathbf{U}(t)]^{\frac{\eps}{2}} &
    \leq  \frac{\widebar{K}\sqrt{\eps}}{2} [\mathbf{U}(t)]^{\frac{\eps}{2}} + \frac{\widebar{K}\sqrt{\eps}}{2},
\end{align}
and consequently 
\begin{align*}
    \mathbf{U}(t) & \leq (e^{\frac{\widebar{K}\sqrt{\eps}}{2}t} -1)^{\frac{2}{\eps}} \leq (e^{\frac{\widebar{K}\sqrt{\eps}}{2}T_{\xi}} -1)^{\frac{2}{\eps}},
\end{align*}
for all $t\in[0,T_{\xi}]$.
If $\mathrm{\mathbf{U}}(t)\geq \sigma$ for some $\sigma>0$. Then, one can always find $\eps\in (0,1)$ close to $0$ such that
\begin{align}
	\mathrm{\mathbf{U}}(t)  \leq \frac{\sigma}{2}, \;\; \text{ for all } t\in [0,{T_{\xi}}],
\end{align}
which leads to a contradiction.
Hence $\mathrm{\textbf{U}}(t)=0$ for all $t\in[0,T_{\xi}]$, which gives the uniqueness. This completes the proof of the theorem.
\end{proof}

By Theorem \ref{Wellposednes:state:lift}, we immediately have the following result.
\begin{theorem}\label{Wellposednes:state}
    Assume that $T_{\xi}>0$ as in \eqref{eqn:time}. Let $(\theta_0, \u_0)\in H^{\xi}(\Omega)\times \Vb^3(\O)$ such that $\u_0$ satisfies $\eqref{state:eq:state}_{3}$, and $\g\in\Uad$, where $\xi=0$ gives the passive scalar and $\xi=1$ gives the two-dimensional Boussinesq-type active scalar. Then, there exists a unique solution pair  $(\theta, \u) \in L^{\infty}(0,T_{\xi}; H^{\xi}(\O)) \times L^{\infty}(0,T_{\xi}; H^3(\O))$  with $(\partial_t\theta, \partial_t \u) \in L^{\infty}(0,T_{\xi};(H^{1-\xi}(\O))^{\prime}) \times L^{2}(0,T_{\xi};\Vb^1(\O)),$ which satisfies the system \eqref{state:eq:state} in the following sense:
  \begin{itemize}
      \item for any $\boldsymbol{\phi}\in \Vb^1(\O),$ 
		\begin{align}\label{Weak:state:u}
		& \left(\partial_t\u, \boldsymbol{\phi}\right)_{\Vb^1(\O)} + 2\nu(\Db(\u), \Db(\boldsymbol{\phi}))  + \beta \int_{\Gamma}(\u\cdot\tau)(\boldsymbol{\phi}\cdot\tau) dS + b(\boldsymbol{\phi}, \u, \Upsilon_\alpha(\u)) - b(\u, \boldsymbol{\phi}, \Upsilon_\alpha(\u)) 
        \nonumber \\
        & = (\xi \theta e_2, \boldsymbol{\phi}) +  \int_{\Gamma}(\g\cdot\tau)(\boldsymbol{\phi}\cdot\tau)dS +  \frac{\alpha}{\nu}\int_{\Gamma}(\partial_t\g\cdot\tau)(\boldsymbol{\phi}\cdot\tau)dS,
		\end{align}
		for a.e. $t\in[0,T_{\xi}]$;
        \item  for any $\psi\in H^{1-\xi}(\O)$
        \begin{align}\label{Weak:state:theta:u}
		& \left\langle\partial_t\theta, \psi\right\rangle_{(H^{1-\xi}(\O))^{\prime}\times H^{1-\xi}(\O)} + \langle\diver[
        \u\theta],\psi\rangle_{(H^{1-\xi}(\O))^{\prime}\times H^{1-\xi}(\O)} =0,
		\end{align}
        for a.e. $t\in[0,T_{\xi}].$
        \item  the initial condition is satisfied: $$(\theta(0),\u(0))=(\theta_0,\u_0).$$ 
  \end{itemize}
  In addition, assuming \(\theta_0 \in H^1(\O)\) for \(\xi = 0\), the solution pair retains the same regularity as in the case \(\xi = 1\) but globally in time.
\end{theorem}

\section{Existence of optimal solution of \eqref{eqn:control:problem}}\setcounter{equation}{0}\label{sec:existence-optimal-control}
In this section, we establish our first main result, namely, the existence of an optimal control for the problem \eqref{eqn:control:problem}.

\begin{theorem}\label{existence:control}
    Consider the scalar field $\theta$ governed by the system \eqref{state:eq:state}. Under the assumption of Theorem \ref{Wellposednes:state}, there exists at least one optimal solution $\g^{\ast}\in \Uad$ to the optimal control problem \eqref{eqn:control:problem}.
\end{theorem}

\begin{proof}[Proof of Theorem \ref{existence:control}]
We first recall that $(\theta(\g),\u(\g))$ denotes the solution of system \eqref{state:eq:state}. In view of Theorem \ref{Wellposednes:state}, there exists a positive constant $C_M>0$ independent of $\g\in\Uad$ such that
\begin{align}
 \|\u(\g)\|_{L^\infty(0,T_\xi;H^3(\O))}
 +\|\partial_t\u(\g)\|_{L^2(0,T_\xi;\Vb^1(\O))}&\leq C_M,
 \label{eq:uniform-u}\\
 \|\theta(\g)\|_{L^\infty(0,T_\xi;H^\xi(\O))}
 +\|\partial_t\theta(\g)\|_{L^\infty(0,T_\xi;(H^{1-\xi}(\O))')}
 &\leq C_M.\label{eq:uniform-theta}
\end{align}

Choose any arbitrary $\g_0\in\Uad$.  The estimates
\eqref{eq:uniform-u}-\eqref{eq:uniform-theta} imply that every term in \eqref{control:eq:cost} is finite.
Consequently,
\[
 \inf_{\g\in\Uad}\J(\g)\le\J(\g_0)<\infty.
\]
On the other hand, for every $\g\in\Uad$,
\begin{align*}
 \J(\g)
 &\ge-\frac\zeta2\int_0^{T_\xi}
            \|\curl\u(t;\g)\|_{L^2}^2 dt \ge-\frac\zeta2T_\xi C_M^2.
\end{align*}
Therefore
\begin{equation}\label{eq:j-finite}
 -\infty<\J^*:=\inf_{\g\in\Uad}\J(\g)<\infty.
\end{equation}
By the definition of the infimum, there is a sequence $\{\gn\}\subset\Uad$
such that
	\begin{equation}\label{eq:minimizing-sequence}
	\J^*\leq\J(\g_n)\leq\J^*+\frac1n   \;\; \text{ and }
 \;\;	\J(\g_n)\to\J^*.
\end{equation}

The boundedness of $\Uad$ gives
\begin{equation}\label{eq:gn-bounded}
	\sup_n\|\g_n\|_{\Ucal}\leq M.
\end{equation}
Since $\Ucal$ is Hilbert, after passing to a subsequence,
\begin{equation}\label{eq:gn-weak}
	\g_n\rightharpoonup\g^*\qquad\text{in }\Ucal.
\end{equation}
The trace operator in \eqref{eq:control-trace} is bounded and linear;
therefore
\[
\g_n(0)\rightharpoonup\g^*(0)
\qquad\text{in }\Vb^2(\Gamma).
\]
Since $\g_n(0)=\g_{\rm in}$ for every $n$, it follows that
$\g^*(0)=\g_{\rm in}$. The norm is weakly lower semicontinuous, so
$\|\g^*\|_{\Ucal}\leq M$. Hence $\g^*\in\Uad$.

Set $\Gbf_n:=\mathcal R\g_n$ and $\Gbf^*:=\mathcal R\g^*.$
The bounded linearity of $\mathcal R$ and
Lemma~\ref{Lifting:estimates} imply
\begin{align}
	\Gbf_n&\rightharpoonup\Gbf^*
	&&\text{in }L^2(0,T_\xi;H^4(\O)),\label{eq:lift-weak}\\
	\partial_t\Gbf_n&\rightharpoonup\partial_t\Gbf^*
	&&\text{in }L^2(0,T_\xi;H^3(\O)).\label{eq:lift-time-weak}
\end{align}
Since
$H^4(\O)\hookrightarrow H^{4-\varepsilon}(\O)\hookrightarrow H^3(\O)$ with compact embedding $H^4(\O)\hookrightarrow H^{4-\varepsilon}(\O)$ for
$0<\varepsilon<1$, Aubin-Lions compactness lemma gives
\begin{equation}\label{eq:lift-strong}
	\Gbf_n\to\Gbf^*
	\qquad\text{in }L^2(0,T_\xi;H^{4-\varepsilon}(\O)).
\end{equation}
For completeness, the boundary controls themselves satisfy
\[
\g_n\to\g^*
\quad\text{in }L^2(0,T_\xi;H^{5/2-\varepsilon}(\Gamma))
\]
after a further subsequence, but only the interior convergence
\eqref{eq:lift-strong} is used in the nonlinear state equation.

Let $(\theta_n,\u_n)=(\theta(\g_n),\u(\g_n))$. The uniform state bounds
give, after extraction,
\begin{align}
	\u_n&\overset{*}{\rightharpoonup}\u^*
	&&\text{in }L^\infty(0,T_\xi;H^3(\O)),\label{eq:state-u-weak}\\
	\partial_t\u_n&\rightharpoonup\partial_t\u^*
	&&\text{in }L^2(0,T_\xi;\Vb^1(\O)),\label{eq:state-ut-weak}\\
	\u_n&\to\u^*
	&&\text{in }C([0,T_\xi];H^2(\O)),\label{eq:state-u-strong}\\
	\theta_n&\overset{*}{\rightharpoonup}\theta^*
	&&\text{in }L^\infty(0,T_\xi;H^\xi(\O)),\label{eq:state-theta-weak}\\
	\partial_t\theta_n&\overset{*}{\rightharpoonup}\partial_t\theta^*
	&&\text{in }L^\infty(0,T_\xi;(H^{1-\xi}(\O))'),
	\label{eq:state-thetat-weak}\\
		\theta_n & \to\theta^*
	&& \text{in }C([0,T_\xi];(H^1(\O))').\label{eq:state-theta-strong}
\end{align}
The two strong convergences in \eqref{eq:state-u-strong} and \eqref{eq:state-theta-strong} are justified by the compact embeddings $H^3(\O)\hookrightarrow H^2(\O)$ and $L^2(\O)\hookrightarrow(H^1(\O))'$, respectively, and \cite[Corollary 4]{Simon1986}.

Our next aim is to prove that $(\theta^*,\u^*)$ is the state corresponding to $\g^*$. Take $\psi\in H^1(\O)$ and $\eta\in C_c^\infty(0,T_\xi)$, and consider the following weak form of transport equation for $(\theta_n,\un)$:
\begin{equation}\label{eq:transport-n}
 -\int_0^{T_\xi}(\theta_n,\psi)\eta'(t) dt
 -\int_0^{T_\xi}(\un\theta_n,\nabla\psi)\eta(t) dt=0.
\end{equation}
By the convergence in \eqref{eq:state-theta-weak}, we get
\begin{equation}\label{eq:transport-time-limit}
 \int_0^{T_\xi}(\theta_n,\psi)\eta' dt
 \to 
 \int_0^{T_\xi}(\theta^*,\psi)\eta' dt.
\end{equation}
For the second term, we consider 
\begin{align}
   & \left| \int_0^{T_\xi}(\un\theta_n,\nabla\psi)\eta(t) dt - \int_0^{T_\xi}(\u^*\theta^*,\nabla\psi)\eta(t) dt\right|
   \nonumber\\
   & = \left| \int_0^{T_\xi}([(\un-\u^*)\theta_n+\u^*(\theta_n- \theta^*)],\nabla\psi)\eta(t) dt \right|
   \nonumber\\
   & \leq \|\eta\|_{L^1(0,T_\xi)} \|\un-\u^*\|_{C([0,T_\xi];L^\infty(\O))}
 \|\theta_n\|_{L^\infty(0,T_\xi;L^2(\O))} \|\nabla\psi\|_{L^2(\O)}  + \left| \int_0^{T_\xi}(\u^*(\theta_n- \theta^*),\nabla\psi)\eta(t) dt \right|
 \nonumber\\
 & \to 0 \text{ as } n\to\infty,
\end{align}
where we have used strong convergence in \eqref{eq:state-u-strong} and weak$^*$ convergence in \eqref{eq:state-theta-weak} for the first and second term, respectively, in the last inequality. Therefore, passing to the limit in \eqref{eq:transport-n}, we obtain
\begin{equation}\label{eq:transport-limit}
 -\int_0^{T_\xi}(\theta^*,\psi)\eta' dt
 -\int_0^{T_\xi}(\u^*\theta^*,\nabla\psi)\eta dt=0.
\end{equation}
Thus
\begin{align*}
    & \left\langle\partial_t\theta^*, \psi\right\rangle_{(H^{1-\xi}(\O))^{\prime}\times H^{1-\xi}(\O)} + \langle\diver[\u^*\theta^*],\psi\rangle_{(H^{1-\xi}(\O))^{\prime}\times H^{1-\xi}(\O)} =0.
\end{align*}
Finally, \eqref{eq:state-theta-strong} and $\theta_n(0)=\theta_0$ give
\begin{equation}\label{eq:theta-initial-limit}
 \theta^*(0)=\theta_0.
\end{equation}

Next, we fix $\boldsymbol\phi\in\Vb^1(\O)$ and
$\eta\in C_c^\infty(0,T_\xi)$, and consider the following weak form of second grade fluid equations for $(\theta_n,\un)$:
\begin{align}
 & \underbrace{\int_0^{T_\xi}(\partial_t\un,\boldsymbol\phi)_{\Vb^1(\O)}  \eta dt}_{=:S_1}
 +2\nu \underbrace{\int_0^{T_\xi}(\Db(\un),\Db(\boldsymbol\phi))\eta dt}_{=:S_2}
 +\beta \underbrace{\int_0^{T_\xi}\int_\Gamma
 (\un\cdot\tau)(\boldsymbol\phi\cdot\tau)\eta dS dt}_{=:S_3}
 \nonumber\\
 &\quad+ \underbrace{\int_0^{T_\xi}b(\boldsymbol\phi,\un,
                 \Upsilon_\alpha(\un) )\eta dt }_{=:S_4}
 - \underbrace{\int_0^{T_\xi}b(\un,\boldsymbol\phi,
                \Upsilon_\alpha (\un) )\eta dt}_{=:S_5}
 \nonumber\\
 &= \xi \underbrace{\int_0^{T_\xi}(\theta_n e_2,\boldsymbol\phi)\eta dt }_{=:S_6}
 + \underbrace{\int_0^{T_\xi}\int_\Gamma
 (\gn\cdot\tau)(\boldsymbol\phi\cdot\tau)\eta dS dt}_{=:S_7}
  +\frac\alpha\nu \underbrace{\int_0^{T_\xi}\int_\Gamma
 (\partial_t\gn\cdot\tau)(\boldsymbol\phi\cdot\tau)
 \eta dS dt }_{=:S_8} .\label{eq:momentum-n}
\end{align}
We consider its terms one by one.

\medskip
\noindent\underline{\emph{The term $S_1$:}} From \eqref{eq:state-ut-weak}, we immediately have
\begin{equation}\label{eq:limit-time-momentum}
 \int_0^{T_\xi}(\partial_t\un,\boldsymbol\phi)_{\Vb^1(\O)}\eta dt
 \to 
 \int_0^{T_\xi}(\partial_t\u^*,\boldsymbol\phi)_{\Vb^1(\O)}\eta dt.
\end{equation}

\medskip
\noindent\underline{\emph{The term $S_2$:}}
The strong convergence \eqref{eq:state-u-strong} implies, in particular,
$\Db(\un)\to\Db(\u^*)$ strongly in $C([0,T_\xi] ;L^2(\O))$.  Hence
\begin{align}
 &\left|\int_0^{T_\xi}
 (\Db(\un)-\Db(\u^*),\Db(\boldsymbol\phi))\eta dt\right| \leq \|\eta\|_{L^1(0,T_\xi)}
 \|\Db(\un)-\Db(\u^*)\|_{C([0,T_\xi] ;L^2(\O))}
 \|\Db(\boldsymbol\phi)\|_{L^2(\O)}\to 0.
 \label{eq:limit-viscous}
\end{align}

\medskip
\noindent\underline{\emph{The term $S_3$:}}
The trace map $H^2(\O) \to H^{3/2}(\Gamma)$ is continuous.  Therefore
\[
 \un\cdot\tau \to\u^*\cdot\tau
 \quad\text{strongly in }C([0,T_\xi];L^2(\Gamma)).
\]
Since $\boldsymbol\phi\cdot\tau \in L^2(\Gamma)$,
\begin{align}
 &\left|\int_0^{T_\xi}\int_\Gamma
 ((\un-\u^*)\cdot\tau)(\boldsymbol\phi\cdot\tau)\eta dS dt\right|
 \nonumber\\
 &\leq \|\eta\|_{L^1(0,T_\xi)}
 \|(\un-\u^*)\cdot\tau\|_{C([0,T_\xi];L^2(\Gamma))}
 \|\boldsymbol\phi\cdot\tau\|_{L^2(\Gamma)}\to 0.
 \label{eq:limit-friction}
\end{align}

\medskip
\noindent\underline{\emph{Terms $S_4$ and $S_5$:}}
Define
\begin{equation*}
 \boldsymbol m_n:=\Upsilon_\alpha(\un) \;\; \text{ and } \;\;
 {\boldsymbol m^*}:=\Upsilon_\alpha(\u^*).
\end{equation*}
By \eqref{eq:state-u-strong}, we find
\begin{equation}\label{eq:mn-strong}
 \boldsymbol m_n\to {\boldsymbol m^*}
 \quad\text{strongly in }C([0,T_\xi];L^2(\O)).
\end{equation}
For $S_4$, consider
\begin{align}
 &b(\boldsymbol\phi,\un,\boldsymbol m_n)
 -b(\boldsymbol\phi,\u,{\boldsymbol m^*})
 =b(\boldsymbol\phi,\un-\u^*,\boldsymbol m_n)
       +b(\boldsymbol\phi,\u^*,
                    \boldsymbol m_n-{\boldsymbol m^*}).
 \label{eq:first-b-split}
\end{align}
Using $H^1(\O)\hookrightarrow L^4(\O)$, we get
\begin{align}
 |b(\boldsymbol\phi,\un-\u^*,\boldsymbol m_n)|
 &\leq \|\boldsymbol\phi\|_{L^4(\O)}
      \|\nabla(\un-\u^*)\|_{L^4(\O)}\|\boldsymbol m_n\|_{L^2(\O)}
 \nonumber\\
 &\leq C\|\boldsymbol\phi\|_{H^1(\O)}
      \|\un-\u^*\|_{H^2(\O)}\|\boldsymbol m_n\|_{L^2(\O)},
 \label{eq:first-b-estimate-a}\\
 |b(\boldsymbol\phi,\u^*,
                    \boldsymbol m_n-\boldsymbol m^*)|
 &\leq C \|\boldsymbol\phi\|_{H^1(\O)}\|\u^*\|_{H^2(\O)}
      \|\boldsymbol m_n-{\boldsymbol m^*}\|_{L^2(\O)}.
 \label{eq:first-b-estimate-b}
\end{align}
The right-hand sides converge uniformly in time to zero by
\eqref{eq:state-u-strong} and \eqref{eq:mn-strong}.  Therefore
\begin{equation}\label{eq:first-b-limit}
 \int_0^{T_\xi}b(\boldsymbol\phi,\un,\boldsymbol m_n)\eta dt
 \to
 \int_0^{T_\xi}b(\boldsymbol\phi,\u^*,
                         {\boldsymbol m^*})\eta dt.
\end{equation}

For $S_5$, consider
\begin{align}
 &b(\un,\boldsymbol\phi,\boldsymbol m_n)
 -b(\u^*,\boldsymbol\phi,{\boldsymbol m^*})
=b(\un-\u^*,\boldsymbol\phi,\boldsymbol m_n)
 +b(\u^*,\boldsymbol\phi,
              \boldsymbol m_n- {\boldsymbol m^*}).
 \label{eq:second-b-split}
\end{align}
Now using $H^2(\O)\hookrightarrow L^\infty(\O)$, we estimate
\begin{align}
 |b(\un-\u^*,\boldsymbol\phi,\boldsymbol m_n)|
 &\le\|\un-\u^*\|_{L^\infty(\O)}
      \|\nabla\boldsymbol\phi\|_{L^2(\O)}
      \|\boldsymbol m_n\|_{L^2(\O)}
 \nonumber\\
 &\leq C\|\un-\u^*\|_{H^2(\O)}
      \|\boldsymbol\phi\|_{H^1(\O)}\|\boldsymbol m_n\|_{L^2(\O)},
 \label{eq:second-b-estimate-a}\\
 |b(\u^*,\boldsymbol\phi,
              \boldsymbol m_n- {\boldsymbol m^*})|
 &\leq C\|\u^*\|_{H^2(\O)}\|\boldsymbol\phi\|_{H^1(\O)}
      \|\boldsymbol m_n- {\boldsymbol m^*}\|_{L^2(\O)}.
 \label{eq:second-b-estimate-b}
\end{align}
Again both right-hand sides converge uniformly in time to zero.  Hence
\begin{equation}\label{eq:second-b-limit}
 \int_0^{T_\xi}b(\un,\boldsymbol\phi,\boldsymbol m_n)\eta dt
 \to
 \int_0^{T_\xi}b(\u^*,\boldsymbol\phi,
                         {\boldsymbol m^*})\eta dt.
\end{equation}

\medskip
\noindent\underline{\emph{The term $S_6$}}:
For $\xi=0$, this term is identically zero.  For $\xi=1$, the weak-star
convergence \eqref{eq:state-theta-weak} implies
\begin{equation}\label{eq:limit-buoyancy}
 \xi\int_0^{T_\xi}(\theta_n e_2,\boldsymbol\phi)\eta dt
 \to 
 \xi\int_0^{T_\xi}(\theta^* e_2,\boldsymbol\phi)\eta dt.
\end{equation}

\medskip
\noindent\underline{\emph{Terms $S_7$ and $S_8$}}:
From \eqref{eq:gn-weak}, we write
\begin{align*}
 \gn\cdot\tau\rightharpoonup\g^*\cdot\tau
 & \text{\; weakly in }L^2(0,T_\xi;L^{2}(\Gamma)), \\
\partial_t\gn\cdot\tau\rightharpoonup \partial_t\g^*\cdot\tau
 & \text{\; weakly in }L^2(0,T_\xi;L^{2}(\Gamma)),
\end{align*}
which implies
\begin{align}
 &\int_0^{T_\xi}\int_\Gamma
 (\gn\cdot\tau)(\boldsymbol\phi\cdot\tau)\eta dS dt
  \to 
 \int_0^{T_\xi}\int_\Gamma
 (\g^*\cdot\tau)(\boldsymbol\phi\cdot\tau)\eta dS dt,
 \label{eq:limit-g-boundary}
\end{align}
and 
\begin{align}
 &\int_0^{T_\xi}\int_\Gamma
 (\partial_t\gn\cdot\tau)(\boldsymbol\phi\cdot\tau)\eta dS dt
  \to 
 \int_0^{T_\xi}\int_\Gamma
 (\partial_t\g^*\cdot\tau)(\boldsymbol\phi\cdot\tau)\eta dS dt.
 \label{eq:limit-g-boundary:t}
\end{align}
Combining \eqref{eq:limit-time-momentum}, \eqref{eq:limit-viscous},
\eqref{eq:limit-friction}, \eqref{eq:first-b-limit},
\eqref{eq:second-b-limit}, \eqref{eq:limit-buoyancy},
\eqref{eq:limit-g-boundary}, and \eqref{eq:limit-g-boundary:t}, we pass to
the limit in \eqref{eq:momentum-n} and obtain
\begin{align}
 & \int_0^{T_\xi}(\partial_t\u^*,\boldsymbol\phi)_{\Vb^1(\O)}  \eta dt
 +2\nu \int_0^{T_\xi}(\Db(\u^*),\Db(\boldsymbol\phi))\eta dt
 +\beta \int_0^{T_\xi}\int_\Gamma
 (\u^*\cdot\tau)(\boldsymbol\phi\cdot\tau)\eta dS dt
 \nonumber\\
 & + \int_0^{T_\xi}b(\boldsymbol\phi,\u^*,
                 \Upsilon_\alpha(\u^*) )\eta dt
 - \int_0^{T_\xi}b(\u^*,\boldsymbol\phi,
                 \Upsilon_\alpha(\u^*) )\eta dt
 \nonumber\\
 &= \xi \int_0^{T_\xi}(\theta^* e_2,\boldsymbol\phi)\eta dt 
 + \int_0^{T_\xi}\int_\Gamma
 (\g^*\cdot\tau)(\boldsymbol\phi\cdot\tau)\eta dS dt
  +\frac\alpha\nu \int_0^{T_\xi}\int_\Gamma
 (\partial_t\g^*\cdot\tau)(\boldsymbol\phi\cdot\tau)
 \eta dS dt. \label{eq:momentum-limit}
\end{align}
Since $\eta$ is arbitrary, \eqref{eq:momentum-limit} holds for almost every
time in the weak sense. In addition, since $\un(0)=\u_0$, the strong convergence in \eqref{eq:state-u-strong} implies $\u^*(0)=\u_0$. Together with \eqref{eq:theta-initial-limit}, this proves the initial
conditions.

We have now proved that the pair $(\theta^*,\u^*)$ is solution of system \eqref{state:eq:state} corresponding to $\g^*$ in the sense of Definition \ref{def:WS:state}, which due to uniqueness gives
\begin{equation}\label{eq:identify-state}
 (\theta^*,\u^*)=(\theta(\g^*),\u(\g^*)).
\end{equation}

We finally pass to the limit in the \eqref{control:eq:cost}. By \eqref{eq:state-theta-strong},
\begin{equation}\label{eq:limit-mix-cost}
	\frac12\|\theta_n(T_\xi)\|_{(H^1(\O))’}^2
	\to
	\frac12\|\theta^*(T_\xi)\|_{(H^1(\O))’}^2.
\end{equation}
By the weak lower semicontinuity property of norm, we get 
\begin{equation}\label{eq:lsc-control-cost}
 \frac\gamma2\|\g^*\|_{\Ucal}^2
 \le\liminf_{n\to\infty}\frac\gamma2\|\gn\|_{\Ucal}^2.
\end{equation}
By the strong convergence in \eqref{eq:state-u-strong} and uniform bound in \eqref{eq:uniform-u}, we obtain
\begin{align*}
 &\left|\int_0^{T_\xi}\|\curl\un\|_{L^2(\O)}^2 dt
 -\int_0^{T_\xi}\|\curl\u^*\|_{L^2(\O)}^2 dt\right|\nonumber\\
 &\quad\le\int_0^{T_\xi}
 \|\curl\un-\curl\u^*\|_{L^2(\O)}
 \left(\|\curl\un\|_{L^2(\O)}+\|\curl\u^*\|_{L^2(\O)}\right) dt
 \to 0,
\end{align*}
which gives
\begin{equation}\label{eq:limit-enstrophy-cost}
 -\frac\zeta2\int_0^{T_\xi}\|\curl\un\|_2^2 dt
 \to 
 -\frac\zeta2\int_0^{T_\xi}\|\curl\u^*\|_2^2 dt.
\end{equation}

Combining \eqref{eq:limit-mix-cost}-\eqref{eq:limit-enstrophy-cost}, and using \eqref{eq:identify-state}, we obtain
\begin{align}
 \J(\g^*)
 &\le\liminf_{n\to\infty}\J(\gn)
 =\J^*.\label{eq:J-liminf}
\end{align}
Because $\g^*\in\Uad$, the definition of $\J^*$ gives $\J^*\leq \J(\g^*)$.  Hence
\[
 \J(\g^*)=\J^\ast=\min_{\g\in\Uad}\J(\g).
\]
This completes the proof of the theorem.
\end{proof}

\section{Linearized state equation}%
\setcounter{equation}{0}\label{sec:linearized-state}

In this section, we investigate the well-posedness of the linearized system associated with the state system, which will be instrumental in deriving the duality relation with the adjoint system and, consequently, the first-order necessary optimality condition. Let us consider the following system:
  \begin{equation}\label{eq:linearized}
\left\{\begin{aligned}
 \partial_t\rho+\u\cdot\nabla\rho
 +\z\cdot\nabla\theta&=0 &&\text{in }(0,T_{\xi})\times\O,\\
 \partial_t\Upsilon_\alpha(\z)-\nu\Delta\z
 + \curl\Upsilon_\alpha(\z)\times\u
 +\curl\Upsilon_\alpha(\u)\times\z+\nabla q
 & = \xi\rho e_2 &&\text{in }(0,T_{\xi})\times\O,\\
 \diver\z&=0, &&\text{in }(0,T_{\xi})\times\O,\\
\rho(0) =0, \;\;  \z(0)& = \boldsymbol{0}, &&\text{in }\O,\\
     \z\cdot\n|_{\Gamma}  =0 \;\;\; \text{ and } \;\;\; [2\nu \n \cdot \Db(\z)  + \beta \z]\cdot\tau|_{\Gamma} &  = \h\cdot\tau.
\end{aligned}
\right.
\end{equation}

\begin{lemma}\label{lem:linearized}
    Assume that $(\theta_0, \u_0)\in H^{1}(\Omega)\times \Vb^3(\O)$. $\g\in\Uad$ and $\h\in\Ucal_0$. Then, there exists a unique solution pair $(\rho, \z)\in L^\infty(0,T_\xi;L^2(\O))\times L^\infty(0,T_\xi;\Wb)$ with $(\partial_t\rho, \partial_t\z)  \in   L^\infty(0,T_\xi;(H^1(\O))^{\prime})\times L^{2}(0,T_\xi;\Vb^1(\O))$ of the system \eqref{eq:linearized}. In addition, $\rho \in C([0,T_\xi];(H^1(\O))^{\prime})$.
\end{lemma}

\begin{proof}[Proof of Lemma \ref{lem:linearized}]
    Note that the following calculations are justified first for the Galerkin
approximations. 
For $\h\in\Ucal_0$, set $\Hbf=\mathcal R\h$, where
\begin{equation}\label{Stokes:Hbf}
	\left\{
	\begin{aligned}
		-\Delta\Hbf+\Hbf+\nabla\pi_{\Hbf}&=0,
		&\diver\Hbf&=0 &&\text{in }\O,\\
		\Hbf\cdot\n&=0,
		&\bigl[2\nu\n\cdot\Db(\Hbf)+\beta\Hbf\bigr]\cdot\tau
		&=\h\cdot\tau &&\text{on }\Gamma,\\
		&\int_\O\pi_{\Hbf}\,dx=0.
	\end{aligned}
	\right.
\end{equation}
By linearity and continuity of $\mathcal R$,
\begin{align*}
	\|\Hbf\|_{L^2(0,T;H^4)}
	+\|\partial_t\Hbf\|_{L^2(0,T;H^3)}
	\leq C_R\|\h\|_{\Ucal},
	\qquad \Hbf(0)=\mathcal R\h(0)=0.
\end{align*}

Define $\w:=\z-\Hbf$. In order to obtain the well-posedness of system \eqref{eq:linearized}, we will establish the well-posedness of the following system:
  \begin{equation}\label{eq:linearized:lift}
	\left\{\begin{aligned}
		\partial_t\rho+\u\cdot\nabla\rho
		+(\w+\Hbf)\cdot\nabla\theta&=0 &&\text{in }(0,T_{\xi})\times\O,\\
		\partial_t\Upsilon_\alpha(\w)-\nu\Delta\w
		+ \curl\Upsilon_\alpha(\w)\times\u
		+\curl\Upsilon_\alpha(\u)\times\w+\nabla q
		& = \xi\rho e_2 \\
		- \partial_t\Upsilon_\alpha(\Hbf) + \nu\Delta\Hbf - \curl\Upsilon_\alpha(\Hbf)\times\u & -\curl\Upsilon_\alpha(\u)\times\Hbf  &&\text{in }(0,T_{\xi})\times\O,\\
		\diver\w&=0, &&\text{in }(0,T_{\xi})\times\O,\\
	\rho(0)=0,  \;\; 	\w(0)& =\boldsymbol{0}, &&\text{in }\O,\\
		\w\cdot\n|_{\Gamma}  =0 \;\;\; \text{ and } \;\;\; [2\nu \n \cdot \Db(\w)  + \beta \w]\cdot\tau|_{\Gamma} &  = 0.
	\end{aligned}
	\right.
\end{equation}

	\medskip
\noindent\textbf{ $H^1$-estimate for $\w$.}
 Take the $L^2(\O)$ inner product of $\eqref{eq:linearized:lift}_2$ with
$\w$.  Since $\w\cdot\n=0$ and $\diver \w=0$, integration by parts gives
\begin{align}\label{time-H1}
 & \frac12\frac{d}{dt}\|\w\|_{\Vb^1(\O)}^2
 +2\nu\|\Db(\w)\|_{L^2(\O)}^2
  +\beta\|\w\cdot\tau\|_{L^2(\Gamma)}^2  
  \nonumber\\
  & = (\xi\rho e_2,\w) - (\curl\Upsilon_\alpha(\w)\times \u,\w) 
    + (- \partial_t\Upsilon_\alpha(\Hbf) + \nu\Delta\Hbf - \curl\Upsilon_\alpha(\Hbf)\times\u -\curl\Upsilon_\alpha(\u)\times\Hbf, \w).
\end{align}
In view of Lemma \ref{lem:curl:estimates}, we get
\begin{equation}\label{first-nonlinear-H1}
 |(\curl\Upsilon_\alpha(\w)\times \u,\w)|
 \leq C\|\u\|_{H^3(\O)}\|\w\|_{H^1(\O)}^2.
\end{equation}
  By H\"older's, Sobolev and Young's inequalities, and Lemma \ref{Lifting:estimates}, we have
\begin{align}
&\left|(\xi\rho e_2,\w) + (- \partial_t\Upsilon_\alpha(\Hbf) + \nu\Delta\Hbf - \curl\Upsilon_\alpha(\Hbf)\times\u -\curl\Upsilon_\alpha(\u)\times\Hbf, \w) \right|
\nonumber\\
&   \leq \xi\|\rho\|_{L^2(\O)}^2  + C(1+\|\u\|_{H^3(\O)} ) \|\w\|_{\Vb^1(\O)}^2
   + C[\|\partial_t\Hbf\|^2_{H^2(\O)} + \|\Hbf\|^2_{H^3(\O)} ].         \label{boundary-H1}
\end{align}
Combining \eqref{time-H1}-\eqref{boundary-H1}, we obtain
\begin{align}
 & \frac12\frac{d}{dt}\|\w\|_{\Vb^1(\O)}^2
 +2\nu\|\Db(\w)\|_{L^2(\O)}^2
 +\beta\|\w\cdot\tau\|_{L^2(\Gamma)}^2
\nonumber\\
&\leq  \xi\|\rho\|_{L^2(\O)}^2   + C(1+\|\u\|_{H^3(\O)})\|\w\|_{\Vb^1(\O)}^2  +C + C[\|\partial_t\Hbf\|^2_{H^2(\O)} + \|\Hbf\|^2_{H^3(\O)} ] . \label{H1-estimate}
\end{align}

\medskip
\noindent\textbf{ $H^2$-estimate for $\w$.}
Set $\Y:=\mathbb{P} \Upsilon_\alpha(\w).$ Take the inner product of $\eqref{eq:linearized:lift}_2$ with $\Y$, and use the equality \eqref{trilinear1} and $\curl\Upsilon_\alpha(\w)=\curl \Y$  to obtain
\begin{align}\label{time-H2}
 \frac12\frac{d}{dt}\|\Y\|_{L^2(\O)}^2 & =  (\Delta \w,\Y) - b(\Y,\u,\Y) - (\curl\Upsilon_\alpha(\u)\times \w,\Y) + (\xi\rho e_2,\Y)
\nonumber \\ 
 & \quad +  (- \partial_t\Upsilon_\alpha(\Hbf) + \nu\Delta\Hbf - \curl\Upsilon_\alpha(\Hbf)\times\u -\curl\Upsilon_\alpha(\u)\times\Hbf, \Y).
\end{align}
Using the H\"older's, Sobolev and Young's inequalities, and Lemma \ref{Lifting:estimates}, we estimate
\begin{align}\label{visc-H2}
    & \left|(\Delta \w,\Y) - b(\Y,\u,\Y) - (\curl\Upsilon_\alpha(\u)\times \w,\Y) + (\xi\rho e_2,\Y)\right|
    \nonumber\\
    &  + |(- \partial_t\Upsilon_\alpha(\Hbf) + \nu\Delta\Hbf - \curl\Upsilon_\alpha(\Hbf)\times\u -\curl\Upsilon_\alpha(\u)\times\Hbf, \Y)|
    \nonumber\\
    & \leq \|\w\|_{H^2(\O)}\|\Y\|_{L^2(\O)} + \|\u\|_{L^{\infty}(\O)}\|\Y\|_{L^2(\O)}^2 + \|\u\|_{H^{3}(\O)} \|\w\|_{L^{\infty}(\O)} \|\Y\|_{L^2(\O)} + \xi \|\rho\|_{L^2(\O)}\|\Y\|_{L^2(\O)}
        \nonumber\\
    & \quad + [\|\partial_t\Hbf\|_{H^2(\O)} + \|\Hbf\|_{H^2(\O)} + \|\Hbf\|_{H^3(\O)}\|\u\|_{H^2(\O)} + \|\u\|_{H^3(\O)} \|\Hbf\|_{H^2(\O)}     ]\|\Y\|_{L^2(\O)}
    \nonumber\\
    & \leq \|\w\|_{\Vb^1(\O)}^2 + (1 + \|\u\|_{H^{3}(\O)})\|\Y\|_{L^2(\O)}^2 +  \xi \|\rho\|_{L^2(\O)}^2 +C [\|\partial_t\Hbf\|^2_{H^2(\O)} + \|\Hbf\|^2_{H^3(\O)}],
\end{align}
where we have also used Lemma \ref{lem:1.6}.  Combining \eqref{time-H2}-\eqref{visc-H2}, we obtain
\begin{align}
 \frac12\frac{d}{dt}\|\Y\|_{L^2(\O)}^2
 &\leq \|\w\|_{\Vb^1(\O)}^2 + (1 + \|\u\|_{H^{3}(\O)})\|\Y\|_{L^2(\O)}^2 +  \xi \|\rho\|_{L^2(\O)}^2 +C [\|\partial_t\Hbf\|^2_{H^2(\O)} + \|\Hbf\|^2_{H^3(\O)}]. \label{H2-energy}
\end{align}
\medskip
\noindent\textbf{ $L^2$-estimate for $\rho$.}
Taking the inner product of $\eqref{eq:linearized:lift}_1$ with $\rho$ gives
\begin{align}
 \frac12\frac{d}{dt}\|\rho\|_{L^2(\O)}^2
 &=-(\w\cdot\nabla\theta,\rho) -(\Hbf\cdot\nabla\theta,\rho) 
 \nonumber\\
 & \leq \|\w+\Hbf\|_{L^\infty(\O)}\|\nabla\theta\|_{L^2(\O)}
       \|\rho\|_{L^2(\O)}                               
    \nonumber\\
    &    \leq C\|\theta\|_{H^1(\O)}^2\|\w\|_{H^2(\O)}^2 + C\|\theta\|_{H^1(\O)}^2\|\Hbf\|_{H^2(\O)}^2
       + C\|\rho\|_{L^2(\O)}^2.     \label{rho-L2}
\end{align}
Adding \eqref{H1-estimate} and \eqref{H2-energy}-\eqref{rho-L2}, and making use of Gronwall's inequality, we conclude that $(\rho, \w)\in L^\infty(0,T_\xi;L^2(\O))\times L^\infty(0,T_\xi;\Wb)$, and satisfies
\begin{align}\label{H1:H2}
 \|\rho\|_{L^\infty(0,T_\xi;L^2(\O))}^2 +   \|\w\|_{L^\infty(0,T_\xi;\Wb)}^2 \leq C_{M} \|\h\|^2_{\Ucal}. 
\end{align}

\medskip
\noindent\textbf{ Weak time-derivative estimate for $\rho$.} From \eqref{eq:linearized:lift}, it is immediate to write that 
\begin{align}
 \|\partial_t\rho\|_{(H^1(\O))^{\prime}}
 &\leq C\left(\|\u\|_{L^\infty(\O)}\|\rho\|_{L^2(\O)}
             +\|\w+\Hbf\|_{L^\infty(\O)}\|\theta\|_{L^2(\O)}\right)     \notag\\
 &\leq C\left(\|\u\|_{H^2(\O)}\|\rho\|_{L^2(\O)}
             +\|\w+\Hbf\|_{H^2(\O)}\|\theta\|_{H^1(\O)}\right). \label{rhot-est}
\end{align}
Therefore, \eqref{H1:H2} gives $\partial_t\rho\in L^\infty(0,T_\xi;(H^1(\O))^{\prime})$. Consequently, $\rho\in C([0,T_\xi];(H^1(\O))^{\prime})$.
 
\medskip
\noindent\textbf{Weak time-derivative estimate for $\w$.}
From $\eqref{eq:linearized:lift}_2$ and \eqref{trilinear1}, we get  
\begin{align}
	& \|\partial_t\w\|^2_{\Vb^1(\O)} 
	\nonumber\\
	& =   - 2\nu(\Db(\w),\Db(\partial_t\w))  - \beta\int_\Gamma(\w\cdot\tau)(\partial_t\w\cdot\tau)dS 
	- b(\partial_t\w,\u,\Upsilon_\alpha(\w)) + b(\u,\partial_t\w, \Upsilon_\alpha(\w)) 
\notag\\
& \quad 	- b(\partial_t\w,\w,\Upsilon_\alpha(\u) ) + b(\w,\partial_t\w, \Upsilon_\alpha(\u))           
	+\xi(\rho e_2,\partial_t\w) 
	\notag\\
	& \quad  
	+  (- \partial_t\Upsilon_\alpha(\Hbf) + \nu\Delta\Hbf - \curl\Upsilon_\alpha(\Hbf) \times\u - \curl\Upsilon_\alpha(\u) \times\Hbf, \partial_t\w) 
	\nonumber\\
	& \leq   C\big[\|\w\|_{\Vb^1(\O)} + \|\w\|_{L^2(\Gamma)} + \|\u\|_{H^2(\O)} \|\w\|_{H^2(\O)}+ \|\Hbf\|_{H^3(\O)} \|\u\|_{H^2(\O)} + \|\u\|_{H^3(\O)} \|\Hbf\|_{H^2(\O)} 
	\nonumber\\
	& \quad + \|\rho\|_{L^2(\O)} + \|\partial_t\Hbf\|_{H^2(\O)} + \|\Hbf\|_{H^2(\O)} \big]\|\partial_t\w\|_{\Vb^1(\O)},  \label{weak-z}
\end{align}
which immediately gives $\partial_t\w\in L^{2}(0,T_\xi;\Vb^1(\O))$. Since system \eqref{eq:linearized} is linear, the uniqueness follows from \eqref{H1:H2}. This completes the proof.
\end{proof}

\section{Stability results}%
\setcounter{equation}{0}\label{sec:stability}

This section's primary goal is to prove a  stability result for the solution of the state equation. This step is essential for studying the control-to-state mapping's directional derivative.  For any  $r >0$, let $(\theta,\u)$ be the solution of system \eqref{state:eq:state} and $(\theta_r,\u_r)$ be the solution of system \eqref{state:eq:state} with $\g$ replaced by $\g+r\h$ in the sense of Definition \ref{def:WS:state}, where $\g\in\Uad$ and $\h\in\Ucal_0$.   Define
\begin{equation}\label{stability:differences}
 \widebar\theta_r:=\theta_r-\theta,
 \qquad
 \widebar\u_r:=\u_r-\u.
\end{equation}
Then $(\widebar\theta_r,\widebar\u_r)$ satisfies
\begin{equation}\label{stability:difference:system}
\left\{
\begin{aligned}
 \partial_t\widebar\theta_r
 +\u_r\cdot\nabla\widebar\theta_r
 +\widebar\u_r\cdot\nabla\theta&=0,
 &&\text{in }(0,T_\xi)\times\O,\\
 \partial_t\Upsilon_\alpha(\widebar\u_r )-\nu\Delta\widebar\u_r
 +\curl\Upsilon_\alpha(\widebar\u_r )\times\u_r
  +\curl\Upsilon_\alpha(\u )\times\widebar\u_r +\nabla\widebar p_r
 &=\xi\widebar\theta_re_2,
 &&\text{in }(0,T_\xi)\times\O,\\
 \diver\widebar\u_r&=0,
 &&\text{in }(0,T_\xi)\times\O,\\
\widebar\theta_r(0)=0,
 \qquad \widebar\u_r(0) & =\boldsymbol{0} ,
 &&\text{in }\O,\\
 \widebar\u_r\cdot\n=0,
 \qquad
 [2\nu\n\cdot\Db(\widebar\u_r)+\beta\widebar\u_r]
 \cdot\tau&=r\h\cdot\tau,
 &&\text{on }(0,T_\xi)\times \Gamma.
\end{aligned}
\right.
\end{equation}

In the next lemma, we present the main result of this section.

\begin{lemma}[Lipschitz stability]\label{thm:stability}
	We have the following stability estimates:
\begin{itemize}
	\item [$(i)$] there exists a constant $C_M>0$, independent of $r$, such that
	\begin{align}
		&\|\theta_r-\theta\|_{L^\infty(0,T_\xi;L^2(\O))}^2
		+\|\u_r-\u\|_{L^\infty(0,T_\xi;H^2(\O))}^2
		\leq C_Mr^2\|\h\|_{\Ucal}^2.                              \label{stability:final}
	\end{align}

\item [$(ii)$]  as $r\to0$, we have 
\begin{align}\label{goal:H1}
	\theta_r-\theta \to 0 \text{ in } L^\infty(0,T_\xi;H^1(\O)).
\end{align}
\end{itemize}
\end{lemma}

\begin{proof}

The proof of $(i)$ is divided into the following three steps. Let us define $\widebar\v_r:= \widebar\u_r- r\Hbf$. Then $(\widebar\theta_r,\widebar\v_r)$ satisfies
\begin{equation}\label{stability:difference:system:vr}
\small	\left\{
	\begin{aligned}
		\partial_t\widebar\theta_r
		+\u_r\cdot\nabla\widebar\theta_r
		+(\widebar\v_r + r\Hbf)\cdot\nabla\theta&=0,
		&&\text{in }(0,T_\xi)\times\O,\\
		\partial_t\Upsilon_\alpha(\widebar\v_r + r\Hbf )-\nu\Delta(\widebar\v_r + r\Hbf)
		+\curl\Upsilon_\alpha(\widebar\v_r + r\Hbf )\times\u_r
		&+\curl\Upsilon_\alpha(\u )\times(\widebar\v_r + r\Hbf) \\ 
		+\nabla\widebar p_r
		&=\xi\widebar\theta_re_2,
		&&\text{in }(0,T_\xi)\times\O,\\
		\diver\widebar\v_r&=0,
		&&\text{in }(0,T_\xi)\times\O,\\
	\widebar\theta_r(0)=0,	
		\qquad \widebar\v_r(0) & =\boldsymbol{0},
		&&\text{in }\O,\\
		\widebar\v_r\cdot\n=0,
		\qquad
		[2\nu\n\cdot\Db(\widebar\v_r)+\beta\widebar\v_r]
		\cdot\tau&=0,
		&&\text{on }(0,T_\xi)\times \Gamma.
	\end{aligned}
	\right.
\end{equation}

\medskip
\noindent\textit{Step 1: the $H^1$-estimate.}
Taking the inner product of $\eqref{stability:difference:system:vr}_2$ with $\widebar\v_r$, using the boundary
condition, and arguing as in the proof of Lemma \ref{lem:linearized} gives
\begin{align}
 \frac12\frac{d}{dt}\|\widebar\v_r\|_{\Vb^1(\O)}^2
 +2\nu\|\Db(\widebar\v_r)\|_{L^2(\O)}^2
 +\beta\|\widebar\v_r\cdot\tau\|_{L^2(\Gamma)}^2
 &\leq C(1+\|\u_r\|_{H^3(\O)})
        \|\widebar\v_r\|_{\Vb^1(\O)}^2  +   C\xi\|\widebar\theta_r\|_{L^2(\O)}^2
    \notag\\  &\quad    +Cr^2\|\h\|_{H^{3/2}(\Gamma)}^2  + Cr^2\|\partial_t\h\|_{H^{1/2}(\Gamma)}^2.\label{stability:H1}
\end{align}

\medskip
\noindent\textit{Step 2: the $H^2$-estimate.}
Set $\widebar\Y_r:=\mathbb{P}\Upsilon_\alpha(\widebar\v_r ).$ Taking the inner product of $\eqref{stability:difference:system:vr}_2$ with $\widebar\Y_r$, and using  $\curl\Upsilon_\alpha(\widebar\v_r ) =\curl\widebar\Y_r$, \eqref{lem:1.6:1},  \eqref{H2-estimate} and Young's inequality, we obtain
\begin{align}
 \frac12\frac{d}{dt}\|\widebar\Y_r\|_{L^2(\O)}^2
&\leq C(1+\|\u_r\|_{H^3(\O)}^2+\|\u\|_{H^3(\O)}^2)
 \left(
 \|\widebar\Y_r\|_{L^2(\O)}^2
 +\|\widebar\u_r\|_{\Vb^1(\O)}^2
 \right) +C\xi\|\widebar\theta_r\|_{L^2(\O)}^2
\nonumber\\
& \quad + Cr^2\|\h\|_{H^{3/2}(\Gamma)}^2  + Cr^2\|\partial_t\h\|_{H^{1/2}(\Gamma)}^2.                  \label{stability:H2}
\end{align}

\medskip
\noindent\textit{Step 3: the scalar estimate.}
From $\eqref{stability:difference:system}_1$, we have
\begin{align}
 \frac12\frac{d}{dt}\|\widebar\theta_r\|_{L^2(\O)}^2
 &=-(\widebar\u_r\cdot\nabla\theta,\widebar\theta_r)  
 \leq C\|\theta\|_{H^1(\O)}^2
         \|\widebar\v_r\|_{H^2(\O)}^2 + C r^2 \|\theta\|_{H^1(\O)}^2
         \|\Hbf\|_{H^2(\O)}^2
       +C\|\widebar\theta_r\|_{L^2(\O)}^2.                     \label{stability:theta}
\end{align}
Adding \eqref{stability:H1}, \eqref{stability:H2}, and
\eqref{stability:theta}, using \eqref{v:H3:estimate}-\eqref{theeta:H1:estimate} (for $\xi=0$) and
\eqref{v:theta:estimate} (for $\xi=1$), and then applying Gronwall's lemma proves
\eqref{stability:final}.

Finally, we show part $(ii)$. The estimate \eqref{stability:final}, together with the uniform $H^3$-bound, implies by interpolation that, for every
$0<\delta<1$,
\begin{equation*}
 \u_r-\u\to 0
 \quad\text{in }
 L^\infty(0,T_\xi;H^{2+\delta}(\O)).
\end{equation*}
In particular, by Sobolev embedding, we have
\begin{equation}\label{stability:W1infty}
 \u_r-\u\to0
 \quad\text{in }
 L^\infty(0,T_\xi;W^{1,\infty}(\O)).
\end{equation}

 Next, we prove that the convergence \eqref{stability:W1infty}   implies
\begin{align}
    \|\theta_r-\theta\|_{L^\infty(0,T_\xi;H^1(\O))}
 \to0.
\end{align}

Let us choose a sequence
\begin{equation}\label{smooth:initial:data}
 \theta_0^m\in H^2(\O),
 \qquad
 \theta_0^m\to\theta_0
 \quad\text{strongly in }H^1(\O).
\end{equation}
For the already fixed state velocities $\u_r$ and $\u$, define the auxiliary
transport solutions by
\begin{align}
 \partial_t\theta_r^m+\u_r\cdot\nabla\theta_r^m&=0,
 &\theta_r^m(0)&=\theta_0^m,                              \label{auxiliary:r}\\
 \partial_t\theta^m+\u\cdot\nabla\theta^m&=0,
 &\theta^m(0)&=\theta_0^m.                               \label{auxiliary:base}
\end{align}
These are only auxiliary scalar transport problems with the state velocities
held fixed; they are not asserted to be new solutions of the coupled state
system.

The functions $\theta_r-\theta_r^m$ and $\theta-\theta^m$ solve homogeneous
transport equations with initial value $\theta_0-\theta_0^m$.  From
\eqref{v:H3:estimate}-\eqref{theeta:H1:estimate} (for $\xi=0$) and
\eqref{v:theta:estimate} (for $\xi=1$), and
$H^3(\O)\hookrightarrow W^{1,\infty}(\O)$, we obtain, uniformly in $r$,
\begin{align}
 \|\theta_r-\theta_r^m\|_{L^\infty(0,T_\xi;H^1(\O))}
 &\leq C_M\|\theta_0-\theta_0^m\|_{H^1(\O)},                 \label{approximation:r}\\
 \|\theta-\theta^m\|_{L^\infty(0,T_\xi;H^1(\O))}
 &\leq C_M\|\theta_0-\theta_0^m\|_{H^1(\O)}.               \label{approximation:base}
\end{align}

Next observe that, for each fixed $m$, $\theta^m$ remains bounded in $H^2(\O)$.  Indeed, we have (for example see the proof of Lemma \ref{thm:adjoint:regularity} below)
\begin{align*}
 \frac{d}{dt}\|\theta^m\|_{H^2(\O)}^2
 \leq C(1+\|\u\|_{H^3(\O)})\|\theta^m\|_{H^2(\O)}^2,
\end{align*}
and therefore
\begin{equation}\label{theta:m:H2}
 \|\theta^m\|_{L^\infty(0,T_\xi;H^2(\O))}
 \leq C_M\|\theta_0^m\|_{H^2(\O)}=:K_m.
\end{equation}

Define $ q_r^m:=\theta_r^m-\theta^m.$ Subtracting \eqref{auxiliary:base} from \eqref{auxiliary:r}, we obtain
\begin{equation}\label{equation:qrm}
 \partial_tq_r^m+\u_r\cdot\nabla q_r^m=\underbrace{-(\u_r-\u)\cdot\nabla\theta^m}_{=:F_r^m},
 \qquad q_r^m(0)=0,
\end{equation}
such that
\begin{align}
 \|F_r^m\|_{H^1(\O)}
 &\leq C\|\u_r-\u\|_{W^{1,\infty}(\O)}
          \|\theta^m\|_{H^2(\O)}.                            \label{forcing:H1}
\end{align}

Taking inner product of \eqref{equation:qrm} and its spatial derivatives by $q_r^m$ and
$\nabla q_r^m$, respectively, gives
\begin{align*}
 \frac{d}{dt}\|q_r^m\|_{H^1(\O)}^2
 &\leq C\bigl(1+\|\nabla\u_r\|_{L^\infty(\O)}\bigr)
          \|q_r^m\|_{H^1(\O)}^2
       +C\|F_r^m\|_{H^1(\O)}^2.                          
\end{align*}
Using \eqref{v:H3:estimate}-\eqref{theeta:H1:estimate} (for $\xi=0$),
\eqref{v:theta:estimate} (for $\xi=1$), \eqref{stability:W1infty},
\eqref{theta:m:H2}, and \eqref{forcing:H1}, Gronwall's inequality gives
\begin{align*}
 \|q_r^m\|_{L^\infty(0,T_\xi;H^1(\O))}^2
 &\leq C_M\int_0^{T_\xi}\|F_r^m(s)\|_{H^1(\O)}^2\,ds 
 \leq C_MT_\xi\|\u_r-\u\|_{L^\infty(0,T_\xi;W^{1,\infty}(\O))}^2K_m^2.
\end{align*}
Consequently, for every fixed $m$,
\begin{equation}\label{fixed:m:convergence}
 \|\theta_r^m-\theta^m\|_{L^\infty(0,T_\xi;H^1(\O))}
 \leq C_M\sqrt{T_\xi}\,K_m\varepsilon_r
 \to0
 \qquad\text{as }r\to0,
\end{equation}
where $\varepsilon_r:=\|\u_r-\u\|_{L^\infty(0,T_\xi;W^{1,\infty}(\O))}$.

\vskip 2mm
\noindent

Let $\eta>0$ be arbitrary. From \eqref{approximation:r},
\eqref{approximation:base}, and \eqref{fixed:m:convergence}, we have
\begin{align}
 \|\theta_r-\theta\|_{L^\infty(0,T_\xi;H^1(\O))}
 &\leq
 \|\theta_r-\theta_r^m\|_{L^\infty(0,T_\xi;H^1(\O))}
 +\|\theta_r^m-\theta^m\|_{L^\infty(0,T_\xi;H^1(\O))}
 +\|\theta^m-\theta\|_{L^\infty(0,T_\xi;H^1(\O))}
 \notag\\
 &\leq
 2C_M\|\theta_0-\theta_0^m\|_{H^1(\O)}
 +C_M\sqrt{T_\xi}\,K_m\varepsilon_r.
 \label{ordered:estimate}
\end{align}

Since
\[
\theta_0^m\to\theta_0
\qquad\text{strongly in }H^1(\O),
\]
we can first choose $m=m(\eta)$ sufficiently large such that
\begin{equation*}
 2C_M\|\theta_0-\theta_0^m\|_{H^1(\O)}
 <\frac{\eta}{2}.
\end{equation*}
For this fixed $m$, $K_m$ is finite and independent of $r$. Since
$\varepsilon_r\to0$ as $r\to0$, we can choose
$r=r(\eta,m)$ sufficiently small such that
\begin{equation*}
 C_M\sqrt{T_\xi}\,K_m\varepsilon_r
 <\frac{\eta}{2}.
\end{equation*}
Therefore, for all sufficiently small $r$,
\begin{equation*}
 \|\theta_r-\theta\|_{L^\infty(0,T_\xi;H^1(\O))}
 <\eta.
\end{equation*}
Since $\eta>0$ was arbitrary, we conclude that
\begin{equation*}
 \|\theta_r-\theta\|_{L^\infty(0,T_\xi;H^1(\O))}
 \to0
 \qquad\text{as }r\to0.
\end{equation*}
This proves \eqref{goal:H1} and this completes the proof.
\end{proof}

\section{Directional differentiability of the control-to-state mapping}%
\setcounter{equation}{0}\label{sec:gateaux}

The differentiability of the control-to-state mapping is examined in this section. More specifically, we will demonstrate that the solution of the linearized system \eqref{eq:linearized} provides the directional derivative of the control-to-state mapping using the stability property defined in the preceding section.

	Because $\Uad$ is a closed subset of the affine space
$\Ucal_{\g_{\rm in}}$, differentiability is understood along feasible
directions. For $\g\in\Uad$, define
\begin{align*}
	\mathcal T_{\Uad}(\g)
	:=\left\{\h\in\Ucal_0:
	\g+r\h\in\Uad\text{ for all sufficiently small }r>0\right\}.
\end{align*}

 Define the difference quotients
\begin{equation*}
 \rho_r:=\frac{\theta_r-\theta}{r} \;\; \text{ and } \;\;
 \z_r:=\frac{\u_r-\u}{r}.
\end{equation*}
In view of Lemma \ref{thm:stability}, we have the following uniform estimate:
\begin{equation}\label{quotient:uniform}
 \sup_{0<r<r_0}\left(
 \|\rho_r\|_{L^\infty(0,T_\xi;L^2(\O))}
 +\|\z_r\|_{L^\infty(0,T_\xi;H^2(\O))}
 \right)\leq C_M\|\h\|_{\Ucal}.
\end{equation}
Define new variables  
\begin{equation*}
 \sigma_r:=\rho_r-\rho
 =\frac{\theta_r-\theta}{r}-\rho \;\; \text{ and }
 \;\;
 \chi_r:=\z_r-\z
 =\frac{\u_r-\u}{r}-\z,
\end{equation*}
where $(\rho,\z)$ satisfies system \eqref{eq:linearized}. Then, the pair $(\sigma_r, \chi_r)$ satisfies the following system:
\begin{equation}\label{error:system}
\left\{
\begin{aligned}
 \partial_t\sigma_r+\u\cdot\nabla\sigma_r
 +\chi_r\cdot\nabla\theta
 &=-\z_r\cdot\nabla(\theta_r-\theta), &&\text{in }(0,T_\xi)\times\O,\\
 \partial_t\Upsilon_\alpha(\chi_r)-\nu\Delta\chi_r
 &+\curl\Upsilon_\alpha(\chi_r)\times\u
 +\curl\Upsilon_\alpha(\u)\times\chi_r\\ 
 +\nabla\pi_r
 &=\xi\sigma_re_2  -  r\curl\Upsilon_\alpha(\z_r)\times\z_r, &&\text{in }(0,T_\xi)\times\O,\\
 \diver\chi_r&=0, &&\text{in }(0,T_\xi) \times\O,\\
   \sigma_r(0) =0, \quad  \chi_r(0)&=\boldsymbol{0}, &&\text{in }\O,\\
 \chi_r\cdot\n=0,\qquad&
 [2\nu\n\cdot\Db(\chi_r)+\beta\chi_r]\cdot\tau=0 &&\text{on }(0,T_\xi)\times \Gamma.
\end{aligned}
\right.
\end{equation}

\begin{theorem}[Directional differentiability]
	\label{thm:Gateaux}
	Let $\g\in\Uad$ and $\h\in\mathcal T_{\Uad}(\g)$.  Let
	$(\theta_r,\u_r)$ and $(\theta,\u)$ be the states corresponding to
	$\g+r\h$ and $\g$, respectively, and let $(\rho,\z)$ solve the linearized
	system \eqref{eq:linearized}.  Then, as $r\downarrow0$,
	\begin{align}
		\frac{\theta_r-\theta}{r}&\to\rho
		&&\text{in }L^\infty(0,T_\xi;L^2(\O))
		\cap C([0,T_\xi];(H^1(\O))'),\label{eq:directional-theta}\\
		\frac{\u_r-\u}{r}&\to\z
		&&\text{in }L^\infty(0,T_\xi;\Vb^1(\O)).
		\label{eq:directional-u}
	\end{align}
	In particular,
	\[
	\left\|
	\frac{\theta(T_\xi;\g+r\h)-\theta(T_\xi;\g)}{r}
	-\rho(T_\xi)
	\right\|_{(H^1(\O))'}\to0.
	\]
\end{theorem}

\begin{proof}
\noindent\textbf{$\Vb^1$-estimate for $\chi_r$.}
Taking the inner product of $\eqref{error:system}_2$ with
$\chi_r$, \eqref{trilinear1} and the homogeneous boundary condition gives
\begin{align}
 & \frac12\frac{d}{dt}\|\chi_r\|_{\Vb^1(\O)}^2
 +2\nu\|\Db(\chi_r)\|_{L^2(\O)}^2
 +\beta\|\chi_r\cdot\tau\|_{L^2(\Gamma)}^2 \nonumber\\
 & = -(\curl\Upsilon_\alpha(\chi_r )\times\u,\chi_r) - r (\curl\Upsilon_\alpha(\z_r) \times\z_r,\chi_r) + \xi(\sigma_re_2,\chi_r).                                      \label{error:chi:H1:first}
\end{align}
The inequality $\eqref{curl:estimates}_3$ implies 
\begin{align}
    |(\curl\Upsilon_\alpha(\chi_r )\times\u,\chi_r)|\leq C\|\u\|_{H^3(\O)}\|\chi_r\|_{\Vb^1(\O)}^2 \leq  C \|\chi_r\|_{\Vb^1(\O)}^2.
\end{align}
The identity $\eqref{trilinear1}_2$, Sobolev embedding and \eqref{quotient:uniform} give
\begin{align}
 |r(\curl\Upsilon_\alpha(\z_r)\times\z_r,\chi_r)|
 &\leq
 r\left(
 \|\chi_r\|_{L^4(\O)}\|\nabla\z_r\|_{L^4(\O)}
 +\|\z_r\|_{L^\infty(\O)}\|\nabla\chi_r\|_{L^2(\O)}
 \right)\|\Upsilon_\alpha(\z_r)\|_{L^2(\O)}                               \notag\\
 &\leq C r\|\z_r\|_{H^2(\O)}^2\|\chi_r\|_{\Vb^1(\O)}
 \nonumber\\
 & \leq C_M (r^2+\|\chi_r\|^2_{\Vb^1(\O)}).             \label{quadratic:remainder}
\end{align}
An application of Cauchy-Schwarz and Young's inequalities estimates
\begin{align}\label{theta:r}
|\xi(\sigma_re_2,\chi_r)|\leq C (\xi\|\sigma_r\|_{L^2(\O)}^2 + \|\chi_r\|_{L^2(\O)}^2)
\end{align}
Therefore, by \eqref{error:chi:H1:first}-\eqref{theta:r}, we obtain
\begin{equation}\label{error:chi:H1}
 \frac{d}{dt}\|\chi_r\|_{\Vb^1(\O)}^2
 \leq C_M\left(
 \|\chi_r\|_{\Vb^1(\O)}^2+\xi\|\sigma_r\|_{L^2(\O)}^2+r^2
 \right).
\end{equation}

\medskip
\noindent\textbf{$L^2$-estimate for $\sigma_r$.}
The scalar error equation has deliberately been written in the form
\begin{equation*}
 \partial_t\sigma_r+\u\cdot\nabla\sigma_r
 +\chi_r\cdot\nabla\theta
 =-\z_r\cdot\nabla(\theta_r-\theta).
\end{equation*}
Taking the inner product with $\sigma_r$, and using
$\diver\u=0$ and $\u\cdot\n=0$, gives
\begin{align}
 \frac12\frac{d}{dt}\|\sigma_r\|_{L^2(\O)}^2
 &\leq \|\chi_r\|_{L^\infty(\O)}\|\nabla\theta\|_{L^2(\O)}
          \|\sigma_r\|_{L^2(\O)}   +\|\z_r\|_{L^\infty(\O)}
       \|\nabla(\theta_r-\theta)\|_{L^2(\O)}
       \|\sigma_r\|_{L^2(\O)}
       \nonumber\\
&\leq C \|\chi_r\|^2_{L^\infty(\O)}\|\nabla\theta\|^2_{L^2(\O)}
    + C \|\z_r\|^2_{L^\infty(\O)}
       \|\nabla(\theta_r-\theta)\|^2_{L^2(\O)} + C
       \|\sigma_r\|^2_{L^2(\O)}.                                  \label{sigma:direct:L2}
\end{align}
By \eqref{quotient:uniform}, there exists $K>1$, independent of $r$, such
that
\begin{equation}\label{chi:H2:uniform}
 \|\chi_r\|_{L^\infty(0,T_\xi;H^2(\O))}
 +\|\z_r\|_{L^\infty(0,T_\xi;H^2(\O))}\leq K.
\end{equation}
Use \eqref{error:chi:H1}, \eqref{sigma:direct:L2} and \eqref{infty:inequality} to obtain, for $0< \varepsilon \leq 1$,
\begin{align*}
	& \frac{d}{dt} \left[\|\sigma_r\|_{L^2(\O)}^2 + \|\chi_r\|_{\Vb^1(\O)}^2\right]
\nonumber\\
	&\leq  C_0 
	(\|\nabla(\theta_r-\theta)\|^2_{L^{\infty}(0,T_{\xi};L^2(\O))}+r^2 + \|\sigma_r\|^2_{L^2(\O)} +  \|\chi_r\|_{\Vb^1(\O)}^2) + \frac{C_1}{\varepsilon} \|\chi_r\|^{2-2\varepsilon}_{\Vb^1(\O)},
\end{align*}
where $C_0,C_1>0$ are positive constants independent of $r$ and $\varepsilon$.

Define $\mathbf{F}_r(t):= \|\nabla(\theta_r-\theta)\|^2_{L^{\infty}(0,T_{\xi};L^2(\O))}+r^2 + \|\sigma_r(t)\|_{L^2(\O)}^2 + \|\chi_r(t)\|_{\Vb^1(\O)}^2$,   then $\mathbf{F}_r(\cdot)$ satisfies, for $0< \varepsilon \leq 1$,
\begin{align}\label{Fr:differential}
	\frac{d}{dt}\mathbf F_r(t)
	\leq
	C_0\mathbf F_r(t)
	+\frac{C_1}{\varepsilon}
	\mathbf F_r(t)^{1-\varepsilon},
	\qquad
	\mathbf F_r(0)=\|\nabla(\theta_r-\theta)\|^2_{L^{\infty}(0,T_{\xi};L^2(\O))}+r^2:=q_r,
\end{align}	
for almost every $t\in(0,T_\xi)$. Since $\mathbf F_r\in AC([0,T_\xi])=W^{1,1}(0,T_\xi)$ and
$\mathbf F_r(t)\geq r^2>0$, the range of $\mathbf F_r$ is contained
in a compact subset of $(0,\infty)$. Hence the function
$\Phi(s)=s^\varepsilon$ is continuously differentiable on a
neighborhood of this range. The chain rule for one-dimensional Sobolev
functions \cite[Corollary~8.11]{Brezis2011} therefore gives
\[
\mathbf F_r^\varepsilon\in W^{1,1}(0,T_\xi),
\qquad
\frac{d}{dt}\mathbf F_r^\varepsilon
=
\varepsilon\mathbf F_r^{\varepsilon-1}\mathbf F_r'
\quad\text{a.e. in }(0,T_\xi).
\]	
Multiplying \eqref{Fr:differential} by
	$\varepsilon\mathbf F_r^{\varepsilon-1}$ gives
	\begin{align*}
		\frac{d}{dt}\mathbf F_r^\varepsilon
		&=
		\varepsilon\mathbf F_r^{\varepsilon-1}\mathbf F_r'
		 \leq
		C_0\varepsilon\mathbf F_r^\varepsilon+C_1.
	\end{align*}
	Therefore, for every $0\leq s\leq t\leq T_\xi$,
	\begin{align}\label{nonlinear:gronwall:interval}
		\mathbf F_r(t)
		\leq
		\left\{
		\mathbf F_r(s)^\varepsilon
		e^{C_0\varepsilon(t-s)}
		+
		C_1\int_s^t
		e^{C_0\varepsilon(t-\tau)}\,d\tau
		\right\}^{1/\varepsilon}.
	\end{align}

	Choose $\tau_0>0$ sufficiently small such that
	\begin{align}\label{tau:small}
		C_1\tau_0<\frac14.
	\end{align}
	Introduce
	\begin{align*}
		B_\varepsilon(\tau_0)
		:=
		C_1\int_0^{\tau_0}
		e^{C_0\varepsilon(\tau_0-\tau)}\,d\tau.
	\end{align*}
	We have
	\begin{align*}
		B_\varepsilon(\tau_0)
		=
		\frac{C_1}{C_0\varepsilon}
		\left(e^{C_0\varepsilon\tau_0}-1\right)
		\to C_1\tau_0<\frac14
		\qquad\text{ as } \quad \varepsilon\to0.
	\end{align*}
	
	We first consider the interval $[0,\tau_0]$. Since
	$\mathbf F_r(0)=q_r\to0$, for sufficiently small $r$ define
	\begin{align}\label{epsilon:r:choice}
		L_r:=\log\left(\frac{1}{q_r}\right) \;\; \text{ and }\;\;
		\varepsilon_r:=\frac{1}{\sqrt{L_r}}.
	\end{align}
	Then
	\begin{align}\label{qr:epsilon}
		\varepsilon_r\to0 \;\;\; \text{ and } \;\;\;
		[q_r]^{\varepsilon_r}
		=
		\exp\left(-\varepsilon_rL_r\right)
		=
		\exp\left(-\sqrt{L_r}\right)
		\to0.
	\end{align}
	Applying \eqref{nonlinear:gronwall:interval} with $s=0$ and
	$\varepsilon=\varepsilon_r$, we obtain
	\begin{align*}
		\sup_{0\leq t\leq\tau_0}\mathbf F_r(t)
		\leq
		\left\{
		[q_r]^{\varepsilon_r}
		e^{C_0\varepsilon_r\tau_0}
		+
		B_{\varepsilon_r}(\tau_0)
		\right\}^{1/\varepsilon_r}.
	\end{align*}
	By \eqref{tau:small} and \eqref{qr:epsilon}, for all sufficiently
	small $r$,
	\begin{align*}
		[q_r]^{\varepsilon_r}
		e^{C_0\varepsilon_r\tau_0}
		+
		B_{\varepsilon_r}(\tau_0)
		\leq\frac12.
	\end{align*}
	Consequently,
	\begin{align*}
		\sup_{0\leq t\leq\tau_0}\mathbf F_r(t)
		\leq
		\left(\frac12\right)^{1/\varepsilon_r}.
	\end{align*}
	Thus,
	\begin{align}\label{first:interval}
		\sup_{0\leq t\leq\tau_0}\mathbf F_r(t)
		\to0 \;\;\; \text{ as } \;\;\; r\to0.
	\end{align}

	To cover the whole interval $[0,T_\xi]$, choose an integer
	$N\in\mathbb N$ such that $T_\xi\leq N\tau_0,$ and define
	\begin{align*}
		t_k:=k\tau_0\wedge T_\xi,
		\qquad k=0,\ldots,N.
	\end{align*}
	Assume inductively that $\sup\limits_{0\leq t\leq t_k}\mathbf F_r(t)\to0.$  In particular,
	\begin{align*}
		\mathbf F_r(t_k)\to0.
	\end{align*}
	For sufficiently small $r$, choose
	\begin{align*}
		\varepsilon_{r,k}
		:=
		\left[
		\log\left(\frac{1}{\mathbf F_r(t_k)}\right)
		\right]^{-1/2}.
	\end{align*}
	Applying \eqref{nonlinear:gronwall:interval} on
	$[t_k,t_{k+1}]$ with $\varepsilon=\varepsilon_{r,k}$, exactly the
	same calculation gives
	\begin{align*}
		\sup_{t_k\leq t\leq t_{k+1}}\mathbf F_r(t)
		\to0 \;\;\; \text{ as } \;\;\; r\to0.
	\end{align*}
	Since $N$ is finite, induction gives
	\begin{align*}
		\sup_{0\leq t\leq T_\xi}\mathbf F_r(t)
		\to0 \;\;\; \text{ as } \;\;\; r\to0.
	\end{align*}
	Finally, $0\leq \|\sigma_r(t)\|_{L^2(\O)}^2 + \|\chi_r(t)\|_{\Vb^1(\O)}^2 \leq\mathbf F_r$, and therefore
	\begin{align*}
		\|\sigma_r\|_{L^\infty(0,T_\xi;L^2(\O))}
		+
		\|\chi_r\|_{L^\infty(0,T_\xi;\Vb^1(\O))}
		\to0.
	\end{align*}
Moreover, since $\sigma_{r}\in C([0,T_\xi];(H^{1}(\O))')$, we also have 
\begin{align}\label{error:strong:final:time}
	\|\sigma_r\|_{C([0,T_\xi];(H^{1}(\O))')}
	\to 0
	\qquad\text{as } \qquad r\to0.
\end{align}
This proves, in particular, convergence of the difference quotient at
$t=T_\xi$ in the topology of the terminal mix-norm. Hence, we complete the proof.
\end{proof}

\section{Adjoint system}%
\setcounter{equation}{0}\label{sec:adjoint}

In this section, we study the adjoint system associated with the underlying control problem which will be used to derive the duality relation. 

Fix $\g\in\Uad$ and let $(\theta,\u)$ be its state. For
$\pfrak\in\Vb^1(\O)$ and $\boldsymbol\phi\in\Wb$, we define
\begin{align}
	\mathcal A_{\u}(\pfrak,\boldsymbol\phi)
	&:=
	2\nu(\Db(\pfrak),\Db(\boldsymbol\phi))
	+\beta\int_\Gamma(\pfrak\cdot\tau)(\boldsymbol\phi\cdot\tau)\,dS
	 +b(\pfrak,\boldsymbol\phi,\Upsilon_\alpha(\u))
	-b(\boldsymbol\phi,\pfrak,\Upsilon_\alpha(\u))
	\nonumber\\
	&\quad+b(\pfrak,\u, \Upsilon_\alpha(\boldsymbol\phi))
	-b(\u,\pfrak,\Upsilon_\alpha(\boldsymbol\phi) ).
	\label{eq:adjoint-form}
\end{align}
We aim to establish that the pair of adjoint variables $(\varrho,\pfrak)$ satisfies
\begin{equation}\label{eq:adjoint-scalar}
	\left\{
	\begin{aligned}
		-\partial_t\varrho-\u\cdot\nabla\varrho
		&=\xi(\pfrak\cdot\e_2)
		&&\text{in }(0,T_\xi)\times\O,\\
		\varrho(T_\xi) =\Lambda^{-2}\theta(T_\xi), \qquad \pfrak(T_\xi)& =\boldsymbol{0},
		&&\text{in }\O,
	\end{aligned}
	\right.
\end{equation}
and the variational form 
\begin{equation}\label{weak:adjoint:p}
	\left\langle-\partial_t(\mathcal M_{\alpha\beta}\pfrak),
	\boldsymbol\phi\right\rangle_{\Wb'\times\Wb}
	+\mathcal A_{\u}(\pfrak,\boldsymbol\phi)
	=(\theta\nabla\varrho,\boldsymbol\phi)
	-\zeta(\curl\u,\curl\boldsymbol\phi),
\end{equation}
for every $\boldsymbol\phi\in\Wb$ and almost every $t\in(0,T_\xi)$. We do not
impose an additional strong tangential boundary condition at this level of
regularity.

\begin{remark}
	The adjoint velocity equation is understood in the variational sense. The conditions $\diver\pfrak=0$ in $\O$ and $\pfrak\cdot\n=0$ on $\Gamma$ are incorporated into the space $\Vb^1(\O)$. We do not prescribe an explicit tangential boundary condition for $\pfrak$. In fact, converting the variational adjoint into a strong system requires spatial integration by parts in the terms containing $\Upsilon_\alpha(\pfrak)$, which generates additional $\alpha$-dependent boundary traces. The vorticity term in the cost functional may generate a further tangential boundary contribution. Consequently, the corresponding strong boundary operator cannot be identified from the available regularity and the variational formulation alone.
\end{remark}

	\begin{lemma}\label{lem:adjoint}
	Assume $(\theta_0,\u_0)\in H^1(\O)\times\Vb^3(\O)$ and $\g\in\Uad$. Then
	\eqref{eq:adjoint-scalar}-\eqref{weak:adjoint:p} has a unique solution
	satisfying $(\varrho,\pfrak) \in L^\infty(0,T_\xi;H^1(\O))\times L^\infty(0,T_\xi;\Vb^1(\O))$ with $(\partial_t\varrho, \partial_t(\mathcal M_{\alpha\beta}\pfrak)) \in L^\infty(0,T_\xi;L^2(\O)) \times L^2(0,T_\xi;\Wb') $.
\end{lemma}

\begin{proof}
	Notice  that $(\pfrak,\varrho)$ is the solution of \eqref{eq:adjoint-scalar}-\eqref{weak:adjoint:p} if and only if $(\hat\pfrak(t),\hat\varrho(t))=(\pfrak(T_{\xi}-t),\varrho(T_\xi - t))$ is the solution of the following system with $\hat\u(t)= \u(T_{\xi}-t)$ and $\hat\theta (t)=\theta(T_\xi -t)$:
	\begin{equation}\label{eq:adjoint:hat}
		\left\{
		\begin{aligned}
		\partial_t\hat\varrho-\hat\u\cdot\nabla\hat\varrho
			& = \xi(\hat\pfrak\cdot e_2)
			&&\text{in }(0,T_\xi)\times\O,\\
			\hat\varrho(0) =\Lambda^{-2}\theta(T_\xi), \qquad \hat\pfrak(0)& =\boldsymbol{0},
			&&\text{in }\O,
		\end{aligned}
		\right.
	\end{equation}
	and the variational form 
	\begin{equation}\label{weak:adjoint:p:hat}
		\left\langle\partial_t(\mathcal M_{\alpha\beta}\hat\pfrak),
		\boldsymbol\phi\right\rangle_{\Wb'\times\Wb}
		+\mathcal A_{\hat\u}(\hat\pfrak,\boldsymbol\phi)
		=(\hat\theta\nabla\hat\varrho,\boldsymbol\phi)
		-\zeta(\curl\hat\u,\curl\boldsymbol\phi),
	\end{equation}
	for every $\boldsymbol\phi\in\Wb$ and almost every $t\in(0,T_\xi)$.
	
		Galerkin solutions are constructed with the bases introduced in
	Subsection~\ref{Galerkin}.  Testing \eqref{weak:adjoint:p:hat} by 	$\hat\pfrak$ at the Galerkin level and using $\eqref{trilinear1}_2$ and $\eqref{curl:estimates}_3$ gives
	 \begin{align}\label{eq:pfrak-H1-est}
		& \frac{1}{2} \frac{d}{dt} \|\hat\pfrak\|^2_{\Vb^1(\O)} + 2\nu\|\Db(\hat\pfrak)\|^2_{L^2(\O)} + \beta\|\hat\pfrak\cdot\tau\|^2_{L^2(\Gamma)} 
		\nonumber\\
		& = - (\curl (\Upsilon_\alpha(\hat\pfrak)) \times \hat\u, \hat\pfrak)+ (\hat\theta\nabla\hat\varrho,\hat\pfrak)
		-\zeta(\curl\hat\u,\curl\hat\pfrak)
		\nonumber\\
		& \leq C \|\hat\u\|_{H^3(\O)} \|\hat\pfrak\|_{H^1(\O)}^2  +  \|\hat \theta\|_{L^4(\O)}\|\nabla\hat\varrho\|_{L^2(\O)} \|\hat\pfrak\|_{L^4(\O)} + \zeta \|\curl\hat\u\|_{L^2(\O)}\|\curl\hat\pfrak\|_{L^2(\O)}
		\nonumber\\
		& \leq C[1+ \|\hat\u\|_{H^3(\O)}] \|\hat\pfrak\|_{\Vb^1(\O)}^2  + C \|\hat \theta\|_{H^1(\O)}^2\|\nabla\hat\varrho\|_{L^2(\O)}^2   + C \zeta^2 \|\curl\hat\u\|_{L^2(\O)}^2.
	\end{align}
	From $\eqref{eq:adjoint:hat}_1$, we get 
	\begin{align}\label{eq:eta-L2-est}
		\frac12\frac{d}{d t}\|\hat\varrho\|_{L^2(\O)}^2 &=\xi(\hat\pfrak\cdot e_2,\hat\varrho) \leq C\|\hat\varrho\|^2_{L^2(\O)} + C \xi \|\hat\pfrak\|^2_{L^2(\O)}. 
	\end{align} 
	Differentiating $\eqref{eq:adjoint:hat}_1$ and taking inner product with $\nabla\varrho$ gives 
	\begin{align} 
		\frac{d}{d t}\|\nabla\hat\varrho\|_{L^2(\O)}^2 \leq C(1+\|\nabla \hat\u\|_{L^\infty(\O)}) \|\nabla\hat\varrho\|_{L^2(\O)}^2 +C\xi\|\nabla \hat\pfrak\|_{L^2(\O)}^2. \label{eq:eta-gradient-est} 
	\end{align} 
	Combining \eqref{eq:eta-L2-est} and \eqref{eq:eta-gradient-est} gives
	\begin{align} 
		\frac{d}{d t}\|\hat\varrho\|_{H^1(\O)}^2 \leq C(1+\|\hat\u\|_{H^3(\O)})\|\hat\varrho\|_{H^1(\O)}^2 + C\xi\|\hat\pfrak\|_{\Vb^1(\O)}^2. \label{eq:eta-H1-est}
	\end{align}
	Adding \eqref{eq:pfrak-H1-est} and \eqref{eq:eta-H1-est} provides
	\begin{align}\label{eq:pfrak:varrho-H1-est}
		&  \frac{d}{dt} [\|\hat\pfrak\|^2_{\Vb^1(\O)} + \|\hat\varrho\|_{H^1(\O)}^2 ] 
		\leq C[1+ \|\hat\u\|_{H^3(\O)} + \|\hat \theta\|_{H^1(\O)}^2] [\|\hat\pfrak\|^2_{\Vb^1(\O)} + \|\hat\varrho\|_{H^1(\O)}^2 ]     + C \|\curl\hat\u\|_{L^2(\O)}^2.
	\end{align}
	Making use of Gronwall's inequality, we obtain $(\hat\varrho , \hat\pfrak)\in L^\infty(0,T_\xi;H^1(\O)) \times L^\infty(0,T_\xi;\Vb^1(\O))$.

	It remains to estimate the weak time derivatives.  From the definition of
	$\mathcal A_{\hat\u}$, the two-dimensional Sobolev embeddings, the trace
	theorem, and the trilinear estimates, for every $\boldsymbol\phi\in\Wb$ we
	have
	\begin{align*}
		|\mathcal A_{\hat\u}(\hat\pfrak,\boldsymbol\phi)|
		&\leq C\bigl(1+\|\hat\u\|_{H^3(\O)}\bigr)
		\|\hat\pfrak\|_{\Vb^1(\O)}\|\boldsymbol\phi\|_{\Wb},\\
		|(\hat\theta\nabla\hat\varrho,\boldsymbol\phi)|
		&\leq C\|\hat\theta\|_{H^1(\O)}
		\|\hat\varrho\|_{H^1(\O)}\|\boldsymbol\phi\|_{\Wb},\\
		|(\curl\hat\u,\curl\boldsymbol\phi)|
		&\leq C\|\hat\u\|_{H^1(\O)}\|\boldsymbol\phi\|_{\Wb}.
	\end{align*}
	Consequently,
	\begin{align*}
		\|\partial_t(\mathcal M_{\alpha\beta}\hat\pfrak)\|_{\Wb'}
		\leq C\Bigl[&
		\bigl(1+\|\hat\u\|_{H^3}\bigr)\|\hat\pfrak\|_{\Vb^1}
		+\|\hat\theta\|_{H^1}\|\hat\varrho\|_{H^1}
		+\zeta\|\hat\u\|_{H^1}\Bigr].
	\end{align*}
	The scalar equation also gives
	\[
	\|\partial_t\hat\varrho\|_{L^2(\O)}
	\leq C\|\hat\u\|_{H^3(\O)}\|\hat\varrho\|_{H^1(\O)}
	+\xi\|\hat\pfrak\|_{L^2(\O)}.
	\]
	The preceding uniform Galerkin bounds therefore yield
	\[
	\partial_t(\mathcal M_{\alpha\beta}\hat\pfrak)
	\in L^2(0,T_\xi;\Wb'),
	\qquad
	\partial_t\hat\varrho\in L^\infty(0,T_\xi;L^2(\O)).
	\]
	Weak and weak-star compactness now permit passage to the limit in the
	Galerkin equations.  The initial conditions of the time-reversed problem are
	preserved by the corresponding weak continuity, and reversing time gives a
	solution of \eqref{eq:adjoint-scalar}-\eqref{weak:adjoint:p}.


  We next prove uniqueness.  Let $(\varrho_i,\pfrak_i)$, $i=1,2$, be two weak
	solutions. Then, the pair $(\mathbf{P}:=\pfrak_1-\pfrak_2, \mathbf{S}:=\varrho_1-\varrho_2)$ satisfies:
\begin{equation}\label{weak:adjoint:S}
	\left\{\begin{aligned}
		-\partial_t\mathbf{S} - \u\cdot\nabla\mathbf{S}
		&= \xi(\mathbf{P} \cdot e_2)  &&\text{in }(0,T_{\xi})\times\O,\\
			\mathbf{S}(T_\xi) =0, \qquad 	\mathbf{P}(T_\xi)& =\boldsymbol{0} &&\text{in }\O,
	\end{aligned}
	\right.
\end{equation}
and weak formulation
	\begin{equation}\label{eq:adjoint:P}
	\left\langle-\partial_t(\mathcal M_{\alpha\beta}\mathbf{P}),
	\boldsymbol\phi\right\rangle_{\Wb'\times\Wb}
	+\mathcal A_{\u}(\mathbf{P},\boldsymbol\phi)
	=(\theta\nabla\mathbf{S},\boldsymbol\phi),
\end{equation}
for every $\boldsymbol\phi\in\Wb$ and almost every $t\in(0,T_\xi)$.

A direct energy argument cannot be applied to
\eqref{eq:adjoint:P}. Indeed, at the natural level of regularity we only have
\[
\mathbf P\in L^\infty(0,T_\xi;\Vb^1(\Omega)),
\qquad
\partial_t(\mathcal M_{\alpha\beta}\mathbf P)
\in L^2(0,T_\xi;\Wb'),
\]
whereas the variational identity \eqref{eq:adjoint:P} is defined for test
functions belonging to \(\Wb\). Since it is not known that
\(\mathbf P\in\Wb\), the choice
\(\boldsymbol\phi=\mathbf P\) is not admissible. Notice that testing by the
adjoint Galerkin approximation is legitimate for deriving the existence
estimates, but it does not by itself establish uniqueness among all weak
solutions. Uniqueness is therefore proven via a dual problem, using arguments similar to those in \cite{LionsMasmoudi2001}. For
arbitrary distributed sources \(f\) and \(F\), we introduce an auxiliary
forward linear system whose velocity component belongs to \(\Wb\). This
velocity can be used as a test function in \eqref{eq:adjoint:P}, and the
resulting duality identity will imply that both \(\mathbf S\) and
\(\mathbf P\) vanish.

Let us choose $f\in L^2(0,T_\xi;L^2(\O))$ and $F\in L^2(0,T_\xi;\Vb^0(\O))$ be arbitrary and fix it. Consider the pair $(\Theta_{f},\mathbf{Z}_{F})$ as the unique  solution of the following system:
\begin{equation}\label{eq:linearized:uni}
	\left\{\begin{aligned}
		\partial_t\Theta_{f}+\u\cdot\nabla\Theta_{f}
		+\mathbf{Z}_{F}\cdot\nabla\theta&= f &&\text{in }(0,T_{\xi})\times\O,\\
		\partial_t\Upsilon_\alpha(\mathbf{Z}_{F})-\nu\Delta\mathbf{Z}_{F}
		+ \curl\Upsilon_\alpha(\mathbf{Z}_{F})\times\u
		+\curl\Upsilon_\alpha(\u)\times\mathbf{Z}_{F}+\nabla \hat q
		& = \xi\Theta_{f} e_2 + F &&\text{in }(0,T_{\xi})\times\O,\\
		\diver\mathbf{Z}_{F}&=0, &&\text{in }(0,T_{\xi})\times\O,\\
	\Theta_{f}(0) =0,\qquad	\mathbf{Z}_{F}(0)& =\boldsymbol{0}, &&\text{in }\O,\\
		\mathbf{Z}_{F}\cdot\n  =0 \;\;\; \text{ and } \;\;\; [2\nu \n \cdot \Db(\mathbf{Z}_{F})  + \beta \mathbf{Z}_{F}]\cdot\tau &  = 0, &&\text{in }(0,T_{\xi})\times\Gamma.
	\end{aligned}
	\right.
\end{equation}
\noindent
The well-posedness of \eqref{eq:linearized:uni} follows by repeating the
Galerkin argument used in the proof of Lemma~\ref{lem:linearized}, with
homogeneous boundary data and with the additional distributed sources \(f\)
and \(F\). More precisely, testing the scalar equation by \(\Theta_f\), the
velocity equation by \(\mathbf Z_F\), and its projected generalized-vorticity
form by $\mathbf Y_F:=\mathbb P\Upsilon_\alpha(\mathbf Z_F),$ produces the same estimates as in the proof of Lemma~\ref{lem:linearized}. The additional terms are controlled by
\[
|(f,\Theta_f)|
\leq \frac12\|f\|_{L^2(\Omega)}^2
+\frac12\|\Theta_f\|_{L^2(\Omega)}^2 \;\; \text{ and } \;\;
|(F,\mathbf Y_F)|
\leq \frac12\|F\|_{L^2(\Omega)}^2
+\frac12\|\mathbf Y_F\|_{L^2(\Omega)}^2.
\]
Consequently, Gronwall's inequality gives
\[
\begin{aligned}
	&\|\Theta_f\|_{L^\infty(0,T_\xi;L^2(\Omega))}^2
	+\|\mathbf Z_F\|_{L^\infty(0,T_\xi;\Wb)}^2
	+\|\partial_t\Theta_f\|_{L^2(0,T_\xi;(H^1(\Omega))')}^2
	+\|\partial_t\mathbf Z_F\|_{L^2(0,T_\xi;\Vb^1(\Omega))}^2
	\\
	& \leq
	C\left(
	\|f\|_{L^2(0,T_\xi;L^2(\Omega))}^2
	+\|F\|_{L^2(0,T_\xi;\Vb^0(\Omega))}^2
	\right),
\end{aligned}
\]
where \(C>0\) depends only on the fixed state \((\theta,\u)\), the
coefficients of the system, \(\Omega\), and \(T_\xi\), and is independent of
\(f\) and \(F\). These uniform estimates permit passage to the limit in the
Galerkin equations and yield a solution of \eqref{eq:linearized:uni}.
Applying the same estimate to the difference of two solutions proves
uniqueness. In particular, \(\mathbf Z_F(t)\in\Wb\) for almost every
\(t\in(0,T_\xi)\), so that \(\mathbf Z_F\) is an admissible test function in
the weak adjoint equation \eqref{eq:adjoint:P}.

The following identities are first verified for the common Galerkin
approximations and then passed to the limit.  Equivalently, the time
integration by parts follows from the Lions--Magenes lemma for the relevant
Gelfand triples.  The initial and terminal data eliminate all endpoint terms.
Using $\diver\u=\diver\mathbf Z_F=0$ and the zero normal traces, the scalar
equations give
\begin{align}
	\int_0^{T_\xi}(f,\mathbf S)\,dt
	&=\xi\int_0^{T_\xi}(\Theta_f,\mathbf P\cdot e_2)\,dt
	-\int_0^{T_\xi}(\theta\nabla\mathbf S,\mathbf Z_F)\,dt.
	\label{eq:transpose-scalar}
\end{align}
The forward velocity equation and the weak adjoint equation, tested through
the Galerkin limit, give
\begin{align}
	\int_0^{T_\xi}(F,\mathbf P)\,dt
	&=-\xi\int_0^{T_\xi}(\Theta_f,\mathbf P\cdot e_2)\,dt
	+\int_0^{T_\xi}(\theta\nabla\mathbf S,\mathbf Z_F)\,dt.
	\label{eq:transpose-velocity}
\end{align}
Adding \eqref{eq:transpose-scalar} and
\eqref{eq:transpose-velocity} yields
\[
\int_0^{T_\xi}(f,\mathbf S)\,dt
+\int_0^{T_\xi}(F,\mathbf P)\,dt=0.
\]
Since $f$ and $F$ are arbitrary, first take $(f,F)=(\mathbf S,0)$ and then
$(f,F)=(0,\mathbf P)$.  These choices are admissible because
$\mathbf S\in L^2(0,T_\xi;L^2)$ and
$\mathbf P\in L^2(0,T_\xi;\Vb^0)$.  Hence
$\mathbf S=0$ and $\mathbf P=\boldsymbol0$, which proves uniqueness.
\end{proof}

\subsection{Duality relation}\label{subsec:duality}
Let $(\rho,\z)$ solve \eqref{eq:linearized} in a direction
$\h\in\Ucal_0$, set $\Hbf=\mathcal R\h$ and $\w=\z-\Hbf$, and let
$(\varrho,\pfrak)$ solve the adjoint system.  Let us define 
\begin{align}\label{eq:lift-forcing}
	\mathcal F_{\u}(\Kbf)
	&:={\partial_t\Upsilon_\alpha(\Kbf)}-\nu\Delta\Kbf
	+\curl\Upsilon_\alpha(\Kbf)\times\u
	+\curl\Upsilon_\alpha(\u)\times\Kbf.
\end{align}

\begin{lemma}[Duality]\label{lem:duality}
	Let $(\rho,\z)$ solve the linearized system in the direction
	$\h\in\Ucal_0$, set $\Hbf=\mathcal R\h$ and $\w=\z-\Hbf$, and let
	$(\varrho,\pfrak)$ solve the adjoint system. Then
	\begin{align}
		(\Lambda^{-2}\theta(T_\xi),\rho(T_\xi))
		&=-\int_0^{T_\xi}(\Hbf\cdot\nabla\theta,\varrho)\,dt
		+\zeta\int_0^{T_\xi}(\curl\u,\curl\w)\,dt
		-\int_0^{T_\xi}
		(\mathcal F_{\u}(\Hbf),\pfrak)\,dt.
		\label{Duality:relation}
	\end{align}
\end{lemma}

\begin{proof}
	At the Galerkin level, test the scalar linearized equation by $\varrho$,
	the scalar adjoint equation by $\rho$, the lifted linearized momentum
	equation by $\pfrak$, and the variational adjoint equation by $\w$. Add the
	four identities and integrate in time. The scalar transport terms cancel by
	incompressibility and zero normal trace, while the velocity terms cancel by
	\eqref{trilinear1}. Since
	\[
	\rho(0)=0,\quad \w(0)=0,\quad \pfrak(T_\xi)=0,
	\quad
	\varrho(T_\xi)=\Lambda^{-2}\theta(T_\xi),
	\]
	the only scalar endpoint term is
	$(\Lambda^{-2}\theta(T_\xi),\rho(T_\xi))$. Finally,
	$\Hbf(0)=\mathcal R\h(0)=0$. Passing to the Galerkin limit yields
	\eqref{Duality:relation}. This completes the proof.
\end{proof}

\section{First-order necessary optimality conditions}%
\setcounter{equation}{0}\label{sec:first-order}

In this section, we derive the first-order necessary optimality conditions.  For a lifting $\Hbf=\mathcal R\h$, define
\begin{align}
	\mathcal D(\Hbf,\theta,\u,\varrho,\pfrak)
	&:=
	\int_0^{T_\xi}(\Hbf\cdot\nabla\theta,\varrho)\,dt
	+\zeta\int_0^{T_\xi}(\curl\u,\curl\Hbf)\,dt
 	+\int_0^{T_\xi}(\mathcal F_{\u}(\Hbf),\pfrak)\,dt,
	\label{eq:boundary-functional}
\end{align}
where $\mathcal F_{\u}$ is given by \eqref{eq:lift-forcing}.

	\begin{theorem}[First-order necessary condition]
	\label{thm:first-order}
	Let $\g^*\in\Uad$ be an optimal control of control problem \eqref{eqn:control:problem}, let
	$(\theta^*,\u^*)=(\theta(\g^*),\u(\g^*))$, and let
	$(\varrho^*,\pfrak^*)$ be the corresponding adjoint solution. Then,
	\begin{equation}\label{eq:optimality-inequality}
		\gamma(\g^*,\g-\g^*)_{\Ucal}
		-\mathcal D\bigl(\mathcal R(\g-\g^*),
		\theta^*,\u^*,\varrho^*,\pfrak^*\bigr)
		\geq0,
	\end{equation}
for every $\g\in\Uad.$
\end{theorem}

\begin{proof}
	Fix $\g\in\Uad$ and put $\h=\g-\g^*$. Convexity of $\Uad$ implies
	$\g^*+r\h\in\Uad$ for $0\leq r\leq1$. Let $(\rho,\z)$ be the solution of
	the linearized system at $\g^*$ in the direction $\h$.  Consider 
	\begin{align*}
		\frac{\J(\g^{\ast}+r\h) - \J(\g^{\ast})}{r}
		& = \frac{r}{2} \left\| \frac{\theta(T_{\xi};\g^{\ast}+r\h) - \theta(T_{\xi};\g^{\ast})}{r} \right\|_{(H^1(\O))'}^2 + \left( \Lambda^{-2}\theta(T_{\xi};\g^{\ast}) ,  \rho(T_{\xi};\h) \right)
		\nonumber\\
		& \quad + \left( \Lambda^{-1}\theta(T_{\xi};\g^{\ast}) , \frac{\Lambda^{-1}[\theta(T_{\xi};\g^{\ast}+r\h)-\theta(T_{\xi};\g^{\ast})]}{r} - \Lambda^{-1}\rho(T_{\xi};\h) \right)
		\nonumber\\
		& \quad + \gamma (\g^{\ast},\h)_{\Ucal} + \frac{r\gamma}{2}\|\h\|^2_{\Ucal} - \frac{\zeta r}{2} \int_0^{T_{\xi}}\left\|\frac{\curl\u(t;\g^{\ast}+r\h) - \curl\u(t;\g^{\ast})}{r}\right\|_{L^2(\O)}^2 d t 
		\nonumber\\ 
		& \quad - \zeta  \int_0^{T_{\xi}}\left(\curl\u(t;\g^{\ast}) , \frac{\curl\u(t;\g^{\ast}+r\h) - \curl\u(t;\g^{\ast})}{r} - \curl\z(t,\h)\right) d t 
		\nonumber\\ 
		& \quad - \zeta  \int_0^{T_{\xi}}\left(\curl\u(t;\g^{\ast}) , \curl\z(t;\h)\right) d t,
	\end{align*}
which follows from the Lipschitz-stability lemma (Lemma \ref{thm:stability}) and the directional differentiability theorem (Theorem \ref{thm:Gateaux}) that
	\begin{align}\label{eqn:93}
		\J^\prime(\g^{\ast})[\h] & =   \left( \Lambda^{-2}\theta(T_{\xi};\g^{\ast}) ,  \rho(T_{\xi};\h) \right)
		+ \gamma (\g^{\ast},\h)_{\Ucal}
		- \zeta  \int_0^{T_{\xi}}\left(\curl\u(t;\g^{\ast}) , \curl\z(t;\h)\right) d t.
	\end{align}
Set $\Hbf=\mathcal R\h$ and $\w=\z-\Hbf$. Substituting the duality identity
\eqref{Duality:relation} into \eqref{eqn:93} and using
$\z=\w+\Hbf$ gives
\begin{equation}\label{eq:reduced-gradient}
	\J'(\g^*)[\h]
	=\gamma(\g^*,\h)_{\Ucal}
	-\mathcal D(\Hbf,\theta^*,\u^*,\varrho^*,\pfrak^*).
\end{equation}
Optimality of $\g^*$ implies
$\J'(\g^*)[\g-\g^*]\geq0$, and \eqref{eq:optimality-inequality} follows. This completes the proof.
\end{proof}

\section{Uniqueness of optimal solution of \eqref{eqn:control:problem}}%
\setcounter{equation}{0}\label{sec:uniqueness}

In this section, we establish that, for sufficiently large $\gamma>0$, the optimal control problem admits a unique solution. To prove the uniqueness of the optimal solution, we require the adjoint variable $\pfrak$ to possess $H^2$-regularity. Therefore, throughout this section, we assume that $\beta=\zeta=0$.

Throughout this section, $C_M$ denotes a constant depending only on $\O, T_\xi, \alpha, \nu, \xi, M, \theta_0, \u_0,$ and on the fixed compatibility datum $\g_{\rm in}$. It is independent of
$\gamma$ and of the particular controls in $\Uad$. This uniformity follows
from \eqref{eq:uniform-lifting} and from the fact that the state and adjoint
equations do not contain $\gamma$.

The following lemma is useful to write the strong form of the adjoint system:
\begin{lemma}[{\cite[Lemma 5.1]{AradaCipriano2015}}]\label{lem:for:adjoint}
	Let $\u,\v\in \wiWb$, $\w\in\Wb$ and $\beta=0$. Then
	\begin{align*}
		(\curl \Upsilon_\alpha(\u\times\v ), \w ) = b(\v,\u,\Upsilon_\alpha(\w)) - b(\u,\v,\Upsilon_\alpha(\w) ).
	\end{align*}
\end{lemma}

In view of Lemma \ref{lem:for:adjoint}, we have the following form of adjoint system:
 \begin{equation}\label{adjoint:regular}
 	\left\{\begin{aligned}
 		-\partial_t\varrho - \u\cdot\nabla\varrho
 		&= \xi(\pfrak \cdot e_2)  &&\text{in }(0,T_{\xi})\times\O,\\
 		- \partial_t\Upsilon_\alpha(\pfrak)  - \nu\Delta\pfrak  - \curl\Upsilon_\alpha(\u)\times\pfrak  + \curl \Upsilon_\alpha(\u\times\pfrak)
 		+ \nabla \pi
 		& = \theta\nabla\varrho,
 		&&   \text{in }(0,T_{\xi})\times\O,\\
 		\diver\pfrak&=0, &&\text{in }(0,T_{\xi})\times\O,\\
 	\varrho(T_\xi) =\Lambda^{-2}\theta(T_{\xi}),\qquad	\pfrak(T_\xi)& =\boldsymbol{0}, &&\text{in }\O,\\
 				\pfrak\cdot\n  =0 \;\; \text{ and }  \;\; [\n \cdot \Db(\pfrak)  ]\cdot\tau   & = 0, &&\text{on }(0,T_{\xi})\times\Gamma.
 	\end{aligned}
 	\right.
 \end{equation}

%
%

In the next result, we show the regularity estimates for the solution of the adjoint system \eqref{adjoint:regular} which are helpful to obtain the uniqueness of optimal solution of our control problem.
\begin{lemma}\label{thm:adjoint:regularity}
	Assume that $\beta=0$, $(\theta_0, \u_0)\in H^{1}(\Omega)\times \Vb^3(\O)$ and $\g\in\Uad$. Then, there exists a unique solution pair $(\varrho , \pfrak)\in L^\infty(0,T_\xi;H^2(\O)) \times L^\infty(0,T_\xi;\Wb)$ with $(\partial_t\varrho , \partial_t\pfrak)\in L^\infty(0,T_\xi;H^1(\O)) \times L^2(0,T_\xi;\Vb^1(\O))$ satisfying the adjoint system \eqref{adjoint:regular}.
\end{lemma}

\begin{proof}
	Recall that there exists a pair  $(\varrho , \pfrak)\in L^\infty(0,T_\xi;H^1(\O)) \times L^\infty(0,T_\xi;\Vb^1(\O))$ with $(\partial_t\varrho , \partial_t\mathcal{M}_{\alpha\beta} \pfrak)\in L^\infty(0,T_\xi;L^2(\O)) \times L^2(0,T_\xi;\Wb')$ which satisfies system \eqref{adjoint:regular}. We provide only the energy estimates for the additional regularity that can be established at the level of the Faedo-Galerkin approximation.

	We recall the notation $(\hat\pfrak,\hat\varrho,\hat\u,\hat\theta)(t) =(\pfrak,\varrho,\u,\theta)(T_\xi -t)$.  Note that $\theta(T_\xi)\in H^1(\O)$ implies $\hat\varrho(0)\in H^3(\O)$.
	
	For every multi-index \(\gamma\) with \(|\gamma|\le2\), applying
	\(D^\gamma\) to the first equation in \eqref{eq:adjoint:hat} gives
	\begin{align*}
		\partial_tD^\gamma \hat\varrho - \u\cdot\nabla D^\gamma \hat\varrho
		= [D^\gamma,\u\cdot\nabla]\hat\varrho
		+ \xi D^\gamma(\hat\pfrak\cdot\boldsymbol e_2).
	\end{align*}
	The transport term again vanishes after pairing with \(D^\gamma \hat\varrho\).
	The two-dimensional commutator estimate
	\begin{equation}\label{eq:commutator}
		\sum_{|\gamma|\leq2}
		\|{[D^\gamma,\u\cdot\nabla]\hat\varrho}\|_{L^2(\O)}
		\leq C \|{\u}\|_{H^3(\O)} \|{\hat\varrho}\|_{H^2(\O)}
	\end{equation}
	follows directly by expanding the derivatives.  For example, the highest
	terms are bounded by
	\begin{align*}
		\|\nabla\u\|_{L^\infty(\O)} \|{D^2\hat\varrho}\|_{L^2(\O)}
		+ \|{D^2\u}\|_{L^4(\O)} \|{\nabla \hat\varrho}\|_{L^4(\O)}
		\leq C\|{\u}\|_{H^3(\O)} \|{\hat\varrho}\|_{H^2(\O)}.
	\end{align*}
	Consequently,
	\begin{equation}\label{eq:r-H2}
		\frac{d}{dt}\|{\hat\varrho}\|_{H^2(\O)}^2
		\leq C (1+\|{\u}\|_{H^3(\O)})\|\hat\varrho\|_{H^2(\O)}^2
		+ C\xi\|\hat\pfrak\|_{H^2(\O)}^2.
	\end{equation}

At the Galerkin level, apply the generalized-vorticity estimate used in the
proof of \cite[Proposition~5.4, p.24]{AradaCipriano2015}.  That calculation only
requires the distributed right-hand side of the velocity adjoint to belong to
$L^2(\O)$.  In the present system this right-hand side is
$\hat\theta\nabla\hat\varrho$, and the two-dimensional embedding
$H^1(\O)\hookrightarrow L^4(\O)$ gives
\[
\|\hat\theta\nabla\hat\varrho\|_{L^2(\O)}
\leq C\|\hat\theta\|_{H^1(\O)}
\|\hat\varrho\|_{H^2(\O)}.
\]
Consequently, the same calculation yields
\begin{align}
	\frac{d}{dt}\|\hat\pfrak\|_{\Wb}^2
	\leq C\bigl(1+\|\hat\u\|_{H^3(\O)}\bigr)
	\|\hat\pfrak\|_{\Wb}^2
	+C\|\hat\theta\|_{H^1(\O)}^2
	\|\hat\varrho\|_{H^2(\O)}^2.                         \label{eq:p-H2}
\end{align}
Combining this estimate with \eqref{eq:r-H2}, and using the uniform state
bounds, gives
\[
\frac{d}{dt}\left(
\|\hat\varrho\|_{H^2}^2+\|\hat\pfrak\|_{\Wb}^2\right)
\leq C_M\left(
\|\hat\varrho\|_{H^2}^2+\|\hat\pfrak\|_{\Wb}^2\right).
\]
Since $\hat\varrho(0)=\Lambda^{-2}\theta(T_\xi)\in H^3(\O)$ and
$\hat\pfrak(0)=\boldsymbol{0}$, Gronwall's inequality gives
\[
(\hat\varrho,\hat\pfrak)
\in L^\infty(0,T_\xi;H^2(\O))
\times L^\infty(0,T_\xi;\Wb).
\]

From \eqref{eq:adjoint:hat}, we have 
\begin{align}
	\|\partial_t\nabla\hat\varrho\|_{L^2(\O)} & \leq \|\hat\u\|_{L^{\infty}(\O)} \|\hat\varrho\|_{H^2(\O)} + \|\nabla\hat\u\|_{L^{\infty}(\O)}\|\nabla\hat\varrho\|_{L^2(\O)} + \xi \|\nabla\hat\pfrak\|_{L^2(\O)}
	\nonumber\\ 
	& \leq  C \|\hat\u\|_{H^3(\O)} \|\hat\varrho\|_{H^2(\O)} + \xi \|\nabla\hat\pfrak\|_{L^2(\O)},
\end{align}
which implies $\partial_t\hat\varrho \in L^\infty(0,T_\xi;H^1(\O))$. Since $\hat\theta\in L^{\infty}(0,T_\xi;H^1(\O))$ and $\hat\varrho\in L^{\infty}(0,T_\xi;H^2(\O))$, from \cite[Proposition 5.4, p.25]{AradaCipriano2015}, we immediately get that $\partial_t\pfrak \in  L^2(0,T_\xi;\Vb^1(\O))$. This completes the proof.
\end{proof}

Let us restate the coupled optimality system for the case $\beta=\zeta=0$:
\begin{center}
	\fbox{ \scriptsize
		\begin{minipage}[b]{\textwidth}
			\begin{equation}\label{CS-1}\tag{State system (Forward)}
				\left\{
				\begin{aligned}
					\partial_t\theta+ \u \cdot\nabla\theta&=0,
					&&  \text{in }(0,T_{\xi})\times\O,\\
					 \partial_t\Upsilon_\alpha(\u) - \nu\Delta\u
					+\curl\Upsilon_\alpha(\u)\times\u
					+\nabla p
					&= \xi\theta e_2,
					&&\text{in }(0,T_{\xi})\times\O,\\
					\diver\u &=0,
					&&  \text{in }(0,T_{\xi})\times\O,\\
					\theta(0)  = \theta_0, \qquad \u(0) & = \u_0, &&  \text{in }\O,\\
					\u\cdot\n   =0, \;\;\; \text{ and } \;\;\; [ \n \cdot \Db(\u) ]\cdot\tau  & = \g\cdot\tau, &&  \text{on }(0,T_\xi)\times \Gamma,
				\end{aligned}
				\right.
			\end{equation}
			\begin{equation}\label{CS-2}\tag{Adjoint system (Backward)}
				\left\{\begin{aligned}
					-\partial_t\varrho - \u\cdot\nabla\varrho
					&= \xi(\pfrak \cdot e_2)  &&\text{in }(0,T_{\xi})\times\O,\\
					- \partial_t\Upsilon_\alpha(\pfrak)  - \nu\Delta\pfrak   - \curl\Upsilon_\alpha(\u)\times\pfrak   + \curl \Upsilon_\alpha(\u\times\pfrak)
					 + \nabla \pi
					& = \theta\nabla\varrho,
					&&   \text{in }(0,T_{\xi})\times\O,\\
					\diver\pfrak&=0, &&\text{in }(0,T_{\xi})\times\O,\\
				\varrho(T_\xi) =\Lambda^{-2}\theta(T_{\xi}), \qquad	\pfrak(T_\xi)& =\boldsymbol{0}, &&\text{in }\O,\\
					\pfrak\cdot\n   =0 \;\;\; \text{ and } \;\;\; [ \n \cdot \Db(\pfrak)] \cdot\tau  & = 0, &&  \text{on }(0,T_\xi)\times \Gamma,
				\end{aligned}
				\right.
			\end{equation}
			\begin{align}\label{CS-3}\tag{Optimality condition}
			 \gamma (\g,\h-\g)_{\Ucal} -\mathcal{D}(\cR(\h-\g), \theta, \u, \varrho, \pfrak)  \geq 0, \text{ for all } \h\in \Uad,
		\end{align}
	where $\mathcal{D}(\cdot, \cdot, \cdot, \cdot, \cdot)$ is given by \eqref{eq:boundary-functional}.
		\end{minipage}
	}
\end{center}

  Suppose that $\g_1,\g_2$ are two optimal control variables for \eqref{eqn:control:problem} and $(\theta_1,\u_1),(\theta_2,\u_2)$ are the corresponding optimal states with the adjoint states $(\varrho_1,\pfrak_1),(\varrho_2,\pfrak_2)$, respectively.

  Define
  \begin{equation}\label{differences}
  	\widebar\theta :=\theta_1-\theta_2,
  	\quad
  	\widebar\u:=\u_1-\u_2, \quad \widebar{\g}:= \g_1-\g_2, \quad  \widebar{\varrho}:=\varrho_1-\varrho_2, \quad \widebar{\pfrak}:=\pfrak_1-\pfrak_2, \quad \widebar{\Gbf}:=\cR\widebar{\g}.
  \end{equation}

In the following result, we establish a stability estimate for the adjoint system with respect to the control variable.
	\begin{lemma}[Adjoint stability]\label{lem:adjoint-stability}
Assume $\beta=\zeta=0$. Let $\g_1,\g_2\in\Uad$, and let
$(\theta_i,\u_i,\varrho_i,\pfrak_i)$ be the corresponding state and adjoint
variables.  Then 
	\begin{align}\label{varrho:pfrack:dif}
		\|\varrho_1-\varrho_2\|_{L^\infty(0,T_\xi;L^2(\O))}
		+\|\pfrak_1-\pfrak_2\|_{L^\infty(0,T_\xi;\Vb^1(\O))}
		&\leq C_M\|\g_1-\g_2\|_{\Ucal},
	\end{align} 
where $C_M>0$ is a constant independent of $\gamma$ and of the particular controls.
\end{lemma}
\begin{proof}
%

	Let us define
	\begin{align*}
		& (\theta_1^* ,\u_1^* ,\theta_2^* ,\u_2^* ,\varrho_1^* ,\pfrak_1^* ,\varrho_2^* ,\pfrak_2^* , \widebar\theta^* , \widebar\u^* , \widebar{\varrho}^* ,\widebar{\pfrak}^*)(t)
		 := (\theta_1 ,\u_1 ,\theta_2 ,\u_2 ,\varrho_1 ,\pfrak_1 ,\varrho_2 ,\pfrak_2 , 
		  \widebar\theta , \widebar\u , \widebar{\varrho} ,\widebar{\pfrak})(T_\xi - t).
	\end{align*}	
	The pair $(\widebar{\varrho}^*,\widebar{\pfrak}^*)$ satisfies
	\begin{equation} 
		\left\{\begin{aligned}
			\partial_t\widebar\varrho^* - \u_1^*\cdot\nabla\widebar\varrho^* - \widebar{\u}^*\cdot\nabla\varrho_2^* - 
			\xi(\widebar\pfrak^* \cdot e_2) &= 0 &&\text{in }(0,T_{\xi})\times\O,\\
				\widebar\varrho^*(0) =\Lambda^{-2}\widebar\theta^*(0), \qquad \widebar\pfrak^*(0)& =\boldsymbol{0} &&\text{in }\O,
		\end{aligned}
		\right.
	\end{equation}
	and
	\begin{align} 
		&\left\langle\partial_t(\mathcal M_{\alpha\beta}\widebar\pfrak^*),
		\boldsymbol\phi\right\rangle_{\Wb'\times\Wb}
		+ 2\nu(\Db(\widebar\pfrak^*) , \Db(\boldsymbol \phi))  
		  +  b(\widebar\pfrak^* ,\boldsymbol \phi,\Upsilon_\alpha(\u_1^*)) 
		 + b(\pfrak_2^*,\boldsymbol \phi,\Upsilon_\alpha(\widebar\u^*))
		-b(\boldsymbol \phi,\widebar\pfrak^*,\Upsilon_\alpha(\u_1^*)) 
		\nonumber\\
		&  - b(\boldsymbol \phi,\pfrak_2^*,\Upsilon_\alpha(\widebar\u^*) )
		+b(\widebar\pfrak^*,\u_1^*,\Upsilon_\alpha(\boldsymbol \phi)) + b(\pfrak_2^*,\widebar\u^*,\Upsilon_\alpha(\boldsymbol \phi))
		   -b(\widebar\u^*,\pfrak_1^*,\Upsilon_\alpha(\boldsymbol \phi) )  -b(\u_2^*,\widebar\pfrak^*,\Upsilon_\alpha(\boldsymbol \phi))  \nonumber\\
		&=(\widebar\theta\nabla\varrho_1^*,\boldsymbol \phi) + (\theta_2^*\nabla\widebar\varrho^* , \boldsymbol \phi), 
	\end{align}
for every $\boldsymbol\phi\in\Wb$ and almost every $t\in(0,T_\xi)$.	 This provides 
	\begin{align}\label{eqn:varrho:dif}
		\frac12\frac{d}{dt} \|\widebar\varrho^*\|^2_{L^2(\O)} 
		& \leq \|\widebar\u^*\|_{L^{\infty}(\O)} \|\nabla\varrho_2^*\|_{L^2(\O)} \|\widebar\varrho^*\|_{L^2(\O)} + \xi \|\widebar\pfrak^*\|_{L^2(\O)} \|\widebar\varrho^*\|_{L^2(\O)}
		\nonumber\\
		& \leq 
		C [\|\widebar\u^*\|^2_{H^{2}(\O)} +  \|\widebar\pfrak^*\|^2_{L^2(\O)}  + \|\widebar\varrho^*\|^2_{L^2(\O)} ]
	\end{align}
	and 
	\begin{align}\label{eqn:pfrak:dif}
		&  \frac12 \frac{d}{dt} \|\widebar\pfrak^*\|^2_{\Vb^1(\O)} + 2\nu \|\Db(\widebar\pfrak^*)\|^2_{L^2(\O)} 
		\nonumber\\
		& = 
		- b(\pfrak_2^*,\widebar\pfrak^*,\Upsilon_\alpha(\widebar\u^*))
		+ b(\widebar\pfrak^*,\pfrak_2^*,\Upsilon_\alpha(\widebar\u^*))
		-b(\widebar\pfrak^*,\u_1^*,\Upsilon_\alpha(\widebar\pfrak^*)) + b(\pfrak_2^*,\widebar\u^*,\Upsilon_\alpha(\widebar\pfrak^*))
		\nonumber\\
		& \quad +b(\widebar\u^*,\pfrak_1^*,\Upsilon_\alpha(\widebar\pfrak^*))  +b(\u_2^*,\widebar\pfrak^*,\Upsilon_\alpha(\widebar\pfrak^*) )   + (\widebar\theta^*\nabla\varrho_1^*,\widebar\pfrak^*) -(\widebar\varrho^*\nabla\theta_2^*, \widebar\pfrak^*)
		\nonumber\\
		& \leq 
		C_M [ \|\widebar\u^*\|^2_{H^{2}(\O)} +\|\widebar\theta^*\|^2_{L^2(\O)}   +  \|\widebar\pfrak^*\|^2_{\Vb^1(\O)}  + \|\widebar\varrho^*\|^2_{L^2(\O)} ],
	\end{align}
	where we have used the fact that $\pfrak_2^*, \u_1^* , \pfrak_1^*, \u_2^*, \varrho_1^* $ and $\theta_2^*$ belong to $L^{\infty}(0,T_\xi ; H^2(\O))$. Now, combining \eqref{eqn:varrho:dif}-\eqref{eqn:pfrak:dif} and applying Gronwall's inequality, we obtain
	\begin{align*}
		\|\widebar\varrho^*\|^2_{L^{\infty}(0,T_\xi;L^2(\O))} +	\|\widebar\pfrak^*\|^2_{L^{\infty}(0,T_\xi; \Vb^1(\O))} \leq C_M [\|\widebar\theta^*\|^2_{L^{\infty}(0,T_\xi;L^2(\O))} +	\|\widebar\u^*\|^2_{L^{\infty}(0,T_\xi; H^2(\O))} ],
	\end{align*} 
	which, by Lemma \ref{thm:stability}, implies  \eqref{varrho:pfrack:dif}.
	This completes the proof.
\end{proof}

	\begin{theorem}[Uniqueness of the optimal control]
	\label{thm:optimal-uniqueness}
	Assume $\beta=\zeta=0$. There exists
	\[
	\gamma_0=\gamma_0(\O,T_\xi,\alpha,\nu,\xi,M,
	\theta_0,\u_0,\g_{\rm in})>0,
	\]
	independent of $\gamma$ and of the particular optimal controls, such that
	the optimal control is unique whenever $\gamma>\gamma_0$.
\end{theorem}

\begin{proof}
 We see that the pair $(\widebar\theta,\widebar\v:=\widebar\u-\widebar\Gbf)$ satisfies
  \begin{equation}\label{difference:system}
  \small	\left\{
  	\begin{aligned}
  		\partial_t\widebar\theta
  		& +\u_1\cdot\nabla\widebar\theta
  		-\widebar\u\cdot\nabla\widebar\theta + \widebar\u\cdot\nabla\theta_1=0,
  		&&\text{in }(0,T_\xi)\times\O,\\
  		\partial_t\Upsilon_\alpha(\widebar\v ) & -\nu\Delta\widebar\v
  		+\curl\Upsilon_\alpha(\widebar\v)\times\u_1
  		- \curl\Upsilon_\alpha(\widebar\v )\times\widebar\v  + \curl\Upsilon_\alpha(\u_1) \times\widebar\v 
  		\\ + \nabla\widebar p & =\xi\widebar\theta e_2
  		 -\partial_t\Upsilon_\alpha(\widebar\Gbf)+\nu\Delta\widebar\Gbf
  		- \curl\Upsilon_\alpha(\widebar\Gbf) \times \u_1 + \curl\Upsilon_\alpha(\widebar\Gbf) \times\widebar\v
  		\\
  		& \quad +\curl\Upsilon_\alpha(\widebar\v )\times\widebar\Gbf  +\curl\Upsilon_\alpha(\widebar\Gbf) \times\widebar\Gbf - \curl\Upsilon_\alpha(\u_1)\times\widebar\Gbf ,
  		&&\text{in }(0,T_\xi)\times\O,\\
  		\diver\widebar\v&=0,
  		&&\text{in } (0,T_\xi)\times\O,\\
  	\widebar\theta(0) & =0, \qquad	\widebar\v(0) =\boldsymbol{0},
  		&&\text{in }\O,\\
  		\widebar\v\cdot\n& =0,
  		\qquad
  		[\n\cdot\Db(\widebar\v)]
  		\cdot\tau=0,
  		&&\text{on }(0,T_\xi)\times \Gamma,
  	\end{aligned}
  	\right.
  \end{equation}
and the pair $(\varrho_1,\pfrak_1)$ satisfies
\begin{equation}\label{eq:adjoint-scalar:1}
	\left\{
	\begin{aligned}
		- \partial_t\varrho_1-\u_1\cdot\nabla\varrho_1
		&=\xi(\pfrak_1\cdot\e_2)
		&&\text{in }(0,T_\xi)\times\O,\\
		\varrho_1(T_\xi) =\Lambda^{-2}\theta_1(T_\xi), \qquad \pfrak_1(T_\xi)& =\boldsymbol{0},
		&&\text{in }\O,
	\end{aligned}
	\right.
\end{equation}
and the variational form 
\begin{equation}\label{weak:adjoint:p:1}
	\left\langle-\partial_t(\mathcal M_{\alpha\beta}\pfrak_1),
	\boldsymbol\phi\right\rangle_{\Wb'\times\Wb}
	+\mathcal A_{\u_1}(\pfrak_1,\boldsymbol\phi)
	=(\theta_1\nabla\varrho_1,\boldsymbol\phi),
\end{equation}
for every $\boldsymbol\phi\in\Wb$ and almost every $t\in(0,T_\xi)$.

Taking the inner products of equations $\eqref{difference:system}_1$, $\eqref{difference:system}_2$, and $\eqref{eq:adjoint-scalar:1}_1$ with $\varrho_1$, $\pfrak_1$, and $\widebar\theta$, respectively, and choosing $\boldsymbol{\phi}=\widebar\v$ in \eqref{weak:adjoint:p:1}, we obtain
\begin{align}
	(\partial_t\widebar\theta, \varrho_1) & =-
	(\u_1\cdot\nabla\widebar\theta
	-\widebar\u\cdot\nabla\widebar\theta + \widebar\u\cdot\nabla\theta_1, \varrho_1),\label{Uni1}\\
	 \langle\partial_t\Upsilon_\alpha(\widebar\v), \pfrak_1\rangle & = - ( - \nu\Delta\widebar\v
	+\curl\Upsilon_\alpha(\widebar\v)\times\u_1
	- \curl\Upsilon_\alpha(\widebar\v )\times\widebar\v  + \curl\Upsilon_\alpha(\u_1)\times\widebar\v, \pfrak_1) 
	\nonumber \\
	 & \quad  + (\xi\widebar\theta e_2, \pfrak_1)
	- (\partial_t\Upsilon_\alpha(\widebar\Gbf)-\nu\Delta\widebar\Gbf
	+\curl\Upsilon_\alpha(\widebar\Gbf)\times\u_1 - \curl\Upsilon_\alpha(\widebar\Gbf)\times\widebar\v, \pfrak_1)
	\nonumber
	\\
	& \quad 
	+ (\curl\Upsilon_\alpha(\widebar\v)\times\widebar\Gbf  +\curl\Upsilon_\alpha(\widebar\Gbf)\times\widebar\Gbf - \curl\Upsilon_\alpha(\u_1)\times\widebar\Gbf, \pfrak_1), \label{Uni2}\\
	(-\partial_t\varrho_1, \widebar{\theta}) & =  ( \u_1\cdot\nabla\varrho_1, \widebar{\theta}) +  \xi(\pfrak_1 \cdot e_2, \widebar{\theta}), \label{Uni3} \\
	-\left\langle\partial_t(\mathcal M_{\alpha\beta}\pfrak_1),
	\widebar\v\right\rangle_{\Wb'\times\Wb} & =
	- 2\nu(\Db(\pfrak_1),\Db(\widebar\v))   -   b(\pfrak_1,\widebar\v,\Upsilon_\alpha(\u_1))
	+b(\widebar\v,\pfrak_1,\Upsilon_\alpha(\u_1))
	\nonumber\\
	& \quad  -b(\pfrak_1,\u_1,\Upsilon_\alpha(\widebar\v))
	+b(\u_1,\pfrak_1,\Upsilon_\alpha(\widebar\v)) + (\theta_1\nabla\varrho_1,\widebar\v). \label{Uni4}
\end{align}

Now, performing $[\eqref{Uni1}+\eqref{Uni2}]-[\eqref{Uni3}+\eqref{Uni4}]$ yields
\begin{align}\label{eqn:1081}
	& (\partial_t\widebar\theta,\varrho_1)+(\partial_t\varrho_1,\widebar\theta)+ \langle\partial_t\Upsilon_\alpha(\widebar\v), \pfrak_1\rangle +\left\langle\partial_t(\mathcal M_{\alpha\beta}\pfrak_1),
	\widebar\v\right\rangle_{\Wb'\times\Wb} - (\widebar{\u}\cdot\nabla\widebar{\theta}, \varrho_1) - (\curl \Upsilon_\alpha(\widebar{\v})\times\widebar{\v}, \pfrak_1) 
	\nonumber\\
	&
	+ (\curl \Upsilon_\alpha(\widebar{\Gbf})\times\widebar{\v}, \pfrak_1) 
	 - (\curl \Upsilon_\alpha(\widebar{\v})\times\widebar{\Gbf}, \pfrak_1) - (\curl \Upsilon_\alpha(\widebar{\Gbf})\times\widebar{\Gbf}, \pfrak_1) + (\curl \Upsilon_\alpha(\u_1)\times\widebar{\Gbf}, \pfrak_1) 
	\nonumber\\
	& + (\curl \Upsilon_\alpha(\widebar{\Gbf})\times\u_1, \pfrak_1) + (\partial_t\Upsilon_\alpha(\widebar{\Gbf}),\pfrak_1) - \nu (\Delta\widebar{\Gbf},\pfrak_1)  
	\nonumber\\ 
	& = 0.   
\end{align}
Changing the role of $(\theta_1,\u_1,\g_1,\varrho_1,\pfrak_1)$ and $(\theta_2,\u_2,\g_2,\varrho_2,\pfrak_2)$, we get
\begin{align}\label{eqn:1082}
	& -(\partial_t\widebar\theta,\varrho_2)-(\partial_t\varrho_2,\widebar\theta) -  \langle\partial_t\Upsilon_\alpha(\widebar\v), \pfrak_2\rangle - \left\langle\partial_t(\mathcal M_{\alpha\beta}\pfrak_2),
	\widebar\v\right\rangle_{\Wb'\times\Wb} - (\widebar{\u}\cdot\nabla\widebar{\theta}, \varrho_2) - (\curl \Upsilon_\alpha(\widebar{\v})\times\widebar{\v}, \pfrak_2) 
	\nonumber\\
	& 
	+ (\curl \Upsilon_\alpha(\widebar{\Gbf})\times\widebar{\v}, \pfrak_2) 
	 - (\curl \Upsilon_\alpha(\widebar{\v})\times\widebar{\Gbf}, \pfrak_2) - (\curl \Upsilon_\alpha(\widebar{\Gbf})\times\widebar{\Gbf}, \pfrak_2) - (\curl \Upsilon_\alpha(\u_2)\times\widebar{\Gbf}, \pfrak_2) 
	\nonumber\\
	& - (\curl \Upsilon_\alpha(\widebar{\Gbf})\times\u_2, \pfrak_2) - (\partial_t\Upsilon_\alpha(\widebar{\Gbf}),\pfrak_2) + \nu (\Delta\widebar{\Gbf},\pfrak_2)  
	\nonumber\\ 
	& = 0.   
\end{align}

Adding \eqref{eqn:1081} and \eqref{eqn:1082}, we have 
\begin{align}\label{eqn:109}
	& \|\widebar{\theta}(T_\xi)\|^2_{(H^1(\O))^\prime} - \int_{0}^{T_\xi} (\widebar{\u}\cdot\nabla\widebar{\theta}, \varrho_1 + \varrho_2) dt - \int_{0}^{T_\xi} (\curl \Upsilon_\alpha(\widebar{\v})\times\widebar{\v}, \pfrak_1 + \pfrak_2) dt
	\nonumber\\
	&  + \int_{0}^{T_\xi} (\curl \Upsilon_\alpha(\widebar{\Gbf})\times\widebar{\v} - \curl \Upsilon_\alpha(\widebar{\v})\times\widebar{\Gbf} - \curl \Upsilon_\alpha(\widebar{\Gbf})\times\widebar{\Gbf}, \pfrak_1 + \pfrak_2) dt
	\nonumber\\
	& 
	+  \int_{0}^{T_\xi} (\curl \Upsilon_\alpha(\u_1)\times\widebar{\Gbf}, \pfrak_1) dt
	 -  \int_{0}^{T_\xi} [(\curl \Upsilon_\alpha(\u_2)\times\widebar{\Gbf}, \pfrak_2) + (\curl \Upsilon_\alpha(\widebar{\Gbf})\times\u_1, \pfrak_1)] dt
	\nonumber\\
	& - \int_{0}^{T_\xi} (\curl \Upsilon_\alpha(\widebar{\Gbf})\times\u_2, \pfrak_2) dt + \int_{0}^{T_\xi} (\partial_t\Upsilon_\alpha(\widebar{\Gbf}), \widebar\pfrak) dt  - \int_{0}^{T_\xi} \nu (\Delta\widebar{\Gbf},\widebar\pfrak)  dt 
	\nonumber\\
	& = 0.
\end{align}

From \eqref{eqn:93}, one can obtain 
\begin{align}\label{eqn:1014}
	\gamma \|\widebar{\g}\|^2_{\Ucal} \leq \mathcal{D}(\widebar\Gbf, \theta_1, \u_1, \varrho_1, \pfrak_1) - \mathcal{D}(\widebar\Gbf, \theta_2, \u_2, \varrho_2, \pfrak_2).
\end{align}
Adding \eqref{eqn:109} and \eqref{eqn:1014} and using integration by parts, we get
\begin{align}
& \|\widebar{\theta}(T_\xi)\|^2_{(H^1(\O))^\prime} + 	\gamma \|\widebar{\g}\|^2_{\Ucal}  
\nonumber\\
& \leq 
-  \int_{0}^{T_\xi} (\widebar{\u}\cdot\nabla(\varrho_1 + \varrho_2) , \widebar{\theta}) dt + \int_{0}^{T_\xi} b(\pfrak_1 + \pfrak_2 ,\widebar{\v} , \Upsilon_\alpha(\widebar{\v})) dt  
 - \int_{0}^{T_\xi} b(\widebar{\v}, \pfrak_1 + \pfrak_2 , \Upsilon_\alpha(\widebar{\v})) dt 
 \nonumber\\ 
 & \quad 
 -  \int_{0}^{T_\xi} b(\pfrak_1 + \pfrak_2 ,\widebar{\v} , \Upsilon_\alpha(\widebar{\Gbf})) dt  
 + \int_{0}^{T_\xi} b(\widebar{\v}, \pfrak_1 + \pfrak_2 , \Upsilon_\alpha(\widebar{\Gbf})) dt 
 + \int_{0}^{T_\xi} b(\pfrak_1 + \pfrak_2 ,\widebar{\Gbf} , \Upsilon_\alpha(\widebar{\v})) dt  
\nonumber\\
& \quad - \int_{0}^{T_\xi} b(\widebar{\Gbf}, \pfrak_1 + \pfrak_2 , \Upsilon_\alpha(\widebar{\v})) dt 
+  \int_{0}^{T_\xi} b(\pfrak_1 + \pfrak_2 ,\widebar{\v} , \Upsilon_\alpha(\widebar{\Gbf})) dt  
 - \int_{0}^{T_\xi} b(\widebar{\Gbf}, \pfrak_1 + \pfrak_2 , \Upsilon_\alpha(\widebar{\Gbf})) dt 
\nonumber\\
& \quad 
+ \int_{0}^{T_\xi} [-(\widebar{\Gbf}\cdot\nabla \varrho_1, \widebar{\theta})  + (\widebar{\Gbf}\cdot\nabla\theta_2, \widebar\varrho)] dt
\nonumber\\
& \leq C\sup_{t\in[0,T_\xi]} [\|\widebar\u(t)\|_{H^2(\O)}^2 + \|\widebar\theta(t)\|_{L^2(\O)}^2  +  \|\widebar\Gbf(t)\|^2_{H^2(\O)} + \|\widebar\varrho(t)\|^2_{L^2(\O)} ],
\end{align}
where we have used $\varrho_1,\varrho_2, \pfrak_1$ and $\pfrak_2$ belong to $L^{\infty}(0,T_\xi;H^2(\O))$. In view of Lemma~\ref{Lifting:estimates}, Lemma \ref{thm:stability} and \eqref{varrho:pfrack:dif}, there exists a constant $\gamma_0>0$ (independent of $\gamma$, $\g_1$ and $\g_2$) such that 
\begin{align}
	\gamma \|\widebar{\g}\|^2_{\Ucal} \leq \gamma_0 \|\widebar{\g}\|^2_{\Ucal}. 
\end{align}
Therefore, if $\gamma>\gamma_0$, we conclude that $\g_1=\g_2$. The optimal control is therefore unique. This completes the proof.
\end{proof}

\section{Conclusion and future perspectives}

We developed a boundary optimal-control framework for a nondiffusive scalar
transported by a two-dimensional second-grade fluid.  The analysis gives a
common, control-independent state horizon, existence of optimal controls,
directional differentiability in the topology of the terminal mix-norm, a
weak adjoint formulation, and a first-order variational
inequality.  In the frictionless case without the enstrophy reward, additional
adjoint regularity also yields uniqueness for a sufficiently large quadratic
control penalty.

The broader conclusion is that the negative-Sobolev approach to optimal
mixing survives the passage from Newtonian to second-grade dynamics, despite
the filtered acceleration, generalized vorticity, and more delicate boundary
and adjoint structures.  In particular, tangential wall actuation can still
be incorporated into a rigorous differentiable optimization framework even
though the scalar equation has no diffusion.  This opens a route to studying
how the second-grade parameter changes admissible stirring mechanisms and
optimal wall protocols.

Natural analytical questions include continuation of the active problem
beyond the present local horizon, uniqueness for $\beta>0$ or $\zeta>0$, and
the limit $\alpha\to0$ for the state, adjoint, and optimal control.  A complementary direction is the numerical approximation of the optimality
system. Since the scalar equation has no diffusion, it is desirable that the
discretization preserve scalar mass and $L^2$-energy and maintain discrete
state-adjoint duality, as in the structure-preserving framework of
\cite{HuLiZhangZuazua2026}. Extending this approach, together with the
boundary-control algorithms developed in \cite{ZhengHuWu2023}, to the
second-grade momentum equation and tangential Navier-slip control would permit
a systematic numerical investigation of the effects of $\alpha$, $\beta$, and
$\zeta$ on optimal mixing.

 \medskip 
 \noindent
 \textbf{Acknowledgments.}  The author gratefully acknowledges Professor Enrique Zuazua for his support and careful comments on an earlier version of this manuscript. This work is funded by the AFOSR 24IOE027 project.


 \medskip\noindent
 \textbf{Data availability.}
 No datasets were generated or analysed during the current study.
 
 \medskip\noindent
 \textbf{Author contributions.}
 The author carried out the mathematical analysis and wrote the manuscript.
 
 \medskip\noindent
 \textbf{Conflict of interest.}
 The author declares no conflict of interest.

\end{document}